\documentclass{amsart}

\usepackage{amsthm}
\usepackage{amssymb}

\usepackage{tikz-cd}

\usepackage{colonequals}

\usepackage{enumerate}
\usepackage{eucal}

\usepackage{hyperref}
\hypersetup{%
  bookmarksnumbered=true,%
  colorlinks=true,%
  linkcolor=blue,%
  citecolor=blue,%
  filecolor=blue,%
  menucolor=blue,%
  urlcolor=blue,%
  bookmarksopen=true,%
  bookmarksdepth=2,%
  pageanchor=true}

\usepackage{mathtools}
\usepackage{todonotes}

\makeatletter
\@namedef{subjclassname@2020}{\textup{2020} Mathematics Subject Classification}
\makeatother

\numberwithin{equation}{section}

\theoremstyle{plain}

\newtheorem{theorem}[equation]{Theorem}
\newtheorem{proposition}[equation]{Proposition}
\newtheorem{lemma}[equation]{Lemma} 
\newtheorem{corollary}[equation]{Corollary}

\theoremstyle{definition}

\newtheorem{example}[equation]{Example}

\theoremstyle{remark}

\newtheorem{remark}[equation]{Remark}
\newtheorem{question}[equation]{Question}

\newcommand{\Ad}{\operatorname{Ad}}
\newcommand{\ann}{\operatorname{ann}}
\newcommand{\bos}[1]{{\boldsymbol{#1}}}
\newcommand{\cato}{\operatorname{C}}
\newcommand{\coker}{\operatorname{Coker}}
\newcommand{\cmod}{\Psi}

\newcommand{\depth}{\operatorname{depth}}
\newcommand{\GL}{{\rm GL}}

\newcommand{\End}{\operatorname{End}}

\newcommand{\ext}{\operatorname{Ext}}
\newcommand{\fitt}{\operatorname{Fitt}}
\newcommand{\Frob}{\operatorname{Frob}}
\newcommand{\fm}{{\mathfrak m}}
\newcommand{\fn}{{\mathfrak n}}
\newcommand{\fp}{{\mathfrak p}}
\newcommand{\fq}{{\mathfrak q}}

\newcommand{\fpt}{\widetilde{\mathfrak p}}
\newcommand{\hh}{\operatorname{H}}
\newcommand{\height}{\operatorname{height}}
\newcommand{\Hom}{\operatorname{Hom}}
\newcommand{\sHom}{\underline{\Hom}}

\newcommand{\image}{\operatorname{Image}}
\newcommand{\Ker}{\operatorname{Ker}}
\newcommand{\length}{\operatorname{length}}
\newcommand{\ord}{\operatorname{ord}}
\newcommand{\mco}{\mathcal O}
	
\newcommand{\pos}[1]{[\![{#1}]\!]}
\newcommand{\Rt}{\widetilde{R}}
\newcommand{\rank}{\operatorname{rank}}
\newcommand{\Spec}{\operatorname{Spec}}

\newcommand{\tors}{\operatorname{tors}}

\newcommand{\tfree}[1]{{#1}^{\operatorname{tf}}}
\newcommand{\trace}{\operatorname{Tr}}

\newcommand{\rhobar}{\overline{\rho}}

\newcommand{\fullT}{\widetilde{\mathbb T}}

\newcommand{\mbb}[1]{\mathbb{#1}}

\newcommand{\TT}{\mathbb{T}}
\newcommand{\ZZ}{\mathbb{Z}}
\newcommand{\QQ}{\mbb{Q}}
\newcommand{\A}{\mbb{A}}
\newcommand{\Q}{\mbb{Q}}

\newcommand{\st}{\operatorname{st}}
\newcommand{\uni}{\operatorname{uni}}
\newcommand{\unr}{\operatorname{unr}}
\newcommand{\phuni}{\operatorname{{\varphi\mbox{-}uni}}}
\newcommand{\phunr}{\operatorname{{\varphi\mbox{-}unr}}}

\begin{document}

\title[The commutative algebra of congruence ideals]{The commutative algebra of congruence ideals\\ and applications to number theory}

\author[S.~B.~Iyengar]{Srikanth B.~Iyengar}
\address{Department of Mathematics,
University of Utah, Salt Lake City, UT 84112, U.S.A.}
\email{srikanth.b.iyengar@utah.edu}

\author[C.~B.~Khare]{Chandrashekhar  B.~Khare}
\address{Department of Mathematics,
University of California, Los Angeles, CA 90095, U.S.A.}
\email{shekhar@math.ucla.edu}

\author[J.~Manning]{Jeffrey Manning}
\address{Mathematics Department,
Imperial College, London, SW7 2RH, UK}
\email{jeffrey.manning@imperial.ac.uk}

\date{\today}


   
\begin{abstract} 
In his proof  of Fermat's Last Theorem,  Wiles deployed a commutative algebra technique, namely a numerical criterion for detecting isomorphisms of rings. In our recent work we pick up on Wiles' work and generalize the numerical criterion to ``higher codimension''. A critical ingredient is a notion of congruence module in higher codimension: this has turned out  to be a key definition  whose utility extends beyond the role it plays in the numerical criterion.  In this paper we trace the origin of some of the ideas that led to  our work, both in number theory and commutative algebra, and  new  directions that emerge from it. We introduce a related notion of a congruence ideal.

When applied to deformation theory of Galois representations and Hecke algebras, which is the setting of Wiles's work on Fermat's Last Theorem, our work leads to the notion of congruence ideals for local deformation rings. This sheds light on the classically studied congruence ideals  for global deformation rings and Hecke algebras.  We outline applications of the commutative algebra we have developed to: (i)   integral modularity lifting theorems in the context of weight one forms, and   (ii) factorization formulas for congruence ideals of global deformation rings at augmentations induced by newforms in which  local congruence ideals enter as the local terms.  The latter leads to  surprising relations between  these local congruence ideals and local Tamagawa ideals of Bloch-Kato associated to the rank 3 adjoint motive $\Ad_f$ of $f$.
\end{abstract}

\maketitle

\setcounter{tocdepth}{1}
\tableofcontents

\section{Introduction}
\label{se:intro}
 Our work takes its cue from an ingredient in Wiles's solution of Fermat's Last Theorem,  and is a contribution to the theme of the unreasonable effectiveness of commutative algebra in number theory. This theme goes  back to  at least Kummer's work (again on Fermat's Last Theorem!) that led to the theory of ideals and to the development of the basic infrastructure of algebraic number theory.  The protagonists of this expository article are congruence modules in codimension  $c \geq 0$ defined in \cite{Iyengar/Khare/Manning:2024a}.  Congruence modules  in codimension 0  have  been a subject of intense interest, mainly in the context of congruences attached to modular forms, and  we hope to illustrate that their higher codimension  cognates shed new light on them. 

 The work presented here   grew out of a serendipitous conversation between the first and second authors a few years ago. A starting point  for it is the notion of congruences, and congruence modules, associated to modular forms.  To set the stage  for our work, we recall  Ramanujan's  celebrated congruence; see \cite{Ramanujan:1916}.   Consider the modular forms of weight 12  and level 1 and, for $\mco=\mbb{Z}_{(691)}$, the $\mco$-lattice consisting of those forms whose Fourier expansion has  coefficients in $\mco$, denoted   $M_{12}(\rm SL_2(\mbb{Z}), \mco)$. This lattice is of rank 2 and its $\mco$-submodule spanned by the Ramanujan delta function, $\Delta(z)$, and the Eisenstein series,  $E_{12}$,  is of index $691$. The corresponding congruence module is the cokernel of the  inclusion of lattices $\mco\Delta\oplus \mco E_{12} \hookrightarrow  M_{12}(\rm SL_2(\mbb{Z}), \mco)$, namely,  $\mbb Z /691\mbb Z$; see the discussion below. Ramanujan's congruence  $\tau(n)\equiv \sigma_{11}(n) \mod{691}$---a term-by-term congruence between the Fourier coefficients of  $\Delta$ and $E_{12}$, the ur-congruence in the subject---is equivalent to the non-vanishing of this congruence module. 

We  introduce a general framework in which congruence modules can be defined.

\subsubsection*{Congruence modules and lattices.}
Let $\mco$ be a discrete valuation ring; for example, $\mbb Z_{(p)}$, the localization of $\mbb Z$ at the prime ideal $(p)$ for some prime number $p$, or its completion $\mbb Z_p$, the ring of $p$-adic integers. Let $E$ denote the field of fractions of $\mco$. Let $V$ be a  finite dimensional $E$-vector space and $L\subset V$ a full sub-lattice, namely, a finitely generated $\mco$-submodule of $V$ such that the induced map is an isomorphism $L\otimes_\mco E\cong V$; equivalently, $V/L$ is torsion as an $\mco$-module.

Let  $V_1 \oplus V_2 = V$ be a $E$-vector space decomposition and  $\pi_i\colon V\to V_i$, for $i=1,2$, the corresponding projections. Set  $L_i \colonequals L\cap V_i$ and $L^i\colonequals \pi_i(L)$; both are free $\mco$-modules of rank $\dim(V_i)$.  Any sublattice $L_1 \oplus L_2$ of $L$ of full rank with $L/L_i$ torsion-free arises from such a construction: namely, there is a bijection between tuples $(V, L,V_1 \oplus V_2)$  and $(V,L, L_1 \oplus L_2)$ with $L_1 \oplus L_2$ of finite index in $L$, and $L/L_i$ torsion-free. It is easy to verify that one has natural isomorphisms 
\[
\frac{L^1}{L_1} \cong \frac L{L_1\oplus L_2} \cong \frac{L^2}{L_2}
\]
of $\mco$-modules. This is the \emph{congruence module}  associated to the given data, and it measures the failure of the vector space decomposition $V=V_1\oplus V_2$ to induce an integral decomposition of the lattice $L$. The congruence  module is nonzero if and only if there exist elements $f_1\in L_1$ and $f_2\in L_2$ such that $f_1\equiv f_2\not\equiv 0 \mod \varpi L$, where $\varpi$ is a uniformizer for $\mco$.  In our example above this is the congruence $\Delta \equiv E_{12} \mod{691}$.  

Consider the dual lattice $\Hom_\mco(L,\mco)$  and  the canonical pairing 
\[
\langle \ , \ \rangle_L\colon L \otimes \Hom_{\mco}(L,\mco) \longrightarrow  \mco\,.
\]
The congruence module is  related to the  discriminant of the induced pairing 
\[
\langle \ , \ \rangle_L\colon L_1 \otimes \Hom_{\mco}(L/L_2,\mco) \to \mco\,.
\] 
If  $x_1,\ldots,x_d$ and $f_1,\ldots,f_d$ are an $\mco$-basis for $L_1$ and  $\Hom(L/L_2,\mco)$, respectively, the congruence module is isomorphic to the cokernel of the map $\mco^d\to \mco^d$ given by the matrix $(\langle f_i,x_j\rangle_L)_{ij}$. Thus its  Fitting ideal over $\mco$   is given by $ \left(\det(\langle f_i,x_j\rangle_L)_{ij}\right)$.

\subsubsection*{Congruence modules  and congruences between cusp forms}
Consider 
\[
V\colonequals S_2(\Gamma_1(N),E)\quad\text{and}\quad L=S_2(\Gamma_1(N),\mco)\,,
\]
the space of cusp forms of level $N$ and weight 2 with  Fourier coefficients in $E$, and its lattice of forms  with  Fourier coefficients in $\mco$, respectively. 

Commutative algebra enters the picture through  the action of a Hecke algebra $\TT$ (a finite, flat and reduced $\mco$-algebra)  that acts on these spaces. A (Hecke) newform $f$ for $\mbb T$ gives rise to an augmentation $\lambda_f\colon \TT \to \mco$. With $\fp\colonequals \Ker(\lambda_f)$ and $I$ the annihilator of $\fp$ in $\TT$, one gets a decomposition $V=V_1\oplus V_2$, where $V_1$ is the subspace of $V$ annihilated by $\fp$ and $V_2$ is the subspace annihilated by $I$; see Lemma~\ref{le:cmod-decomposition}.

There is a perfect pairing $\TT \times L \to \mco$ given by $(T,g)=a_1(g|T)$, the coefficient of $q=e^{2\pi iz}$ in the   Fourier expansion  of $g|T\in L$.  This pairing above   implies  that congruence  modules of $\mbb T$ and $L$  associated to $f$  are dual to each other, so one may as well work with the congruence module of $\mbb T$. This perspective has proved to be powerful because the Hecke algebra  $\mbb T$ acts on a plethora of other modules, like the Betti cohomology of the  modular curve $X_1(N) $ with $E$ or $\mco$-coefficients, which are easier to study  using  cohomological techniques. 

The connection of congruence modules to discriminants of pairings  discussed above is a key observation in Hida's result relating congruences  mod $p$ between $f$ and other Hecke eigenforms $g$  to the  $p$-divisibility of  (the algebraic part of)  an  $L$-value $L(1,\Ad_f)$;  see \cite{Hida:1981a}. Hida studies congruences of newforms in  $S_2(\Gamma_1(N),\mco)$ using  the action of the Hecke algebra on  $L=\hh^1(X_1(N),\mco)$ equipped with its Poincar\'e pairing. To show that  the  congruence module attached to  $L$  and $f$ captures  all congruences between $f$ and other forms  of weight 2 and level $N$,  one has to relate  congruence modules arising from $S_2(\Gamma_1(N),\mco)$ and $\hh^1(X_1(N),\mco)$; see  \cite{Hida:1981b} and \cite{Ribet:1983a}. This is an ongoing theme in the study of congruences; cf.~\cite{Bockle/Khare/Manning:2021b} for a recent result in this vein.  The fine integral study of congruences between modular forms is  at the heart of Wiles' proof~\cite{Wiles:1995} of the modularity of semistable  elliptic curves over $\mbb Q$; see Thorne's survey~\cite{Thorne:2024} for an excellent introduction to the role of congruence in the study of Galois representations.

\subsubsection*{Congruence modules in all codimensions}

The decomposition $V_1\oplus V_2$ of the space $V$ of cusp forms above could be into the space of new forms and old forms, or the space spanned by an eigenform  $f$ for $\mbb T$  and its orthogonal complement; congruence modules in these two situations have been studied in \cite{Ribet:1984} and \cite{Hida:1981a} respectively.  The second situation  can be abstracted as follows.

Consider an augmentation $\lambda\colon A\to \mco$ of $\mco$-algebras, where $\mco$ is a discrete valuation ring, and  $A$ is a (commutative) complete local $\mco$-algebra such that the map $\lambda_\fp$ is an isomorphism, where $\fp=\Ker \lambda$.  The \emph{congruence module} of a finitely generated  $A$-module $M$ (with respect to $\lambda$) is the $\mco$-module
\[
\cmod_\lambda(M)\colonequals \frac{M}{M[\fp]+M[I]} \quad \text{where $I\colonequals A[\fp]$.}
\]
Here, for any ideal $J\subset A$, we set
\[
M[J]\colonequals \{x\in M\mid J\cdot x=0\}\,,
\]
the $J$-torsion submodule in $M$.   The hypothesis that $\lambda_\fp$ is an isomorphism implies that the length of the $\mco$-module $\cmod_{\lambda}(M)$ is finite. As is explained in Remark~\ref{re:cmod-lattice}, the definition fits into the paradigm of congruence modules attached to lattice decompositions discussed above, and in particular, $\cmod_\lambda(M)=0$ if and only if the the $A$-submodule $M[\fp]\subseteq M$ has a direct complement.

Motivated by number theory, in \cite{Iyengar/Khare/Manning:2024a} we  generalized the above definition and defined congruence modules in higher codimension. With $A$ as before, we consider augmentations $\lambda\colon A\to \mco$ such that local ring $A_\fp$ is regular (the classical situation recalled above is when we assume  it is a field), for $\fp\colonequals \Ker \lambda$. The  Krull dimension of $A_\fp$, denoted $c$ in the sequel, is known as the codimension (or the height) of $\fp$; the classical situation concerns the case $c=0$. In this generality, the \emph{congruence module} of a finitely generated $A$-module $M$, again denoted $\cmod_\lambda(M)$, is the cokernel of the map
\[
\ext^c_A(\mco,M)\longrightarrow  \tfree{\ext^c_A(\mco,M/\fp M)}
\]
induced by the quotient map $M\to M/\fp M$. Here$\tfree{(-)}$ denotes passage to the torsion-free quotient, as an $\mco$-module; see Section~\ref{se:congruence-ideals} for details and antecedents of this map in local algebra. We develop some of the purely commutative algebraic properties of this invariant in \cite{Iyengar/Khare/Manning:2024a}, and obtain a generalization to higher codimensions of the numerical criterion of Wiles, Lenstra, and Diamond; see \cite{Wiles:1995, Lenstra:1995, Diamond:1997}.

In this manuscript, to further scientific diversity (borrowing a phrase from \cite{Faltings:1998})  and also because it connects better with the pairing discussed earlier, we introduce a notion of a congruence ideal associated to $M$, based on an adjoint of the map used to define congruence modules; namely, as the image of the natural map
\[
\ext^c_A(\mco,M)\otimes_{\mco} \Hom_A(M,\mco)\longrightarrow \tfree{\ext^c_A(\mco,\mco)}\,.
\]
We denote this $\eta_\lambda(M)$; see Section~\ref{se:congruence-ideals} for details of this construction. The congruence ideal records only part of the information encoded in the congruence module---see Lemma~\ref{le:eta-psi}---but suffices for many of the intended applications to number theory. 

Congruence ideals for higher codimension turn out to have uses which we had not anticipated;  in our work we have been led by them to go where they will take us. Unlike the case of $c=0$, when there is a direct relation between congruence ideals and congruences,  in higher codimension the meaning of the congruence ideal is more elusive and  there is no direct link  to congruences; its implicit significance and meaning  is teased out by the contexts in which it is viewed. Our work shows  that it is the correct generalization as it slots in perfectly into our generalization of the numerical criterion.

\subsubsection*{Number theoretic applications.}
 Congruences  in codimension zero have played a crucial role in the relationship between Galois representations and automorphic  forms. They are a key ingredient in the strategy used by Scholze \cite{Scholze:2015a} to attach Galois representations to torsion classes in the cohomology of arithmetic manifolds. In the reverse direction, congruence ideals in the codimension 0  case played a key role in the strategy in \cite{Wiles:1995} and \cite{Taylor/Wiles:1995}  for  attaching modular forms to Galois representations. The applications we outline below hinge  on applying our work on congruence modules in higher codimension  to ``patched'' deformation rings.  Patching was introduced by Taylor and Wiles in \cite{Taylor/Wiles:1995} to complete   the  proof of the modularity of semistable elliptic curves over $\QQ$ initiated in \cite{Wiles:1995}. The proof  in \cite{Wiles:1995} and \cite{Taylor/Wiles:1995} relied  on using patching and the numerical criterion (in codimension $c=0$) independently, with the  latter  applied to classical (unpatched) objects. 
 
 A key insight of our work is that applying a numerical criterion (or analyzing congruence ideals), necessarily in codimension $c>0$, \emph{after} patching (of patched deformation rings)  increases its reach and leads to  novel factorization formulas for congruence ideals of global deformation rings and Hecke algebras in codimension 0. We give two applications of our commutative algebra results to number theory; both of them   rely on this insight.
 
\subsubsection*{Jacquet-Langlands}
 We prove  a  Jacquet-Langlands  correspondence for Hecke algebras $\mbb T$  that  act faithfully on the coherent cohomology of Shimura curves $X$, arising from indefinite quaternion algebras  $D$ over $\QQ$,  with coefficients in $\omega_X^{\otimes k}$ for weight $k=1$; see Theorem \ref{th:weightone}.  The novelty here is that, as  weight  $k=1$ Hecke algebras tend to have a lot of torsion  (in contrast to the Hecke algebra acting  on  the coherent cohomology of the sheaves $\omega^{\otimes k}$ for  weights $k\geq 2$),  our result  cannot be deduced from the classical Jacquet-Langlands correspondence for these Hecke algebras after tensoring with $\mbb Q$.  To accomplish this, we    use modularity lifting (reciprocity) to establish  functorialty (lifting torsion in  cohomology of Shimura curves to torsion in  cohomology of modular curves),  as the latter   is obvious at the level of Galois parameters (deformation rings).   We sketch a proof of the main theorem of  our forthcoming work  with  Diamond \cite{Diamond/Iyengar/Khare/Manning:2026a}, that generalizes the work of \cite{Iyengar/Khare/Manning:2024a} and proves the required modularity lifting results in the weight $1$ case. These results  apply  to show  that mod $p^n$ Galois representations $\rho\colon G_\Q \to  \GL_2(\ZZ/p^n\ZZ)$  that are unramified at $p$ and  need   not lift to characteristic 0 arise from mod $p^n$ weight one forms.

 \subsubsection*{Factorization of congruence ideals} Applied to local deformation rings, and combined with patching,  our definition of congruence ideals (in codimension $c\ge 1$) produces a  local-global  factorization of congruence ideals (for $c=0$)  of classical Hecke algebras.    We use this to refine the results about the Wiles defect of Hecke algebras proved in \cite{Bockle/Khare/Manning:2024}. Rather than simply computing the Wiles defect which measures the difference between lengths of cotangent spaces and congruence modules at augmentations of Hecke algebras induced by newforms as in \emph{op. cit.}, here  we give formulas for each of these lengths. A key ingredient needed  for this approach is computation of  congruence ideals of augmentations  arising from newforms of  local deformation rings  considered in \cite{Bockle/Khare/Manning:2024}; see Theorem \ref{th:BK}. Patching  is integral to this application as well: it allows one to prove  integral $R=\TT$ theorems if the local deformation rings are Cohen-Macaulay. This allows one to apply Theorem \ref{th:deformation} to deduce our factorization of congruence ideals of $R$ and $\TT$.

  This also leads to computations of congruence ideals (at augmentations induced by newforms) of integral Betti cohomology groups of modular and Shimura curves. This ``numerical method'' was used in \cite{Bockle/Khare/Manning:2024} to deduce   defects of modules  from defects   of rings; here we apply it  to determine congruence ideals of modules from  knowing congruence ideals  of rings.  One of the facts that this leans on is  that  for a  module $M$ over $A$, and an augmentation $\lambda\colon A \to \mco$  supported by $M$, one has $\eta_\lambda(A) \subseteq \eta_\lambda(M)$.  This sheds new light on results of  Ribet and Takahashi  \cite{Ribet/Takahashi:1997} on degrees of (optimal)  parametrizations of   elliptic curves over $\QQ$ by Shimura curves.

\subsubsection*{Relation to our earlier work and  ongoing work.}

The work surveyed here has been years in the making. The initial impetus was  a hunch (see \cite{Fakhruddin/Khare/Ramakrishna:2021})  that, if developed further, the numerical criterion might have the potential  to give information about deformation rings and Hecke algebras that was not accessible by known methods.  Of the two techniques Wiles introduced in his work on Fermat's Last Theorem, patching has been developed into a big thoroughfare  that runs through the subject of linking Galois representations and  automorphic forms. The technique relying on the  numerical criterion had not been developed much.   One  open problem  we had in mind was to prove  smoothness of the ring parametrizing  ordinary, fixed determinant deformations of   Galois representations  $\rho_\pi\colon G_K \to \GL_2(\overline{\Q}_p)$ that arise  from cohomological, cuspidal,  automorphic representations $\pi$ of $\GL_2(\A_K)$,  for imaginary quadratic fields $K$.  To follow the method of \cite[Chapter 3]{Wiles:1995} to approach  this (equivalently, to prove the the vanishing of a certain ``dual Selmer group'') it was natural to seek a definition for  congruence modules at  points of higher codimension. This led to the work of \cite{Iyengar/Khare/Manning:2024a} (see also \cite{Brochard/Iyengar/Khare:2023b}, \cite{Iyengar/Khare/Manning:2024b}) where we introduce congruence modules for such points \emph{assuming} they are smooth.  Thus while we  still cannot solve (alas!)  what was the target problem,  it came as a  pleasant surprise to us that if one used  this definition at  the plentifully available smooth points of ``patched'' deformation rings one was able to deduce  new integral modularity lifting results  (see Theorem \ref{th:weightone}) by following the original strategy of Wiles to deduce non-minimal modularity lifting results from the  minimal case.

In a parallel development,  in \cite{Bockle/Khare/Manning:2021b, Bockle/Khare/Manning:2024} the failure of Hecke rings to be complete intersections was studied via analyzing the Wiles defect. When \cite{Bockle/Khare/Manning:2021b}  was written, its authors computed the Wiles defect of global deformation rings and Hecke algebras by whatever means lay at hand,  and  only deduced  \emph{a posteriori}  that the  answer was a sum of local terms. This led  in \cite{Bockle/Khare/Manning:2024} to defining the  Wiles defect of local deformation rings,  which facilitated an \emph{a priori}  proof  that the defect of a global deformation ring is a sum of defects of the corresponding local deformation rings. 

In this paper  we take advantage of the commutative algebra developed in \cite{Iyengar/Khare/Manning:2024a, Iyengar/Khare/Manning/Urban:2024} to give a new approach to, and refinements of,  the  results of \cite{Bockle/Khare/Manning:2021b, Bockle/Khare/Manning:2024}. These results are in the tame case:  the local deformation rings at $p$ are nice (even smooth), while  away from $p$  they may  be  complicated (and not complete intersections).

Congruence ideals  of local deformation rings  are local analogs of   congruence ideals of  global  deformation rings and Hecke algebras.  The congruence ideals  for augmentations induced by $f \in S_k(\Gamma_1(N))$ of   local deformation rings at $p$ are vital   to our forthcoming work (joint with Diamond) on the  $p$-part of the  Bloch and Kato conjecture of \cite{Bloch/Kato:1990}  for the  value at $s=1$ of degree 3 adjoint $L$-function $L(\Ad_f,s)$ for a newform $f \in S_k(\Gamma_1(N))$ with $k \geq 2$.   The authors of   \cite{Diamond/Flach/Guo:2004} proved  the conjecture for primes $p$ with  $(p,Nk!)=1$, under the  Taylor-Wiles hypothesis that the  mod $p$ representations $\rhobar_f|_{\Q(\mu_p)}$ are irreducible. Our forthcoming work (joint with F.~Diamond)  in \cite{Diamond/Iyengar/Khare/Manning:2026b},  which proves cases of the conjecture not covered in \cite{Diamond/Flach/Guo:2004}; in particuar, we allow $k \geq p$ or $N$ to be divisible by powers of $p$. Our work
hinges on a   surprising connection between  local congruence ideals at $q$ and the local  Tamagawa ideals at $q$ of Bloch and Kato for   the adjoint motive $\Ad_f$ arising from $f$; this may be viewed as a local analog of the main theorem of \cite{Iyengar/Khare/Manning/Urban:2024}. Corollary \ref{cor:tamagawa} is a hint of that, in the easier case of $q \neq p$.

\subsubsection*{Leitfaden} The  main rationale of this paper is to provide a bird's eye view of our work in \cite{Iyengar/Khare/Manning:2024a}, and apply its methods to reprove and refine the results of \cite{Bockle/Khare/Manning:2021b} and \cite{Bockle/Khare/Manning:2024}.  This paper is kindred  to the article~\cite{Iyengar:2026} in the Proceedings of the International Congress of Mathematicians 2026 written by the first author, which surveys   the commutative algebra developed in \cite{Iyengar/Khare/Manning:2024a} (and also in \cite{Brochard/Iyengar/Khare:2023a}, \cite{Brochard/Iyengar/Khare:2023b}, \cite{Brochard/Iyengar/Khare:2025}).  Section \ref{se:congruence-ideals} below  views  the commutative algebra results of \emph{op. cit.} from a different slant and  includes  variations, and embellishments,  of  themes that were  initiated in our earlier paper: in particular we define the notion of a congruence ideal and state most of our results using congruence ideals rather than congruence modules.  In Section \ref{se:tool-kit}, \ref{se:determinantal-rings} and \ref{se:BKM-revisited}  we refine results about Wiles defects of  local deformation rings in \cite{Bockle/Khare/Manning:2021b} and \cite{Bockle/Khare/Manning:2024} by computing their congruence ideals. In Section \ref{se:NT}  we use these computations of local congruence ideals  to give formulas (see Theorem \ref{th:BK}) for congruence ideals of Hecke algebras acting on cohomology of Shimura and modular curves studied in \emph{op. cit.}  As another application of our methods,  we  sketch the application of Theorem \ref{th:defect0} to proving integral modularity lifting results in the weight one setting.

The mathematics presented here  illustrates an often observed  feature of number theory  that it can have the flavor of applied mathematics: one applies to its problems techniques from other branches of mathematics, in our case commutative algebra, that are sometimes expressly  developed in response to its demands.

\section{Congruence ideals}
\label{se:congruence-ideals}
In this section we introduce the class of augmented $\mco$-algebras and a notion of congruence ideal attached to a module over such an algebra that are the focus of this manuscript.  For the most part, the proofs of the statements are only sketched. The emphasis is on the purely commutative algebra aspects of the theory, but the development is inspired by, and tailored for applications to, number theory. We comment on these motivations, and also on connections to literature in commutative algebra, in the subsections titled ``Notes". See also the first part of Section~\ref{se:NT}.

Throughout this text $\mco$ denotes a complete discrete valuation ring and $E$ its field of fractions. We fix also a  uniformizer $\varpi$ for $\mco$.  For any $\mco$-module $U$,  set
\[
\tors U \colonequals \Ker(U\to U\otimes_\mco E)\quad\text{and} \quad \tfree U\colonequals \coker (\tors U\subseteq U)\,;
\]
these are the torsion submodule, and the torsion-free quotient, respectively,  of $U$. By construction, there is an exact sequence $0\to \tors U\to U \to \tfree U\to 0$; when $U$ is finitely generated, this is split, though not canonically. We also have to consider the $\mco$-modules:
\[
U^*\colonequals \Hom_{\mco}(U,\mco)\quad\text{and}\quad U^\vee \colonequals \Hom_{\mco}(U,E/\mco)\,.
\]
The one on the left is the usual $\mco$-linear dual whereas the one on the right is the Matlis dual, for $E/\mco$ is the injective hull of the residue field $k$ of $\mco$. The Matlis dual of the exact sequence defining $\tfree{U}$ yields the exact sequence
\[
0\longrightarrow (U^{\vee})_{\rm div} \longrightarrow U^{\vee} \longrightarrow \mathrm{cotors}(U^{\vee})\longrightarrow 0
\]
relating the divisible submodule of $U^\vee$ and its cotorsion quotient submodule.

\subsection*{Augmented $\mco$-algebras.} 
Let $A$ be a noetherian local $\mco$-algebra with residue field $k_A=\mco/(\varpi)$, the residue field of $\mco$. We assume that $A$ is complete with respect to its maximal ideal, $\fm_A$. Consider an augmentations $\lambda\colon A\to \mco$ of $\mco$-algebras and set
\[
\fp\colonequals \Ker \lambda\quad\text{and}\quad c\colonequals \dim A_\fp\,.
\]
The number $c$ is the height, also called the codimension, of the ideal $\fp$.  The $\mco$-module $\fp/\fp^2$ is the \emph{conormal module} of $\lambda$; in number theory literature, it is often referred to as the cotangent module, and we follow suit. Its rank as an $\mco$-module is the embedding dimension of the local ring $A_\fp$, since one has isomorphisms of $E$-vector spaces
\[
(\fp/\fp^2)\otimes_{\mco} E \cong \fm/\fm^2\,,
\]
where $\fm$ is the maximal ideal $\fp A_{\fp}$ of $A_{\fp}$. Adjunction yields isomorphisms
\begin{equation}
    \label{eq:rank-cotangent}
    \Hom_{\mco}(\fp/\fp^2,E) \cong \Hom_E(\fp/\fp^2\otimes_{\mco}E, E)\cong \Hom_E(\fm/\fm^2,E)\,.
\end{equation}
 of $E$-vector spaces. This observation is used later on.

In what follows, we consider $\mco$ as an $A$-module via $\lambda$. 
For each $A$-module $M$ and integer $i$ the modules $\ext^i_A(\mco,M)$ have a natural $\mco$-action. In particular, the $\mco$-modules $\{\ext^i_A(\mco,\mco)\}_{i\in\mbb{N}}$ play a key role in what follows. Applying $\Hom_A(-,\mco)$ to the exact sequence $0\to \fp\to A\to \mco \to 0$ of $A$-modules yields an isomorphism
\begin{equation}
    \label{eq:normal}
    \ext^1_A(\mco,\mco) \cong \Hom_\mco(\fp/\fp^2,\mco)\,.
\end{equation}
 The module on the right is the normal module of $\lambda$. In contexts that interest us, it determines the torsion-free part of the higher Ext-modules; see Theorem~\ref{th:serre}.
 
 For each $A$-module $M$ there is a natural map
\begin{equation}
\label{eq:pre-eta}
\ext^c_A(\mco,M)\otimes_{\mco} \Hom_A(M,\mco)\longrightarrow \ext^c_A(\mco,\mco)\,.    
\end{equation}
This map has a particularly simple description when $c=0$, for then
\[
\ext^0_A(\mco,M)=\Hom_A(\mco,M) = M[\fp]\,,
\]
namely, the $\fp$-torsion submodule of $M$, and \eqref{eq:pre-eta} is the map 
\[
\Hom_A(\mco,M)\otimes_\mco \Hom_A(M,\mco)\longrightarrow \mco
\]
given by composition. By adjunction $\Hom_A(M,\mco)\cong \Hom_{\mco}(M/\fp M,\mco)$ and the map  coincides with the composite map
\[
M[\fp]\otimes_\mco \Hom_{\mco}(M/\fp M,\mco) \longrightarrow 
            (M/\fp M)\otimes_\mco \Hom_\mco(M/\fp M,\mco)\longrightarrow \mco
\]
where the one of the left is induced by the  composition $M[\fp]\subseteq M\twoheadrightarrow M/\fp M$; the one on the right is evaluation.

For $c\ge 1$, in the Yoneda interpretation of Ext groups an element $\zeta\in \ext^c_A(\mco,M)$ is represented by an exact sequence of $A$-modules
\[
0\longrightarrow M\longrightarrow X_{c-1}\longrightarrow \cdots \longrightarrow X_0\longrightarrow \mco \longrightarrow 0 
\]
Given $f\in \Hom_A(M,\mco)$, the image of $\zeta\otimes f$ under the map \eqref{eq:pre-eta} is the push-out of $\zeta$ along $f$. 
See MacLane~\cite[Chapter 3]{MacLane:1963a} for details. Here is another perspective: $\zeta$ can be identified with a morphism $\mco\to \Sigma^c M$ in the derived category of $A$, and the map above is the obvious composition 
\[
\mco \xrightarrow{\ \zeta\ } \Sigma^c M\xrightarrow{\ \Sigma^c f\ } \Sigma^c\mco\,.
\]
Composing \eqref{eq:pre-eta} with the natural quotient map $\ext^c_A(\mco,\mco)\to \tfree{\ext^c_A(\mco,\mco)}$ gives the $\mco$-linear map
\begin{equation}
\label{eq:eta-defn}
\ext^c_A(\mco,M)\otimes_{\mco} \Hom_A(M,\mco)\longrightarrow \tfree{\ext^c_A(\mco,\mco)}\,.
\end{equation}
We write $\eta_\lambda(M)$ for the image of this map, viewed as a submodule of  $\tfree{\ext^c_A(\mco,\mco)}$, and call it the \emph{congruence ideal} of $M$; the term ``ideal" is a misnomer, though its use is partly justified by the fact that
$\tfree{\ext^c_A(\mco,\mco)}$ is a free $\mco$-module of rank one---see Lemma~\ref{le:top-class}---and for any choice of $\mco$-linear isomorphism $\tfree{\ext^c_A(\mco,\mco)}\cong \mco$ the image of $\eta_\lambda(M)$ in $\mco$ is an ideal that is independent of the choice of the isomorphism.

It is straightforward to check that $\eta_\lambda(A)$ is the image of the composition
\begin{equation}
\label{eq:eta-ring}
\ext^c_A(\mco,A)\longrightarrow \ext^c_A(\mco,\mco)\longrightarrow \tfree{\ext^c_A(\mco,\mco)}
\end{equation}
where the one on the left is induced by the natural surjection $A\to \mco$. The result below is the starting point of the algebraic theory of congruence modules and congruence ideals.

\begin{theorem}
\label{th:regular-eta}
Let $A$ be a noetherian local $\mco$-algebra and $\lambda\colon A\to \mco$ a morphism of $\mco$-algebras. The following conditions are equivalent.
\begin{enumerate}[\quad\rm(1)]
    \item $A$ is regular at $\lambda$;
    \item $\eta_\lambda(A)\ne 0$;
    \item $\eta_\lambda(M)\ne 0$ all for finitely generated $A$-modules $M$ with $M_\fp\ne 0$.
\end{enumerate}
Moreover, when these conditions hold, one has $\eta_\lambda(A)=\tfree{\ext^c_\lambda(\mco,\mco)}$ if and only if $A$ is regular.
\end{theorem}

\begin{proof}[Sketch of proof]
Since $\eta_\lambda(A)$ is the image of the map $\tfree{\ext^c_A(\mco,A)}\to\tfree{\ext^c_A(\mco,\mco)}$ of torsion-free $\mco$-modules, it is nonzero if and only if the induced map obtained by localizing at $\fp$, is nonzero; that is to say, the map
\[
\ext^c_{A_\fp}(E,A_\fp)\longrightarrow \ext^c_{A_\fp}(E,E)
\]
induced by the surjection $A_\fp\to E$, is nonzero.  This condition is equivalent to regularity of the local ring $A_\fp$, by a result of Lescot~\cite[Theorem~1.4]{Lescot:1983}; see also \cite[Theorem~2.4]{Avramov/Iyengar:2013}. This justifies (1)$\Leftrightarrow$(2). The claim that (1)$\Rightarrow$(3) can be verified by a direct computation; see the proof of Lemma~\ref{le:ext-tfree} below.

The last part of the theorem is extracted from \cite[Theorem~2.6]{Iyengar/Khare/Manning/Urban:2024}; its proof again uses Lescot's result.
\end{proof}

The result below is easy to verify, given the functoriality of the map \eqref{eq:eta-defn}.

\begin{lemma}
\label{le:add-eta}
One has $\eta_\lambda(M\oplus N)=\eta_\lambda(M)+\eta_\lambda(N)$ for finitely generated $A$-modules  $M,N$. \qed
    \end{lemma}

\subsubsection*{Notes}
The invariant $\eta_\lambda(A)$ appears in many places in the literature in number theory, albeit only in the case when $c=0$; see ~\cite{Mazur:1977,Hida:1981b,Ribet:1983a, Ribet:1984,Wiles:1995}. In number theoretic context, where we consider augmentations of local or global deformation rings arising from automorphic forms, the condition on $\lambda$, that $A_\fp$ is regular, is expected and often provable. For example, in the local case it corresponds to genericity conditions on local components of automorphic forms. 

Analogues of the map \eqref{eq:eta-ring} have appeared also in commutative algebra. In that context, the relevant map is $\ext^*_R(k,R)\to \ext^*_R(k,k)$ that is induced by the natural map $R\to k$ from a noetherian local ring $R$ to its residue field $k$. See the work of Lescot~\cite{Lescot:1983}, invoked in the proof of Theorem~\ref{th:regular-eta}, and also that of Martsinkovsky~\cite{Martsinkovsky:1996}, and Avramov and Veliche~\cite{Avramov/Veliche:2007}. Keeping in the spirit of \cite{Avramov/Halperin:1983}, there are closely related to developments in rational homotopy theory, dealing with the evaluation map; see, for instance, Felix and Lupton~\cite{Felix/Lupton:2007}.

\subsection*{The category $\cato_\mco(c)$.}
For applications to number theory, the focus is on pairs $(A,\lambda)$ that satisfy the equivalent conditions in Theorem~\ref{th:regular-eta}; that is to say, where the local ring $A_{\fp}$ is regular. We write $\cato_\mco$ for the category of such pairs; a morphism $ \alpha\colon (A,\lambda_A) \to (B,\lambda_B)$ in this category is a map $\alpha\colon A\to B$ of $\mco$-algebras such that 
$\lambda_B\circ\alpha = \lambda_A$. For any integer $c\ge 0$, the subcategory of pairs $(A,\lambda)$ in $\cato_\mco$ with $\dim A_\fp=c$ is denoted $\cato_\mco(c)$.

\begin{example}
    Here are some typical examples, from \cite{Darmon/Diamond/Taylor:1997}, to keep in mind: 
\begin{align*}
&A \colonequals \{(a,b)\in \mco\times \mco \mid a\equiv b \mod \varpi^n\} \quad\text{where $n$ is a fixed positive integer.}  \\
&B \colonequals \{(a,b,c)\in \mco\times \mco\times \mco \mid a\equiv b\equiv c\mod \varpi\} 
\end{align*}
where $\varpi$ is a uniformizer for $\mco$. One has isomorphisms of $\mco$-algebras
\[
A\cong \frac{\mco\pos{x}}{(x(x-\varpi^n))} \quad\text{and}\quad B\cong \frac{\mco\pos{x,y}}{(x(x-\varpi),y(y-\varpi),xy)}\,.
\]
Clearly, there are multiple $\mco$-valued points in $\Spec A$ and  $\Spec B$, and sometimes this freedom in choosing the point is important; see \cite{Iyengar/Khare/Manning:2024b}. For instance, for $A$ one has augmentations $\lambda_0$ and $\lambda_n$ that assign $x$ to $0$ and $\varpi^n$, respectively. A direct calculation yields that $\eta_{\lambda_0}(A)=(\varpi^n)=\eta_{\lambda_n}(A)$.

See Sections~\ref{se:determinantal-rings} and \ref{se:BKM-revisited} for many other, more complicated, examples. These  occur in nature as local deformation rings, and may be regarded as stock in trade  for us, just as, say, certain topological spaces might be part of a topologist's tool kit.
\end{example}

Here is a simple consequence of the regularity of the ring $A_\fp$. In the statement, $\beta_{A_\fp}(M_\fp)$ denotes the minimal number of generators of the $A_\fp$-module $M_\fp$.  For a more refined description of $\tfree{\ext^c_A(\mco,\mco)}$, see Theorem~\ref{th:serre}.

\begin{lemma}
\label{le:top-class}
For $A$ in $\cato_{\mco}(c)$, and any finitely generated $A$-module $M$, one has
\[
\rank_\mco \ext^c_A(\mco,M) = \rank_\mco (M/\fp M) = \beta_{A_\fp}(M_\fp)\,.
\]
In particular, as $\mco$-modules one has $\tfree{\ext^c_A(\mco,\mco)}\cong \mco$.
\end{lemma}

\begin{proof}
Set $R=A_\fp$; this is a local ring, with maximal ideal $\fp A_\fp$ and residue field $E$, the field of fractions of $\mco$. Since one has an isomorphism 
\[
\ext^c_A(\mco,M)_\fp \cong \ext^c_R(E,M_\fp)\,,
\]
it suffices to verify that the rank of the $E$-vector space on the right equals $\beta_R(M_\fp)$. The ring $R$ is regular, by the hypothesis on $\lambda$, so the Koszul complex on a minimal generating set for its maximal ideal, $\fp A_\fp$ is a minimal free resolution of  $E$. Computing $\ext^c_R(E,M_\fp)$ using the Koszul complex yields the desired result.
\end{proof}

\subsection*{Depth.} Fix $(A,\lambda)$ in $\cato_\mco(c)$ and a finitely generated $A$-module $M$ such that 
\[
\mathrm{depth}_AM\ge c+1\quad\text{and}\quad M_\fp \ne 0\,.
\]
The condition on the depth of $M$ holds when, for instance, it is \emph{maximal Cohen-Macaulay}, that is to say,  if $\depth_AM=\dim A$, as is often the case in the number theoretic context.  The hypothesis that $\depth_AM\ge c+1$ implies  $\depth_{A_\fp}M_\fp\ge c$, and hence  $M_\fp$ is a free $A_\fp$-module, by the Auslander-Buchsbaum equality. Set
\[
\mu\colonequals \rank_{A_\fp}M_\fp\,. 
\]
Since the codomain of the map \eqref{eq:eta-defn} is a torsion-free  $\mco$-module, it factors through $\tfree{\ext^c_A(\mco,M)}\otimes_\mco \Hom_A(M,\mco)$. In fact, when $M$ has sufficient depth the Ext-module in question is already torsion-free. This follows from the next observation.

\begin{lemma}
    \label{le:ext-tfree}
If $\depth_AM\ge c+1$, the $\mco$-module $\ext^c_A(\mco,M)$ is torsion-free.    
\end{lemma}

\begin{proof}
We have an  exact sequence $0\to \mco \xrightarrow{\varpi}\mco \to k_A\to 0$ of $A$-modules. Applying $\Hom_A(-,M)$ to it yields an exact sequence
\[
\cdots \longrightarrow \ext^{c}_A(k_A,M)\longrightarrow \ext^c_A(\mco,M)\xrightarrow{\ \varpi\ } \ext^c_A(\mco,M)\longrightarrow \cdots
\]
The hypothesis $\depth_AM\ge c+1$ is equivalent to $\ext_A^i(k_A,M)=0$ for $i\le c$; see \cite[Theorem~1.2.8]{Bruns/Herzog:1998a}. Thus the sequence above yields the desired result.
\end{proof}

\subsection*{Congruence module.}
Consider the map 
\[
\ext^c_A(\mco,M)\longrightarrow \tfree{\ext^c_A(\mco,\mco)}\otimes_{\mco} \tfree{(M/\fp_AM)}\cong \tfree{\ext^c_A(\mco,M/\fp_AM)}.
\]
that is adjoint to \eqref{eq:pre-eta}. It coincides with the map induced by the canonical surjection $M\to M/\fp_AM$. The cokernel of this map is the \emph{congruence module} of $M$ at $\lambda$, and denoted $\cmod_\lambda(M)$. It is the focus of our work in \cite{Iyengar/Khare/Manning:2024a}. The result below, which is an analogue of Theorem~\ref{th:regular-eta}, characterizes the $\mco$-algebras in $\cato_\mco(c)$ among all $\mco$-algebras; it is extracted from \cite[Theorem~2.2]{Iyengar/Khare/Manning/Urban:2024}.

\begin{theorem}
\label{th:Psi-torsion}
Let $A$ be a noetherian local $\mco$-algebra and $\lambda\colon A\to \mco$ a morphism of $\mco$-algebras. The following conditions are equivalent.
\begin{enumerate}[\quad\rm(1)]
    \item The local ring $A$ is regular at $\lambda$; that is to say, $(A,\lambda)$ is in $\cato_\mco$.
    \item The $\mco$-module $\cmod_\lambda(A)$ is torsion.
    \item The $\mco$-module $\cmod_\lambda(M)$ is torsion for  finitely generated $A$-modules $M$. 
\end{enumerate}
When these conditions hold, one has $\cmod_\lambda(A)=0$ if and only if $A$ is regular.\qed
\end{theorem}

The proof of this result is similar to that of Theorem~\ref{th:regular-eta}; in fact, one can be readily deduced from the other, given Lemma~\ref{le:eta-psi}.

\begin{remark}
\label{re:cmod-lattice} 
Suppose $A$ in $\cato_\mco(c)$ and $M$ is a finitely generated $A$-module with $\depth_AM\ge 1$. From \cite[Propostion~2.10]{Iyengar/Khare/Manning:2024a} we get that
\[
\cmod_\lambda(M) = \frac{M}{M[\fp]+M[I]} \quad \text{for $I=A[\fp]$}\,.
\]
This ties in with the discussion of congruence modules attached to lattice decomposition discussed in the Introduction. Indeed, suppose that  $M$ is finitely generated also as an $\mco$-module; it is torsion-free, since $\depth_AM\ge 1$, so it can be viewed as a full sub-lattice of $V= E\otimes_\mco M$, where $E$ is the field of fractions of $\mco$. Set  $V_1=V[\fp]$; this has a unique $A$-module complement in $V$ given by $V_2=V[I]$, so we get a vector space decomposition $V=V_1\oplus V_2$; see Lemma \ref{le:cmod-decomposition} below. The congruence module attached to this data, as in the Introduction, is precisely $\cmod_\lambda(M)$.
\end{remark}

Motivated by Theorem~\ref{th:Psi-torsion} one can ask what the condition  $\cmod_\lambda(M)=0$ means for an $A$-module $M$ with $A$ in $\cato_\mco(c)$. For the moment, we can answer this question only when $c=0$ and $\depth_AM\ge 1$. 

\begin{lemma}
\label{le:cmod-decomposition}
Fix $A$ in $\cato_\mco(0)$ and a finitely generated $A$-module $M$ such that $\depth_AM\ge 1$. Then $\cmod_\lambda(M)=0$ holds precisely when the $A$-submodule $M[\fp]\subseteq M$ has a direct complement. In that case, the complement is unique, and equal to $M[I]$.
\end{lemma}

\begin{proof}
With $I=A[\fp]$ the ideal $\fp+I\subseteq A$ is $\fm_A$-primary, since $A$ is in $\cato_\mco(0)$. Thus, since $M$ has positive depth $M[\fp]\cap M[I]=\{0\}$. Moreover, suppose $W$ is a $A$-submodule of $M$ with $M[\fp]\oplus W=\{0\}$. Then $\fp\cdot(I W)=0$, by the definition of $I$, so $IW\subseteq M[\fp]\cap W =0$. Thus $W\subseteq M[I]$, and we get inclusions
\[
M[\fp]\oplus W \subseteq M[\fp]\oplus M[I]\subseteq M\,.
\]
It follows that if $\cmod_\lambda(M)=0$, then $M[\fp]$ has a directly complement, $M[I]$, and also that if $M[\fp]$ has any direct complement $W$, then $W=M[I]$.
\end{proof}

\subsection*{A defect formula.}
For any $A$-module $M$ one has the K\"unneth map
\begin{equation}
\label{eq:kunneth}
\ext^c_A(\mco,A)\otimes_\mco (M/\fp M) \cong \ext^c_A(\mco,A)\otimes_AM\longrightarrow \ext^c_A(\mco,M)\,.
\end{equation}
In terms of Yoneda extensions, given an element $\zeta = 0\to A\to \cdots \to \mco \to 0$ in $\ext^c_A(\mco,A)$ and an element $m\in M$, the map sends $\zeta\otimes m$ to the push-out of the extension along the $A$-linear map $A\to M$ given by $1\mapsto m$.

Consider the induced map on torsion-free quotients
\begin{equation}
    \label{eq:kappaM}
\kappa_\lambda(M) \colon \tfree{\ext^c_A(\mco,A)} \otimes_\mco \tfree{(M/\fp M)}
            \longrightarrow \tfree {\ext^c_A(\mco,M)}\,.
\end{equation}
The domain and codomain of this map are free $\mco$-modules of rank equal to $\beta_{A_\fp}(M_\fp)$; see Lemma~\ref{le:top-class}.  From \cite[Lemma 3.7]{Iyengar/Khare/Manning:2024a} one gets an exact sequence of $\mco$-modules
\begin{equation}
\label{eq:cmodAM}
0\longrightarrow \coker \kappa_\lambda(M) \longrightarrow \cmod_\lambda(A)^\mu
        \longrightarrow \cmod_\lambda(M) \longrightarrow 0\,,
    \end{equation}
where $\mu$ is the rank of $A_\fp$-module $M_\fp$. A similar argument yields the following analogue for congruence ideals. 

\begin{lemma}
    \label{le:diff}
For $(A,\lambda)$ in $\cato_\mco(c)$ and any finitely generated $A$-module $M$, the map $\kappa_\lambda(M)$ is one-to-one, and there is an equality
\[
\eta_\lambda(A)=\ann_\mco (\coker \kappa_\lambda(M)) \cdot \eta_\lambda(M)
\]
as $\mco$-submodules of $\tfree{\ext^c_A(\mco,\mco)}$. \qed
\end{lemma}

Fix $(A,\lambda)$ in $\cato_\mco$ and a finitely generated $A$-module $M$, and set $\mu=\rank_{A_\fp}(M_\fp)$. It follows from  Theorem~\ref{th:Psi-torsion} that as $\mco$-modules, one has an isomorphism
\begin{equation}
\label{eq:cmod-structure}
\cmod_\lambda(M)\cong \bigoplus_{i=1}^{\mu} \frac{\mco}{(\varpi^{e_i})} 
            \quad \text{for integers $0\le e_1\le e_2\le  \cdots \le e_\mu<\infty$.}
\end{equation}
The following observation is straightforward to verify. The last statement follows from the fact that $\eta_\lambda(A)$ annihilates $\cmod_\lambda(M)$; see (\ref{eq:cmodAM}).

\begin{lemma}
\label{le:eta-psi}
\pushQED{\qed}
    One has $\eta_\lambda(M)=(\varpi^{e_1}) \tfree{\ext^c_A(\mco,\mco)}$, and hence 
    \[
    \cmod_\lambda(M)\cong \tfree{\ext^c_A(\mco,\mco)}/\eta_\lambda(M) \quad \text{when $\mu=1$}. \qedhere
    \] 
    Moreover $\eta_\lambda(A) \subseteq (\varpi^{e_\mu})$, and hence if $\eta_\lambda(A)=\eta_\lambda(M)$, then $\cmod_\lambda(M)=\cmod_\lambda(A)^\mu$, where $\mu$ is the rank of $M$ at $\lambda$.
\end{lemma}

Thus the congruence module of $M$ is a more refined invariant than the congruence ideal. Nevertheless, in this write-up, we focus on the congruence ideal, for reasons explained in the Introduction. 

One of the key results in our work is a criterion for $M$ to have a free summand, in terms of the congruence modules of $A$ and of $M$; see  \cite[Theorem~9.2]{Iyengar/Khare/Manning:2024a}. Here is the analogue involving the congruence ideals.

\begin{theorem}
\label{th:splitting}
Fix $A$ in $\cato_\mco(c)$ and a finitely generated $A$-module $M\ne 0$.  When the ring $A$ is Gorenstein and $M$ is maximal Cohen-Macaulay, $\eta_\lambda(M)=\eta_\lambda(A)$ if and only if $M\cong A^\mu\oplus W$, where $\mu\colonequals \beta_{A_\fp}(M_\fp)$ and $W_\fp=0$.
\end{theorem}

\begin{proof}
The equality  $\eta_\lambda(M)=\eta_\lambda(A)$ holds if and only if $\ann_\mco(\coker \kappa(M))=\mco$; equivalently, if and only if $\coker \kappa(M) =0$; see Lemma~\ref{le:diff}. Thus the desired statement is equivalent to  \cite[Theorem~9.2]{Iyengar/Khare/Manning:2024a}.    
\end{proof}

\subsubsection*{Notes}
The import of Lemma~\ref{le:diff} is that, in the number theoretic context, in particular in modularity lifting, one has often a good grasp on the module $M$, and little information on the ring $A$; this has been explained in the Introduction. Lemma~\ref{le:diff} allows one to deduce information on $A$ through that on $M$. 

The K\"unneth map~\eqref{eq:kunneth} factors through $\overline{\ext}^c_A(\mco,M)$, the \emph{bounded cohomology} module  of the pair $(\mco,M)$, introduced in \cite{Avramov/Veliche:2007}. These cohomology modules sit in a long exact sequence
\[
\cdots \longrightarrow \overline{\ext}^i_A(\mco,M)\longrightarrow \ext^i_A(\mco,M) \longrightarrow
         \widehat{\ext}^i_A(\mco,M)\longrightarrow \overline{\ext}^{i+1}_A(\mco,M)\longrightarrow\cdots
\]
where the $\widehat{\ext}^i_A(\mco,M)$ are the \emph{stable cohomology} modules of $(\mco,M)$. When the ring $A$ is Gorenstein, this coincides with the Tate cohomology modules of the pair; see, for instance, Buchweitz~\cite{Buchweitz:2021a}. 

The stable cohomology $\widehat{\ext}^*_A(L,M)$ of a pair $(L,M)$ is equal to $0$ when one of $L$ or $M$ has finite projective dimension, which, for maximal Cohen-Macaulay modules, is equivalent, to freeness. This remark clarifies why properties of the K\"unneth map control the existence of free summands in $M$, as in Theorem~\ref{th:splitting}.

The analogue of the bounded and stable cohomology modules for the augmentation $R\to k$ of a local ring to its residue field has been studied extensively in local algebra; see~\cite{Avramov/Veliche:2007}, and more recently, in the work of the first author with Maitra and Tribone~\cite{Iyengar/Maitra/Tribone:2025}. In this context, the long exact sequence above take the form
\[
\ext^*_R(k,R)\otimes_k \mathrm{Tor}^R_*(k,N) \longrightarrow \ext^*_R(k,N) 
            \longrightarrow \widehat{\ext}^*_R(k,N)\longrightarrow
\]
See \cite{Avramov/Veliche:2007} for some new applications of this sequence in commutative algebra.

\subsection*{Change of modules.}
Let $\omega_A$ be a dualizing complex for $A$ normalized as in 
\cite[\href{https://stacks.math.columbia.edu/tag/0A7M}{Tag 0A7M}]{stacks-project}. It exists because $A$ is a complete local ring; the normalization ensures that it is  unique up to an isomorphism of complexes.  Given a maximal Cohen-Macaulay module $M$, let $M^\vee$ denote the $A$-module $\hh_{d}(\Hom_A(M,\omega_A))$, where $d$ is the Krull dimension of $A$. We say $M$ is \emph{self-dual} if $M\cong M^\vee$; see \cite[\S4]{Iyengar/Khare/Manning:2024a}.

Fix a surjective map $\pi\colon M\to N$ where $M,N$ are maximal Cohen-Macaulay $A$-modules that are self-dual and such that $\pi_\fp$ is an isomorphism. Consider the following diagram where $\alpha$ is defined to make the diagram commute:
\[
\begin{tikzcd}
\ext^c_A(\mco,M) \arrow[d,dashrightarrow, "\alpha"]  \arrow[rr,"{\ext^c_A(\mco,\pi)}"]
            && \ext^c_A(\mco,N) \arrow[d, "\cong"]  \\
\ext^c_A(\mco,M) \cong \ext^c_A(\mco,M^\vee)&&  \arrow[ll,"{\ext^c_A(\mco,\pi^\vee)}"] \ext^c_A(\mco,N^\vee)
\end{tikzcd}
\]
 By Lemma~\ref{le:ext-tfree} the source and target of $\alpha$ are free, that is to say
\[
\alpha\colon \ext^c_A(\mco,M) \longrightarrow \ext^c_A(\mco,M)
\]
is an endomorphism of a free $\mco$-module, of rank equal to $\beta_{A_\fp}(M_\fp)$. In \cite[Proposition~4.4]{Iyengar/Khare/Manning:2024a}, we prove that 
\[
\length_\mco \Psi_\lambda(M) = \length_\mco \Psi_\lambda(N) + \length_\mco (\mco/(\det \alpha))\,. 
\]
Here is the analogous statement in terms of congruence ideals; it can be proved along the same lines as \emph{op.~cit.}

\begin{lemma}
\label{le:change-of-module}
    One has an equality $\eta_\lambda(M)=\ann_\mco(\coker \alpha)\cdot \eta_\lambda(N)$. \qed
\end{lemma}

\subsubsection*{Notes} 
Combined with the invariance of domain property, Proposition~\ref{pr:invariance-of-domain} below, Lemma~\ref{le:change-of-module} plays a crucial role in establishing $R=\mbb{T}$ theorems in number theory; see the discussion on Wiles' work in Section~\ref{se:NT}.

It is worth extracting the observation below, implicit in the proof of Lemma~\ref{le:change-of-module}. 

\begin{lemma}
\label{le:eta-test}
If there exists an $A$-linear map $M\to N$ of finitely generated $A$-modules such that the induced map $\Hom_A(N,\mco)\to \Hom_A(M,\mco)$ is surjective, then one has that $\eta_\lambda(M)\subseteq \eta_\lambda(N)$. 
\end{lemma}

\begin{proof}
Let $f \colon M\to N$ be such a map, set $f_*\colonequals \ext^c_A(\mco,f)$ and $f^*=\Hom_A(f,\mco)$. Consider the diagram where the horizontal maps are the maps \eqref{eq:eta-defn} involved in the definition of the congruence ideals:
\[
\begin{tikzcd}
    \ext^c_A(\mco,M) \otimes_\mco \Hom_A(M,\mco) \arrow[r,"\upsilon_M"] & \ext^c_A(\mco,\mco) \\
    \ext^c_A(\mco,N) \otimes_\mco \Hom_A(N,\mco) \arrow[r,"\upsilon_N"] 
            \arrow[u, leftarrow, shift left=6ex,"{f_*}"]
            \arrow[u,shift right=6ex,"{f^*}" swap]
    & \ext^c_A(\mco,\mco) \arrow[u,equals]
\end{tikzcd}
\]
One can check that for $\zeta\in \ext^c_A(\mco,M)$ and $\alpha\in \Hom_A(N,\mco)$, one 
\[
\upsilon_M(\zeta\otimes f^*\alpha) = \upsilon_N(f_*\zeta\otimes \alpha)
\]
Given the hypothesis that $f^*$ is surjective, it is then clear that $\eta_\lambda(M)\subseteq\eta_\lambda(N)$.
\end{proof}

Here is one application of this result.

\begin{proposition}
    \label{pr:trace-map}
 Fix a Cohen-Macaulay ring $(A,\lambda)$ in $\cato_{\mco}(c)$  with dualizing module $\omega_A$ and a finitely generated $A$-module $M$. 
 \begin{enumerate}[\quad\rm(1)]
     \item 
     When $M$ is a Cohen-Macaulay module one has $\eta_\lambda(M)=\eta_\lambda(M^\vee)$.
     \item 
     If $\beta_{A_\fp}(M_\fp)=1$ and the trace map $\Hom_A(M,\omega_A)\otimes_AM\to \omega_A$ is onto, then $\eta_{\lambda}(M)=\eta_{\lambda}(A)$.
      \end{enumerate}
\end{proposition}

 \begin{proof}
     By the local duality theorem~\cite[\href{https://stacks.math.columbia.edu/tag/0A7C}{Tag 0A7C}]{stacks-project} the functor $M\mapsto \Hom_A(M,\omega_A)$ on the bounded derived category of $A$ induces an isomorphism of $\mco$-modules
     \[
     \ext^c_A(\mco,M) \cong \Hom_A(M^\vee,\mco)\,.
     \]
    This also uses the hypothesis that the module $M$ is maximal Cohen-Macaulay, and the fact that the ring $\mco$ is Gorenstein. Given this, part (1) is immediate from the definition~\eqref{eq:eta-defn} of congruence ideals.
     
     As to (2), one has always that $\eta_{\lambda}(A)\subseteq \eta_{\lambda}(M)$, by Lemma~\ref{le:eta-psi}, so it suffices to find an $A$-linear map $f\colon M\to \omega_A$ such that the induced map $\Hom_{A}(f,\mco)$ is onto, for then one gets that
     \[
     \eta_{\lambda}(M)\subseteq \eta_{\lambda}(\omega_A) = \eta_{\lambda}(A)\,,
     \]
     where the inclusion is by Lemma~\ref{le:eta-test} and the equality is by (1).
     
     The surjectivity of the trace map is tantamount to the existence of a surjective $A$-linear map
\[
\coprod_{i=1}^n f_i\colon M^n\twoheadrightarrow \omega_A \quad \text{for $f_i\colon M\to \omega_A$}\,.
\]
Since $\Hom_A(\omega_A,\mco)\ne 0$, as can be verified by localizing at $\fp$, one can find a nonzero $A$-linear map $g\colon \omega_A \to \mco$. Since the image of such a map is an ideal of $\mco$ and hence isomorphic to $\mco$, so we can assume that $g$ is surjective. Then the composition $g(\coprod f_i)$ is also surjective. Since $\mco$ is local, it follows that $gf_i\colon M\to \mco$ is surjective for some integer $i$. This map factors through $M/\fp M$ and gives rise to the commutative diagram 
\[
\begin{tikzcd}
    M \arrow[r,"f_i"] \arrow[d] & \omega_A \arrow[d,twoheadrightarrow,"g"] \\
    \tfree{(M/\fp M)} \arrow[r,twoheadrightarrow,"\overline{gf_i}"] & \mco
\end{tikzcd}
\]
The hypotheses $\beta_{A_\fp}(M_\fp)=1$ translates to $\rank_{\mco}(M/\fp M)=1$. It thus follows from the diagram above that the induced map $\overline{gf_i}$ is an isomorphism. Hence the map  $\Hom_A(gf_i,\mco)$ is an isomorphism, which is more that what we need.
 \end{proof}

\begin{remark}
In the context of the Proposition~\ref{pr:trace-map}(2), assume also that $M$ is maximal Cohen-Macaulay module. It follows from local duality that when the trace map is surjective, the natural map $A\to \Hom_A(M,M)$ is an isomorphism; that is to say, $M$ is \emph{balanced}. The converse holds when $c=0$, but for $c\ge 1$ surjectivity of trace is a strictly stronger condition.

Another comment: the converse of Proposition \ref{pr:trace-map}(2) does not hold; see \cite[Remark 3.13]{Bockle/Khare/Manning:2021b}.    
\end{remark}
 
\subsection*{The structure of $\tfree{\ext^*_A(\mco,\mco)}$.}
To go further, we need one more input from commutative algebra. Fix $(A,\lambda)$ in $\cato_\mco(c)$. The graded $\mco$-module 
\[
\ext^*_A(\mco,\mco)=\{\ext^i_A(\mco,\mco)\}_{i\geqslant 0}
\]
has a structure of an associative algebra, with product given by Yoneda composition. The product agrees (up to sign) with the composition of maps in the derived category of $A$. Typically, this Ext algebra is highly non-commutative, and non-zero in infinitely many degrees, unless $A$ is regular; see \cite{Gulliksen/Levin:1969a} or \cite[Section~10]{Avramov:1998a}.

By construction, one has an inclusion of subrings
\[
\tfree{\ext^*_A(\mco,\mco)} \subseteq \ext^*_{A_\fp}(K,K)\cong \wedge_K \ext^1_{A_\fp}(K,K)\,,
\]
where $K$ is the residue field of $A_\fp$; said otherwise, the field of fractions of $\mco$. The isomorphism holds because the local ring $A_\fp$ is regular; see \cite[Chapter 2, \S4]{Gulliksen/Levin:1969a}. In particular, the ring on the right is the graded-commutative, and hence so is its subring $\tfree{\ext^*_A(\mco,\mco)}$. Thus, by the universal property of exterior algebras, the inclusion $\ext^1_A(\mco,\mco)\subset \ext^*_A(\mco,\mco)$ induces a map of graded $\mco$-algebras
\[
\xi_\lambda \colon \wedge^*_{\mco} \ext^1_A(\mco,\mco) \longrightarrow \tfree{\ext^*_A(\mco,\mco)}\,.
\]
Here is the punchline; see \cite[Theorem~6.8]{Iyengar/Khare/Manning:2024a}. It may be seen as an integral version of Serre's theorem describing the Ext-algebra of a regular local ring, mentioned above. It  is a critical ingredient in tracking the change of congruence modules and congruence ideals; see Proposition~\ref{pr:invariance-of-domain} and Theorem~\ref{th:deformation}.

\begin{theorem}
   \label{th:serre}
   For $(A,\lambda)$ in $\cato_\mco(c)$, the map $\xi_\lambda$ defined above is bijective.  In particular there is a natural isomorphism of $\mco$-modules 
   \[
   \tfree{\ext^c_A(\mco,\mco)} \cong \wedge^c_\mco \Hom_{\mco}(\fp/\fp^2,\mco),.
   \]
\end{theorem}

\begin{proof}[Sketch of proof]
    In \cite{Iyengar/Khare/Manning:2024a} we gave a proof that draws on the methods of differential graded algebras, and in particular \cite{Iyengar:2001}, which in turn relies heavily on the work of Tate~\cite{Tate:1957a} and Andr\'e~\cite{Andre:1982}. Here is a different argument that hinges on a natural map 
    \[
    \ext^*_A(\mco,\mco)\longrightarrow \Hom_\mco (\wedge^*_{\mco}(\fp/\fp^2),\mco)\,,
    \]
    introduced by Grothendieck~\cite[\S3]{Grothendieck:1957b}, and called the \emph{fundamental local homomorphism} by Lipman~\cite[pp. 111]{Lipman:1984a}. It is adjoint to the composition
    \[
\ext^i_A(\mco,\mco)\otimes_\mco \wedge_{\mco}^i(\fp/\fp^2)  \to \ext^i_A(\mco,\mco)\otimes_\mco \mathrm{Tor}_i^A(\mco,\mco)\to \mathrm{Tor}^A_0(\mco, \mco)\cong \mco\,.
\]
In \cite[(2.9)]{Lipman:1987a} Lipman asks us to verify that the fundamental local homomorphism is a map of graded algebras, where the product on the left is the usual one, and the one on the right is induced by the natural co-algebra structure on $\wedge^*_{\mco}(\fp/\fp^2)$; in particular, it is graded-commutative. Since the target of the map above is torsion-free as an $\mco$-module it induces the map of graded algebras on the right:
    \[
     \wedge^*_{\mco} \ext^1_A(\mco,\mco) \xrightarrow{\ \xi_\lambda\ } \tfree{\ext^*_A(\mco,\mco)} 
            \longrightarrow \Hom_\mco (\wedge^*_{\mco}(\fp/\fp^2),\mco)\,.
    \]
    Keeping in mind \eqref{eq:normal}, it is clear that the composition is an isomorphism. Localizing at $\fp$ it is easy to verify that both maps above are one-to-one, so we deduce that they are both isomorphisms. 

It remains to do the homework that Lipman assigns. The solution we came up with is similar to the proof in Gulliksen and Levin~\cite[Chapter 2, \S3]{Gulliksen/Levin:1969a} that, for a local ring $R$ with residue field $k$, the composition product on $\ext_R^*(k,k)$ is dual to a co-product on $\mathrm{Tor}^R_*(k,k)$. This again uses methods of differential graded homological algebra, so perhaps these techniques are unavoidable. But the proof sketched above yields more: $\tfree{\ext^*_A(\mco,\mco)}$ is a Hopf algebra over $\mco$, with coproduct dual to the product on $\mathrm{Tor}^A_*(\mco,\mco)$. In particular, it commutative and co-commutative.
\end{proof}

\subsection*{Invariance of domain.}
Let $\alpha\colon (A,\lambda_A)\to (B,\lambda_B)$ be a morphism in $\cato_\mco(c)$. For any $B$-module $N$, functoriality of Ext gives a commutative diagram
\[
\begin{tikzcd}
    \ext^c_A(\mco,N) \otimes_\mco \Hom_A(N,\mco) \arrow[r] & \ext^c_A(\mco,\mco) \\
    \ext^c_B(\mco,N) \otimes_\mco \Hom_B(N,\mco) \arrow[r] 
            \arrow[u,shift left=6ex,"{\ext_{\alpha}^c(\mco,N)}"]
            \arrow[u,shift right=6ex,"{\Hom_{\alpha}(N,\mco)}" swap]
    & \ext^c_B(\mco,\mco) \arrow[u,"{\ext_{\alpha}^c(\mco,\mco)}" swap]
\end{tikzcd}
\]
Passing to torsion-free quotients, one gets an induced commutative diagram 
\[
\begin{tikzcd}
    \eta_{\lambda_A}(N) \subseteq \tfree{\ext^c_A(\mco,\mco)} \\
    \eta_{\lambda_B}(N) \subseteq \tfree{\ext^c_B(\mco,\mco)} 
     \arrow[u,shift left=6ex,"{\eta_{\alpha}(N)}"]
            \arrow[u,shift right=6ex,"{\ext^c_\alpha(\mco,\mco)}" swap]   
\end{tikzcd}
\]
The statement below, which we think of as an "invariance of domain" property of congruence ideals, is analogous to \cite[Theorem~7.4]{Iyengar/Khare/Manning:2024a}, and can be proved along the same lines; see Proposition~\ref{pr:invariance-of-domain} for an extension of this result.

\begin{proposition}
\label{pr:invariance-of-domain}
    In the situation above, if the map $\alpha$ is surjective, then the maps $\eta_{\alpha}(N)$ and $\tfree{\ext^c_{\alpha}(\mco,\mco)}$ are isomorphisms. \qed
    \end{proposition}

\subsubsection*{Notes}
The invariance of domain property is obvious for $c=0$, and used implicitly in Wiles' work~\cite{Wiles:1995}; see Section~\ref{se:NT}, and in particular the discussion around \eqref{eq:goingup}. 

\subsection*{Deformations.}
Fix $(A,\lambda_A)$ in $\cato_\mco(c)$, an element $f$ in $\fp\colonequals \Ker\lambda_A$, and consider the quotient ring $B\colonequals A/(f)$. The augmentation $\lambda_A$ induces an augmentation $\lambda_B\colon B\to \mco$. In the statement below $\fp^{(2)}$ is the second symbolic power of $\fp$.

\begin{lemma}
 The pair $(B,\lambda_B)$ is in $\cato_\mco(c-1)$ if and only if $f$ is not in $\fp^{(2)}$. 
\end{lemma}

\begin{proof}
Since the local ring $A_\fp$ is regular,  the local ring $B_\fp$ is also regular, with $\dim B_\fp = \dim A_\fp -1$, if and only if the image of $f$ in $A_\fp$ is not in $\fp^2A_\fp$; see \cite[Proposition~2.2.4]{Bruns/Herzog:1998a}. The latter condition is equivalent to $f$ not in $\fp^2A_\fp \cap A=\fp^{(2)}$.
\end{proof}

Fix $f\in \fp\setminus \fp^{(2)}$ and $(B,\lambda_B)$ as above, so that $(A,\lambda_A)\to (B,\lambda_B)$ is a surjective map in $\cato_\mco$, with $B$ in $\cato_\mco(c-1)$. Let $M$ be a finitely generated $A$-module on which $f$ is not a zerodivisor and $N\colonequals M/fM$, viewed as a $B$-module.

In what follows, given an element $x$ in a finitely generated $\mco$-module $U$, we write $\mathrm{ord}(x)$ for the \emph{order ideal} of $x$; namely, the ideal  $\{\alpha(x)\mid \alpha \text{ in } \Hom_\mco(U,\mco)\}$ of $\mco$. Note that $\mathrm{ord}(x)$ is $(0)$ if and only if $x$ is torsion, and equal to $\mco$ if and only if the image of $x$ in $\tfree U$ can be extended to a basis of $U$. 

The congruence modules of $M$ and $N$, with respect to $\lambda_A$ and $\lambda_B$, respectively, as related:
\[
\length_{\mco}\cmod_{\lambda_B}(N) = \length_{\mco}\cmod_{\lambda_A}(M) + \mu\cdot \length_{\mco}(\frac{\mco}{\mathrm{ord}([f])})\,,
\]
where $\mu=\rank_{A_\fp}(M_\fp)$ and  $[f]$ denotes the class of the element $f$ in the $\mco$-module $\fp/\fp^2$. This is contained in the proof of \cite[Theorem~8.2]{Iyengar/Khare/Manning:2024a}; see also \cite[Lemma~8.7]{Iyengar/Khare/Manning:2024a}. Here is the analogous statement for congruence ideals. It is a key input in forthcoming applications to the Bloch-Kato conjecture; see Theorem \ref{th:BK}.

\begin{theorem}
\label{th:deformation}
In the situation and notation above, there is a natural exact sequence of $\mco$-modules 
\[
0\longrightarrow \frac{\tfree{\ext^c_A(\mco,\mco)}}{\eta_{\lambda_A}(M)}
    \longrightarrow \frac{\tfree{\ext^{c-1}_B(\mco,\mco)}}{\eta_{\lambda_B}(N)}
        \longrightarrow \frac{\tfree{\ext^{c-1}_B(\mco,\mco)}}{\mathrm{ord}([f])\tfree{\ext^{c-1}_B(\mco,\mco)}}
            \longrightarrow 0\,.
\]
\end{theorem}

\begin{proof}[Sketch of proof]
Arguing as in the proof of \cite[Theorem~8.2]{Iyengar/Khare/Manning:2024a}, one constructs a commutative diagram 
\[
\begin{tikzcd}
    \eta_{\lambda_A}(M) \arrow[r, hookrightarrow] \arrow[d,"\cong"] & \tfree{\ext^c_A(\mco,\mco)} \arrow[d,"\cong"] & \\
    \eta_{\lambda_B}(N) \arrow[r,hookrightarrow]  & \mathrm{ord}([f])\cdot  \tfree{\ext^{c-1}_B(\mco,\mco)} \arrow[r, hookrightarrow] & \tfree{\ext^{c-1}_B(\mco,\mco)}
\end{tikzcd}
\]  
of $\mco$-linear maps. This gives the desired exact sequence.
\end{proof}

\subsection*{Cotangent modules.}
\label{ss:cotan}
Let $\lambda \colon A\to\mco$ be a map of $\mco$-algebras and set $\fp\colonequals \Ker\lambda$. A local ring is regular if and only if its embedding dimension equals its Krull dimension, so one gets that the pair $(A,\lambda)$ in $\cato_\mco(c)$ if and only if $\rank_\mco(\fp/\fp^2)=\height \fp$. The $\mco$-module $\fp/\fp^2$ is the \emph{cotangent module} of $\lambda$. In what follows the torsion part of the cotagent module also plays a key role:
\[
\Phi_\lambda(A) \colonequals \tors(\fp/\fp^2)\,.
\]
Let $\fitt_c(\fp/\fp^2)$ be the $c$'th Fitting ideal of the $\mco$-module $\fp/\fp^2$; equivalently, the $0$'th Fitting ideal of the $\mco$-module $\Phi_\lambda(A)$. Thus
\[
\fitt_{c}(\fp/\fp^2) = (\varpi^l) \quad\text{where $l\colonequals \length_{\mco} \Phi_\lambda(A)$.}
\]
The result below is the analogue of Theorem~\ref{th:deformation}. It easier to verify; see \cite[Lemma~8.7]{Iyengar/Khare/Manning:2024a}.

\begin{lemma}
    \label{le:deformation-cot}
    \pushQED{\qed}
    Let $(A,\lambda_A)\to (B,\lambda_B)$ be as in Theorem~\ref{th:deformation}, with $B=A/(f)$. One has an equality 
    \[
    \fitt_{0} \Phi_{\lambda_B}(B) = (\varpi^{\ord([f])}) \fitt_{0} \Phi_{\lambda_A}(A) \,. \qedhere
    \]
\end{lemma}

The link between the cotangent module and the congruence ideals is captured in the following result. Here $\mathrm{grade}(\fp,M)$ denotes the length of the longest $M$-regular sequence in $\fp$.

\begin{theorem}
\label{th:fitting-ideal}
Fix $(A,\lambda)$ in $\cato_\mco(c)$ and a finitely generated $A$-module $M$ with $\mathrm{grade}(\fp,M)\ge c$. One has
\[
    \fitt_c(\fp/\fp^2)\cdot\cmod_\lambda(M) =0\,. 
\]  
In particular, $ \fitt_c(\fp/\fp^2)\cdot \tfree{\ext^c_A(\mco,\mco)} \subseteq \eta_\lambda(M)$.
\end{theorem}

\begin{proof}[Sketch of proof]
 Using Theorem~\ref{th:deformation} and Lemma~\ref{le:deformation-cot} one can prove this result by an induction on $c$, as in the proof of \cite[Theorem~8.2]{Iyengar/Khare/Manning:2024a}. The base case $c=0$ follows from the fact that  $0$'th Fitting ideal of a module is contained in the annihilator.
\end{proof}

Under the stronger hypotheses that $A$ is Cohen-Macaulay and flat over $\mco$, one can build on ideas in \cite{Wang:1994,Iyengar/Takahashi:2021} and give a different proof of the result above by  exploiting the link between the Fitting ideal in question and the Jacobian ideal of the $\mco$-algebra $A$; see \cite{Kunz:1986}. Here is a natural question that arises in this context.

       \begin{question}
Does the containment $ \fitt_c(\fp/\fp^2)\subseteq \eta_\lambda(A)$ always hold?
       \end{question}

This hold when $A$ is Cohen-Macaulay, by the preceding theorem. Given the results in \cite{Iyengar/Takahashi:2021}, it seem plausible that statement holds under weaker hypotheses, but we suspect that the inclusion will not hold in general.

\subsection*{Numerical criteria.}
The result below builds on Theorem~\ref{th:splitting}. It is the key to applications of congruence modules, and ideals, to proving $R=\mbb T$ theorems. For illustrations see \cite{Iyengar/Khare/Manning:2024a, Diamond/Iyengar/Khare/Manning:2026a}.

\begin{theorem}
\label{th:defect0}
    Fix $(A,\lambda)$ in $\cato_\mco(c)$ and a finitely generated $A$-module $M$ with $\depth_AM\ge c+1$. Set $\mu\colonequals \rank_{A_\fp}(M_\fp)$. The following conditions are equivalent:
    \begin{enumerate}[\quad\rm(1)]
        \item 
        $A$ is complete intersection and $M\cong A^\mu\oplus W$, with $W_\fp=0$.
        \item
        $\fitt_c(\fp/\fp^2)\cdot\tfree{\ext^c_A(\mco,\mco)} =\eta_\lambda(M)$;
        \item 
        $\mu\cdot \length_\mco\Phi_\lambda(A)=\length_{\mco} \cmod_\lambda(M)$.
    \end{enumerate}
\end{theorem}

\begin{proof}[Sketch of proof]
    The equivalence (1)$\Leftrightarrow$(3) is the content of \cite[Theorem~9.2]{Iyengar/Khare/Manning:2024a}. 

    (2)$\Leftrightarrow$(3) This follows  using Theorem~\ref{th:fitting-ideal} and   Lemma \ref{le:eta-psi}.
    \end{proof}

The hypotheses on the $\depth_AM$ cannot be relaxed, even when $M=A$; here is an example.

\begin{example}
Consider the $\mco$-algebra
\[
A\colonequals \frac{\mco\pos{x,y}}{(x(x-\varpi),x\varpi^2,xy)}\,.
\]
One has $\dim A=2$ whereas $\depth A=0$. For the augmentation $\lambda\colon A\to \mco$ that assigns $x,y$ to $0$, it is easy to verify that $(A,\lambda)$ is in $\cato_\mco(1)$ and that $\eta_\lambda(A)=(\varpi)=\fitt_1(\fp/\fp^2)$. However $A$ is not a complete intersection.
\end{example}

The result below is a simple consequence of Theorem~\ref{th:defect0}, applied to $M=B$, and the invariance of domain property of congruence ideals.

\begin{corollary}
\label{co:iso-criterion-cotangent}
Let $\alpha\colon (A,\lambda_A) \to (B,\lambda_B)$ be a surjective map in $\cato_{\mco}(c)$ with $B$ a complete intersection. If $\fitt_{\mco}\Phi_{\lambda_A}(A)=\fitt_{\mco}\Phi_{\lambda_B}(B)$, then $\alpha$ is bijective and $A$ is a complete intersection. \qed
\end{corollary}

Here is an analogue of the preceding result that involves the congruence ideals of $A$ and $B$. It follows from Theorem~\ref{th:splitting}. See Theorem~\ref{th:gorsum} for a related statement.

\begin{corollary}
\label{co:iso-criterion}
    Let $\alpha \colon (A,\lambda_A) \to (B,\lambda_B)$ be a surjective map in $\cato_{\mco}(c)$ with $A$ Gorenstein and $B$ Cohen-Macaulay. If $\eta_{\lambda_A}(A)=\eta_{\lambda_B}(B)$, then $\alpha$ is bijective. \qed
\end{corollary}

The preceding corollaries are inspired by Wiles' work in \cite[Appendix]{Wiles:1995} that concerns the case $c=0$.

\subsection*{Wiles defect}
Part (3) of Theorem~\ref{th:defect0} can be rephrased as the condition that the invariant
\begin{equation}
\label{eq:wiles-defect}
\delta_{\lambda_A}(M)\colonequals \mu\cdot \length_{\mco}\Phi_{\lambda_A}(A) - \length_{\mco}\cmod_{\lambda_A}(M)\,,
\end{equation}
introduced in \cite{Iyengar/Khare/Manning:2024a} as the \emph{Wiles defect} of $M$ at $\lambda_A$, equals $0$. In fact, as is done in \cite{Iyengar/Khare/Manning:2024a},  many results about the behavior of congruence modules encountered above can be expressed in terms of this invariant. In \cite{Bockle/Khare/Manning:2024} the Wiles  defect $\delta_{\lambda_A}(A)$ is  defined when one further assumes that  $A$ is  Cohen-Macaulay. Under either of these definitions the Wiles defect remains the same on going modulo regular sequences, and coincides for $c=0$, so one sees that $\delta_{\lambda_A}(M)$ of   \cite{Iyengar/Khare/Manning:2024a} coincides with the defect defined in \cite{Bockle/Khare/Manning:2024}, under its more restrictive hypothesis that  $A$ is Cohen-Macaulay and $M=A$.

\subsubsection*{Notes}
Theorem~\ref{th:defect0} generalizes the numerical criterion of Wiles, Lenstra and Diamond; see ~\cite{Wiles:1995, Lenstra:1995, Diamond:1997}. In the context of Corollary~\ref{co:iso-criterion-cotangent}, when $A$ and $B$ are essentially of finite type, the hypotheses on the cotangent modules is equivalent to the condition that the natural map $\Omega_{A/\mco}\otimes_AB\to \Omega_{B/\mco}$ on the modules of K\"ahler differentials over $\mco$ is an isomorphism; Eisenbud and Mazur~\cite{Eisenbud/Mazur:1997} call such a map $\alpha\colon A\to B$ an \emph{evolution} of $B$ over $\mco$. Allowing this language also for algebras in $\cato_\mco$, Corollary~\ref{co:iso-criterion-cotangent} is the statement that complete intersection rings admit no non-trivial evolutions. In their work, which was also inspired by that of Wiles, Eisenbud and Mazur identify other classes of rings that admit no non-trivial evolution.

\section{Calculus of congruence ideals}
\label{se:tool-kit}

In this section we gather some results to help compute congruence modules and congruence ideals. These are guided by, and used in, Section~\ref{se:BKM-revisited}, where we compute congruence ideals of  local deformation rings.

In this section we consider a map of rings $\alpha \colon (A,\lambda_A) \to (B,\lambda_B)$ in $\cato_\mco(c)$ such that following conditions hold:
\begin{itemize}
    \item 
    $\alpha$ is finite, that is to say, $B$ is finitely generated as an $A$-module;
    \item
     $\alpha$ is an isomorphism at $\lambda_A$, meaning that the natural map $\alpha_{\fp_A}\colon A_{\fp_A}\to B_{\fp_A}$ is an isomorphism; equivalently, the $A$-module $B/A$ is not supported at $\fp_A$;
\end{itemize}
Observe that both conditions hold when $\alpha$ is surjective.

Here are some observations that flow immediately from the assumptions on $\alpha$. 

\begin{lemma}
\label{le:finite-isotops}
   In the setup above, for any finitely generated $B$-module $M$ the canonical map of $\mco$-modules
   \[
   \tfree{(M/\fp_AM)} \longrightarrow \tfree{(M/\fp_BM)}
   \]
   is an isomorphism, and hence so is the map $\Hom_B(M,\mco)\to \Hom_A(M,\mco)$.
\end{lemma} 

\begin{proof}
Since one has a natural surjective map 
\[
M/\fp_AM \cong \mco\otimes_AM  \longrightarrow \mco\otimes_BM\cong M/\fp_BM\,,
\]
it suffices to verify the induced map
\[
E\otimes_\mco(\mco\otimes_AM) \longrightarrow E\otimes_\mco(\mco\otimes_BM)
\]
is an isomorphism. This holds because
\begin{gather*}
E\otimes_\mco(\mco\otimes_AM) \cong E\otimes_A M\cong E\otimes_{A_{\fp_A}}M_{\fp_A} \\
E\otimes_\mco(\mco\otimes_BM) \cong E\otimes_B M\cong E\otimes_{B_{\fp_A}}M_{\fp_A}
\end{gather*}
and the map $A_{\fp_A}\to B_{\fp_A}$ is an isomorphism.
\end{proof}

Consider the map induced by $\alpha$ on cotangent modules
\[
\Phi_{\alpha}\colon \fp_A/\fp_A^2 \longrightarrow \fp_B/\fp_B^2\,.
\]
Both these have rank $c$, since $A$ and $B$ are in $\cato_\mco(c)$, and the cokernel is torsion, since $\alpha$ is an isomorphism at $\lambda_A$. Thus the dual map
\begin{equation}
\label{eq:finite-tangent-map}
\Phi_{\alpha}^*\colon \Hom_{\mco}(\fp_A/\fp_A^2,\mco)\longrightarrow \Hom_{\mco}(\fp_B/\fp_B^2,\mco)
\end{equation}
is one-to-one, and of the same rank. In particular, the cokernel of the map $\wedge^c(\Phi_{\alpha}^*)$ has finite length.
This observation comes into play in the statement of the result below.  It is a generalization of the invariance of domain property of congruence modules, Proposition~\ref{pr:invariance-of-domain}, for when $\alpha$ is surjective $\Phi_{\alpha}^*$ is an isomorphism, and hence the congruence modules of $M$ over $B$ and over $A$, respectively, are isomorphic.

\begin{proposition}
\label{pr:finite-module}
With setup as above, let $M$ be a finitely generated $B$-module with $\depth_BM\ge c$. Then one has an exact sequence of $\mco$-modules
\[
    0\longrightarrow \cmod_{\lambda_B}(M) \longrightarrow \cmod_{\lambda_A}(M) \longrightarrow \coker(\wedge^c \Phi_{\alpha}^*)^\mu\longrightarrow 0
    \]
    where $\mu$ is the rank of $M$ at $\lambda_B$.
\end{proposition}

\begin{remark}
\label{re:finite-invariance}
By Theorem~\ref{th:serre} one can identify the $c$'th exterior power of the tangent module of $\lambda_A$ and $\lambda_B$ with   the torsion-free quotients of the $\ext^c_A(\mco,\mco)$ and $\ext^c_B(\mco,\mco)$, respectively, and hence $\wedge^c\Phi_{\alpha}^*$ with the map
 \[
\tfree{\ext^c_{\alpha}(\mco,\mco)} \colon \tfree{\ext^c_A(\mco,\mco)}\longrightarrow \tfree{\ext^c_B(\mco,\mco)}
\]
This observation is used the proof below.
\end{remark}

\begin{proof}
One has a commutative diagram of $\mco$-modules
 \[
    \begin{tikzcd}
        \ext^c_B(\mco,M)  \arrow[d,hookrightarrow] \arrow[r] 
                & \ext^c_A(\mco,M) \arrow[d,hookrightarrow] \arrow[dr,hookrightarrow] \\
        \tfree{\ext^c_B(\mco, M/\fp_BM)}  \arrow[r,hookrightarrow] 
                & \tfree{\ext^c_A(\mco,M/\fp_BM)}
                    &\tfree{\ext^c_A(\mco,M/\fp_AM)} \arrow[l,"\cong"]
    \end{tikzcd}
    \]
The isomorphism follows from Lemma~\ref{le:finite-isotops}. Given Remark~\ref{re:finite-invariance},  the theorem follows once we verify that the map in the top row is an isomorphism, and this is verified exactly as in the proof of ~\cite[Theorem~7.4]{Iyengar/Khare/Manning:2024a}.
\end{proof}

When the map $\alpha$ is also surjective one has the following result about change of congruence ideals; observe that it implies Corollary~\ref{co:iso-criterion}.

\begin{theorem}
    \label{th:gorsum}
Let $\alpha\colon (A,\lambda_A)\to (B,\lambda_B)$ be a surjective map in $\cato_{\mco}(c)$, with $A$ a Gorenstein ring and $B$ Cohen-Macaulay. With $J \colonequals \lambda_A(A[\Ker\alpha])$ one has 
\[
\eta_{\lambda_A}(A) = J \eta_{\lambda_B}(B)\,.
\]
\end{theorem}

The conclusion of the preceding result is equivalent to:
\[
\length_{\mco} {\cmod_{\lambda_A}(A)} = \length_{\mco}{\cmod_{\lambda_B}(B)} + \length_{\mco}{( \mco/J)}\,.
\]
We formulated it in terms congruence ideals, given the focus of this manuscript.

\begin{proof}
Set $I\colonequals \Ker\alpha$ and $\fp\colonequals \Ker\lambda_A$. Since $A$ and $B$ are Cohen-Macaulay rings of dimension $c$ and $\alpha$ is surjective map, $B$ is maximal Cohen-Macaulay as an $A$-module. Moreover, since $\alpha_\fp$ is an isomorphism one has $I_\fp=0$. Hence the $\mco$-module $\ext^c_A(\mco,I)$ is torsion. Keeping in mind that $\ext^i_A(\mco,-) =0$ for $i\le c-1$ on maximal Cohen-Macaulay modules, the  exact sequence of $A$-modules 
\[
0\longrightarrow I\longrightarrow A\longrightarrow B\longrightarrow 0\tag{$\ast$}
\]
induces the exact sequence in the top row of the diagram below of $\mco$-modules
 \[
 \begin{tikzcd}
0 \arrow[r] & \ext^c_A(\mco,A) \arrow[r] \arrow[d]
            & \ext^c_A(\mco,B) \arrow[r] \arrow[d]
            & \ext^{c+1}_A(\mco,I) \arrow[r] & 0 \\
            &\tfree{\ext^c_A(\mco,A/\fp)}\arrow[r,"\cong"] 
            & \tfree{\ext^c_A(\mco,B/\fp B)}
\end{tikzcd}
 \]
 We already know the vertical maps are one-to-one and that the cokernels are $\cmod_{\lambda_A}(A)$ and $\cmod_{\lambda_A}(B)$, respectively. Thus we deduce the second equality below
 \begin{align*}
 \length_\mco \cmod_{\lambda_A}(A)  - \length_\mco \cmod_{\lambda_B}(B) 
    &= \length_\mco \cmod_{\lambda_A}(A)  - \length_\mco \cmod_{\lambda_A}(B) \\
    &= \length_\mco \ext^{c+1}_A(\mco, I)\,.
 \end{align*}
 The first equality is by the invariance of domain property, Proposition~\ref{pr:invariance-of-domain}. It thus remains to verify that the length of $\ext^{c+1}_A(\mco,I)$ equals that of $\mco/J$.

This is where the Gorenstein property of $A$ comes in. It is convenient to phrase the argument in terms of the stable derived category; see, for instance, \cite[\S7.5]{Buchweitz:2021a}. Writing $\sHom_A(-,-)$ for morphisms in the stable category, one has isomorphisms
 \[
 \ext^{c+1}_A(\mco,I) \cong \ext^1_A(\Omega^c \mco, I) \cong \sHom_A(\Omega^c\mco,\Omega^{-1}I)
 \]
where the first one is standard dimension-shifting, and the second one holds because the $A$-module $\Omega^c\mco$ is maximal Cohen-Macaulay. Now we use Auslander duality~\cite{Buchweitz:2021a}: This gives the first isomorphism below
\begin{align*}
\sHom_A(\Omega^c\mco,\Omega^{-1} I)^\vee 
    & \cong \sHom_A(\Omega^{-1} I,\Omega^{1-(c+1)}\Omega^c\mco)   \\
    &\cong \sHom_A(\Omega^{-1} I,\Omega^{-c}\Omega^c \mco) \\
    &\cong \sHom_A(\Omega^{-1} I, \mco) \\
    &\cong \sHom_A(B,\mco)\\
    &\cong \coker (\Hom_A(B,A)\to \Hom_A(B,\mco)\cong \mco)\\
    &\cong \mco/J
\end{align*}
Here $(-)^\vee$ is the dual with respect to $E/\mco$, where $E$ is the field of fractions of $\mco$. The third isomorphism holds because $I$ is maximal Cohen-Macaulay; this follows from the exact sequence ($\ast$). This sequence also yields that $\Omega^{-1}I\cong B$ in the stable derived category of $A$. This completes the proof.
\end{proof}

\begin{corollary}
\label{co:fiber-product}
Let $A$ be a Gorenstein ring and $I,J$ ideals in $A$ such that the following conditions hold:
\begin{enumerate}[\quad\rm(1)]
    \item 
    $V(I)$ and $V(J)$ are unions of components of $\Spec A$;
    \item 
    $I\cap J=0$ and $\mathrm{height}(I+J)\ge 1$;
    \item 
    $A/J$ has no embedded associated primes;
    \item
    $A/I$ is Cohen-Macaulay.
\end{enumerate}
Let $\alpha\colon A\to A/I$ be the canonical surjection. For each smooth $\mco$-point $\lambda\colon A/I\to \mco$ there is an equality $\eta_{\lambda\alpha}(A) =  \lambda\alpha(J) \cdot \eta_{\lambda}(A/I)$.
\end{corollary}

\begin{proof}
 Given Theorem~\ref{th:gorsum} we only need to verify that 
 \[
 \lambda\alpha(A[I]) = \lambda\alpha(J)\,.
 \]
 Consider the Mayer-Vietoris exact sequence of $A$-modules
 \[
 0\longrightarrow A \longrightarrow \frac AI \oplus \frac AJ \longrightarrow \frac A{I+J}\longrightarrow 0\,.
 \]
 Applying $\Hom_A(A/I,-)$ and keeping in mind that $\ext^i_A(A/I,A)=0$ for $i\ge 1$ because $A$ is Gorenstein and $A/I$ is maximal Cohen-Macaulay, one gets an exact sequence of $C$-modules
 \[
 0\longrightarrow A[I] \longrightarrow \frac AI \oplus \Hom_A(A/I,A/J) \longrightarrow \frac A{I+J}\longrightarrow 0\,.
 \]
 The associated primes of $\Hom_A(A/I,A/J)$ are the associated primes of $A/J$ contained in $V(I)$. Thus
 conditions (1)---(3) imply $\Hom_A(A/I,A/J)=0$, and so we deduce from the exact sequence above that
 \[
 I + A[I] = I + J\,.
 \]
Since $\lambda(I)=0$, this gives the desired equality.
\end{proof}

We record the following observation, connected with Theorem~\ref{th:gorsum}, for later use.

\begin{lemma}
Let $\alpha\colon A\to B$ be a surjective map of local rings in $\cato_\mco(c)$ with $A$ Gorenstein and $B$ Cohen-Macaulay. Then $B$ is Gorenstein if and only if the ideal $A[\Ker\alpha]$ of $B$ is principal.
\end{lemma}

\begin{proof}
Since $A$ is Gorenstein and  $\dim A=\dim B$ the $B$-module $\Hom_A(B,A)$, that is naturally identified with $A[\Ker\alpha]$, is a dualizing module for $B$. The result follows from this observation, since a Cohen-Macaulay local ring is Gorenstein if and only if its dualizing module is free, and then necessarily of rank one.
\end{proof}

\section{Determinantal rings}
\label{se:determinantal-rings}
In this section we calculate congruence ideals for certain generic determinantal rings. These are used in the next section to calculate congruence ideals of Steinberg deformation rings. The advantage of starting with the determinantal case is that we can exploit the various symmetries of the situation to simplify the calculations. 

Throughout $(\mco,\varpi,k)$ is a complete discrete valuation ring as usual.

\subsection*{Analytic $\mco$-algebras}
In the remainder of this manuscript, we only consider rings in the category $\cato_{\mco}$ that are also analytic $\mco$-algebras, that is to say, rings of the form
\[
A=\frac{\mco\pos{x_1,\dots,x_n}}{(f_1,\dots,f_m)}
\]
where $\mco\pos{x_1,\dots,x_n}$ is the ring of power series in indeterminates $\bos{x}\colonequals x_1,\dots,x_n$.
An augmentation $\lambda\colon A\to \mco$ of $\mco$-algebras is determined its value on $\bos{x}$, that is to say, by elements $\bos{a}\colonequals a_1,\dots,a_n$ in $(\varpi)$, where $a_i\colonequals \lambda(x_i)$. The cotangent module of the map $\lambda$  has a presentation
\[
\mco^m \xrightarrow{\ \left.\left(\frac{\partial f_j}{\partial {x_i}}\right)\right|_{\bos{a}}\ }\mco^n \longrightarrow \frac{\fp}{\fp^2}\longrightarrow 0\,.
\]
where $(\frac{\partial f_j}{\partial {x_i}})$ is the Jacobian matrix associated to $\bos{f}$. In particular, the Fitting ideals of $\fp/\fp^2$ are the Fitting ideals of the $A$-module $\hat{\Omega}_{A/\mco}$ of continuous K\"ahler differentials of $A$ as an $\mco$-algebra, extended to $\mco$ along $\lambda$, that is to say
\[
\fitt_{i}(\fp/\fp^2) = \lambda_{\bos{a}}(\fitt_i(\hat{\Omega}_{A/\mco}))\,.
\]
Since the $\bos{f}$ are polynomials,  we can even replace $\hat{\Omega}_{A/\mco}$ by $\Omega_{A'/\mco}$, where $A'$ is $\mco[\bos{x}]$, the polynomial ring over $\bos{x}$, modulo the ideal $(\bos{f})$. 

\subsection*{Determinantal rings}
 In the next few paragraphs we compute the congruence ideals and torsion in the cotangent modules of smooth $\mco$-valued points on the spectrum of the determinantal ring defined by size $2$ minors of a generic matrix of size $2\times n$. For applications to number theory presented later in this section we only require the case $n=4$, and most of the calculations can be checked in a computer algebra system, such as Magma,  and we have done just that. We tackle the general case, in the anticipation that these results will prove useful later on.

Let $n\ge 2$ be a positive integer and consider a matrix of size  $2\times n$ consisting of indeterminates:
\[
X\colonequals
\begin{bmatrix}
    x_{11} & x_{12} & \dots & x_{1n}\\
    x_{21} & x_{22} & \dots & x_{2n}
\end{bmatrix}
\]
Let $P\colonequals \mco\pos{X}$ be the  ring of formal power series over $\mco$ in these indeterminates and set
\[
A\colonequals P/I_2(X) 
\]
where $I_2(X)$ is the ideal generated by minors of size $2$ in $X$. 

The ring $A$ can be obtained from ring $\mco[X]/I_2(X)$, where $\mco[X]$ denotes the polynomial ring in the indeterminates, by completing it at the ideal $(X)$. This change of rings is flat, and moreover the ideal $I_2(X)$ is homogeneous, for the standard grading on $\mco[X]$. This means that most properties of the ideal  $I_2(X)$, or the ring $A$, can be verified by working in $\mco[X]$, for the corresponding quotient ring. For this reason, we gloss over the distinction between $\mco[X]/I_2(X)$ and the ring $A$ in the sequel. Moreover, $A$ and $\mco[X]/I_2(X)$ are torsion-free $\mco$-modules, so one can often deduce their properties from the corresponding properties of rings obtained by going modulo the uniformizer, $\varpi$.

For instance, from \cite[Theorem~7.3.1]{Bruns/Herzog:1998a} we get that, viewed in $k[X]$, the ideal $I_2(X)$ is prime and of height $n-1$, and the quotient ring $k[X]/I_2(X)$ is Cohen-Macaulay. These properties carry over to $I_2(X)$ viewed as an ideal in $P$. In particular, the ring $A$ is a complete Cohen-Macaulay local domain of Krull dimension $n+2$; observe that the Krull dimension of $P$ is $2n+1$. See also \cite[Section~3.4]{Bruns/Conca/Raicu/Varbaro:2022}.

An $\mco$-valued point $\lambda\colon A\to \mco$ is determined by a matrix
\[
\bos{a}\colonequals 
\begin{bmatrix} 
a_{11}&\dots&a_{1n} \\
a_{21}&\dots&a_{2n}
\end{bmatrix}
\]
with coefficients in $(\varpi)$ whose rank is $\le 1$;  here $a_{ij}=\lambda(x_{ij})$. We write $\lambda_{\bos{a}}$ for this augmentation, and set
\[
w_{\bos a}\colonequals \sup\{s\ge 0\mid \bos a\equiv \bos {0}\mod \varpi^s\} = \inf\{\ord(a_{ij}) \mid i,j\}\,.
\]
Evidently $w_{\bos a}$ is finite if and only if $(\bos{a})\ne 0$.

\begin{proposition}
The height of the ideal $\fp_{\bos{a}}\colonequals \Ker\lambda_{\bos{a}}$ equals $n+1$, and  
the pair $(A,\lambda_{\bos{a}})$ is in $\cato_{\mco}(n+1)$ if and only if  $\bos{a}\ne 0$ and has rank $1$. When this holds
\[
\fitt_{n+1}(\fp_{\bos{a}}/\fp^2_{\bos{a}}) = (\varpi^{(n-1) w_{\bos a}})\,.
\]
\end{proposition}

\begin{proof}
The ring $A$ is Cohen-Macaulay, of dimension $n+2$, and the Krull dimension of $A/\fp_{\bos{a}}\cong \mco$ is one.
Thus the height of $\fp_{\bos{a}}$ is $n+1$. 

Since $I_2(X)$ is generated by homogeneous elements, for the standard grading on $P$, it clear that $A$ is not regular at $(X)$,
the irrelevant ideal of $A$; that is to say, $\lambda_{\bos{a}}$ is not regular when $\bos{a}=0$.

Assume $\bos{a}$ is such that $a_{st}\ne 0$ for some $(s,t)$. Then in $P$ the indeterminate $x_{st}$ becomes invertible after localization at $\fp_{\bos{a}}$, and it is easy to verify that 
\[
I_2(X)P_{\fp_\bos{a}} = (x_{ij} \mid (i,n)\ne (s,t))P_{\fp_{\bos{a}}}\,,
\]
Thus  $A_{\fp_\bos{a}}$ is a regular local ring, that is to say, $\bos{a}$ is a smooth $\mco$-point in $\Spec A$.
This settles the first part of the claim. 

As to the second, it is not hard to calculate that 
    \[
    \fitt_{n+1}(\hat{\Omega}_{A/\mco}) = (X)^{n-1}A\,,
    \]
    the $(n-1)$st power of the ideal generated by the residue classes of the $x_{ij}$. See, for instance, \cite[Proposition~7.3.4]{Bruns/Herzog:1998a} or \cite[Theorem~3.4.6]{Bruns/Conca/Raicu/Varbaro:2022}, which treat the case of determinantal rings over a field.   Thus we conclude that, as ideals in $A$, one has
    \[
    \fitt_{n+1}(\fp/\fp^2) = \lambda_{\bos{a}}(X)^{n-1}\,.
    \]
    This justifies the claim about the torsion part of the cotangent module of $\lambda_{\bos{a}}$. Observe that the ideal above is nonzero if and only if $\bos{a}\ne 0$. By ~\cite[Theorem~2.6]{Iyengar/Khare/Manning/Urban:2024}, the ring $A$ is regular if, and only if, $\fp/\fp^2$ is torsion-free, so the equality above gives another proof of the claim that $\lambda_{\bos{a}}$ is smooth if and only if $\bos{a}\ne 0$.
\end{proof}

In what follows we we treat congruence ideals as ideals in $\mco$, following the discussion after \eqref{eq:eta-defn}.
The number $\delta_{\lambda_{\bos a}}(A)$ is the Wiles defect of $A$ at $\lambda_{\bos{a}}$; see \eqref{eq:wiles-defect}.

\begin{proposition}
\label{pr:det-congruence}
Fix a smooth $\mco$-valued point $\lambda_{\bos{a}}\colon A\to \mco$ on $\Spec A$. One has
\[
\eta_{\lambda_{\bos a}}(A) = (\varpi^{w_{\bos a}}) \quad\text{and}\quad \delta_{\lambda_{\bos a}}(A) = (n-2) w_{\bos a}\,.
\]
\end{proposition}

\begin{proof}
We use Theorem~\ref{th:gorsum} and to that end start by finding a surjective map $\alpha\colon C\to A$ in $\cato_\mco(n+1)$, with $C$ a complete intersection ring. Finding such a $C$ is tantamount to finding a $P$-regular sequence contained in the ideal $I_2(X)$ defining $A$; this is not difficult. The tricky part is find a sequence such that the given smooth $\mco$-valued point on $\Spec A$ is also smooth as a point on the complete intersection.

We simplify this search by exploiting the $\GL_2(\mco)\times \GL_n(\mco)$ action on $A$ induced by the natural action on $X$. Since $\mco$ is a discrete valuation ring, this action is transitive on the collection of $\mco$-valued points $\lambda_{\bos{a}}$ with a fixed $w_{\bos{a}}$. It is then not hard to see that congruence ideal $\lambda_{\bos{a}}$ is independent of $(\bos{a})$, as long as $w_{\bos{a}}$ is fixed. We may thus assume that there exist an integer $s\ge 1$ such that
\begin{equation}
\label{eq:good-point}
\bos{a}  \colonequals 
\begin{bmatrix} 
\varpi^s&\dots&\varpi^s \\
0&\dots&0
\end{bmatrix}
\end{equation}
The relevance of this choice will become clear shortly.

In what follows we write $[i\ j]$ for the size 2 minor of $X$ obtained by taking $i$'th and $j$'th column of the matrix $X$; that is
\[
[i\ j] \colonequals
\begin{vmatrix}
    x_{1i} & x_{1j} \\
    x_{2i} & x_{2j}
\end{vmatrix} = x_{1i}x_{2j} - x_{2i}x_{1j}\,.
\]
Consider the minors $f_i\colonequals [i, i+1]$ for $1\le i\le n-1$ and set
\[
C\colonequals \frac{P}{(f_1,\dots,f_{n-1})}\,.
\]
The result below yields that $C$ is a complete intersection of codimension $n-1$; equivalently, of Krull dimension $n+2$.

\begin{lemma}
\label{le:C-is-ci}    
The sequence $\bos{f}\colonequals f_1,\dots,f_{n-1}$ is regular in $P$.
\end{lemma}

\begin{proof}
The key point is that in the lexicographic monomial order on $\mco[X]$, where $x_{ij}<x_{st}$ if $i<s$ or $i=s$ and $j<t$ the initial monomials $\mathrm{in}(f_1),\dots,\mathrm{in}(f_{n-1})$ are $x_{11}x_{22}, \dots, x_{1\, n-1}x_{2\, n}$, and this is evidently a regular sequence in $\mco[X]$. This implies that $f_1,\dots,f_{n-1}$ are themselves a regular sequence; see, for instance, \cite[Proposition~1.2.12]{Bruns/Conca/Raicu/Varbaro:2022}.
\end{proof}

Evidently, the sequence $\bos{f}$ is contained in the ideal $I_2(X)$ so the quotient map $P\to A$ factors through $C$ to yield a surjective map $\alpha\colon C\to A$. Now, for an arbitrary smooth point $\lambda_\bos{a}\colon A\to \mco$ in $\Spec A$ the corresponding point $\lambda_{\bos{a}} \alpha$ on $\Spec C$ may not be smooth. However, we have already explained that it suffices to consider an $\mco$-valued point of the form \eqref{eq:good-point}. For such an $(\bos{a})$ a direct calculation yields that the cotangent module $\fp/\fp^2$ of $\lambda_{\bos{a}}\alpha$ is the cokernel of the following $n\times (n-1)$ matrix:
\[
\begin{bmatrix}
    \varpi^s & 0 & 0 & \cdots  & 0 & 0 & 0\\
    \varpi^s & \varpi^s & 0 &\cdots & 0 & 0 & 0 \\
    0 & \varpi^s & \varpi^s &\cdots & 0 & 0 & 0\\
    \vdots & \vdots & \vdots &\cdots &\vdots&\vdots &\vdots \\
    0 & 0 & 0 &\cdots & 0 & \varpi^s &\varpi^s \\
    0 & 0 & 0 &\cdots & 0  &0 & \varpi^s 
\end{bmatrix}
\]
Here we are exploiting the fact that since $a_{2i}=0$, one has
\[
\left.\frac{\partial f_j}{\partial x_{1i}}\right|_{\bos{a}} = 0 \quad\text{for all $i,j$.}
\]
See also Remark~\ref{re:determinant-case}. Therefore one gets that
\begin{equation}
\label{eq:fitt-C}
\fitt_{n+1}(\fp/\fp^2)=(\varpi^{(n-1)s})=(\varpi^{(n-1)w_{\bos{a}}})\,.
\end{equation}
In particular, $\lambda_{\bos{a}}\alpha$ is a smooth point on $\Spec C$.

The next step is to compute the ideal $C[\ker \alpha]$, which is the image in $C$ of the ideal $((\bos{f}): I_2(X))$ in $P$.
We claim that 
\[
((\bos{f}): I_2(X)) \equiv \prod_{i=2}^{n-1}(x_{1i},x_{2i}) \mod (\bos{f}) \,.
\]
Indeed, for any integer $1\le i\le n-2$ one has that
\[
\begin{vmatrix}
    x_{1,i} & x_{1,i+1} & x_{1,i+2} \\
    x_{1,i} & x_{1,i+1} & x_{1,i+2} \\
    x_{2,i} & x_{2,i+1} & x_{2,i+2}
\end{vmatrix}
=0
\]
and this gives us relations
\[
x_{1,i+1}[i, i+2] = x_{1,i+2}[i,i+1] +  x_{1,i}[i+1,i+2]\,.
\]
By symmetry, one gets also relations
\[
x_{2,i+1}[i, i+2] = x_{2,i+2}[i,i+1] +  x_{2,i}[i+1,i+2]\,.
\]
Using these it is straight forward to verify that
\[
\prod_{i=2}^{n-1}(x_{1i},x_{2i})\subseteq ((\bos{f}): I_2(X))\,.
\]
It suffices to verify that equality holds modulo $\varpi$, so then one can work in the ring $k[X]/I_2(X)$.
One can then exploit the fact that both ideals in questions are, after twisting by $-2$, isomorphic to the graded canonical module of the ring $k[X]/I_2(X)$; for the ideal on the left, this follows from the discussion in \cite[Theorem~6.7.13]{Bruns/Conca/Raicu/Varbaro:2022}, while for the ideal on the right, it holds because $C/\varpi C$ is Gorenstein. It then follows by comparing Hilbert series that the inclusion is an equality. In conclusion
\[
\lambda_{\bos{a}}(C[\Ker\alpha]) = (\varpi^{(n-2)w_{\bos{a}}})\,.
\]
Therefore applying Theorem~\ref{th:gorsum} we get the second equality below
\begin{align*}
(\varpi^{(n-1)w_{\bos{a}}}) = \eta_{\lambda_{\bos a}\alpha}(C) 
= (\varpi^{(n-2)w_{\bos{a}}}) \eta_{\lambda_{\bos a}}(A)
\end{align*}
The first equality is from \eqref{eq:fitt-C} and holds because $C$ is a complete intersection. We conclude that $\eta_{\lambda_{\bos a}}(A)=(\varpi^{w_{\bos{a}}})$, as claimed.
\end{proof}

\begin{remark}
\label{re:determinant-case}
In the notation of the proof of Proposition~\ref{pr:det-congruence}, by a direct calculation we get that
\[
\fitt_{n+1}(\hat{\Omega}_C/\mco) = (\prod_{i=2}^{n-2}x_{1i}\cdot X+\prod_{i=2}^{n-2}x_{2i}\cdot X)C
\]
This is subsumed in the computations described in a manuscript in preparation by Khan and Maithani~\cite{Khan/Maithani:2026} that deals with congruence ideals of more general determinantal rings.
\end{remark}

\section{Congruence ideals of local deformation rings}
\label{se:BKM-revisited}
In this section we calculate the congruence ideals of local (Steinberg and unipotent) deformation rings   that are considered in \cite{Bockle/Khare/Manning:2024}, at  augmentations of local (unipotent)  deformation rings that are pulled back from  either its Steinberg or unramified quotient; the latter case was not considered in \cite{Bockle/Khare/Manning:2024}. Using our computations in the former (Steinberg) case, in Section~\ref{se:NT} we refine the results in \cite{Bockle/Khare/Manning:2024} about computing the defect of Hecke algebras at augmentations induced by newforms  $f \in S_2(\Gamma_0(N))$  (for  squarefree level N, that is prime to $p$) by a new method that relies on the study of congruence ideals. 

When \cite{Bockle/Khare/Manning:2024} was written there was no notion of a congruence ideal for points of  higher codimension, and only the Wiles defect of the ring is computed, in effect by a reduction to the classical case where $c=0$. The approach here is more direct.

Let $p$ be a prime number and $\mco$ the ring of integers in a finite extension of $\mbb{Q}_p$. Let $\varpi$ denote the uniformizer of $\mco$ as before, $k$ its residue field, and $E$ its field of fractions. Let $F_v$ be a local field  (a finite extension of $\Q_\ell$ for $\ell \neq p$)  and $q_v$ the order of the residue field of $F_v$.   (The notation anticipates being in a global situation where $F$ is a number field and $v$ a place of $F$ not above $p$.) Fix a representation $\rhobar\colon G_{F_v}\to \GL_2(k)$, where $G_{F_v}$ is the absolute Galois group of $F_v$, with $\det(\rhobar)=\bar \epsilon_p$ the mod $p$ cyclotomic character and the conductor $N(\rhobar)|q_v$. We  assume  that  if $\rhobar$ is unramified, then $\trace(\rhobar(\Frob_v))=\pm (q_v+1)$.

Let $R_v^\square$ be  the universal framed deformation ring of $\rhobar$. We write $R^{\st}_v$ for the Steinberg quotient of $R_v^\square$, defined in \cite[pp. 34]{Bockle/Khare/Manning:2024}. The unipotent and unramified quotients are denoted $R^{\uni}_v$ and $R_v^{\unr}$, respectively.  When $q_v$ is $-1$ mod $p$,  the ring $R_v^{\st}$ is the $R^{\st(\beta_v)}$ of op. cit.  with $\beta_v=1$ or $-1$. These rings are reduced and flat over $\mco$, and thus characterized by their $\overline{\QQ}_p$ points. In the case of  $R_v^{\st}$ the points in $\Spec R_v^{\st}(\overline{\QQ}_p)$ correspond to representations of $G_v$ that are of the form 
\[
 \begin{bmatrix} 
 \epsilon_p\chi & * \\ 0 & \chi 
 \end{bmatrix}
 \] 
 for $\chi$ an unramified quadratic character (that is determined by $\rhobar$ when $q_v \not\equiv -1 \pmod{p}$, and with $\chi(\Frob_v)=\beta_v$ when $q_v \equiv -1 \pmod{p}$) and $\epsilon_p$ the $p$-adic cyclotomic character. We implicitly choose $\beta_v$  so that the augmentation  $\lambda$ below factors through $R^{\st}$. We also consider the modified deformation ring $R_v^{\phuni}$,  which is an extension of $R_v^{\uni}$.  

We recall presentations of these rings assuming that $q_v$ is 1 mod $p$, and $\rhobar$ is scalar (and for simplicity, the identity matrix). 

\subsection*{The deformation rings}
Set $R\colonequals \mco\pos{a,b,c,\alpha,\beta,\gamma,X}$. All three deformation rings---the unipotent; the Steinberg; and the unramified--can be realized as a quotient of $R$. Let $I^{\uni}$, $I^{\st}$, and $I^{\unr}$ denote the corresponding ideals; these are described explicitly in what follows. The ideal $I^{\uni}$ is generated by elements
\begin{gather*}
    X\alpha,\   X\beta,\  X\gamma,\ \alpha^2+\beta\gamma,\  b\alpha - a\beta,\  a\alpha+b\gamma\, \\
    c\beta  - b\gamma + (q_v-1)\alpha,\ c\alpha - (q_v+a-1)\gamma,\ a^2 +bc + aX + (q_v-1)a + q_vX\,. 
\end{gather*}

The universal framed deformation  corresponding to $R_v^{\uni}$ factors through the tame quotient $G^t_q$ of $G_{F_v}$, and if $\sigma$ is a  lift of Frobenius, and  $\tau$ is a topological generator of the inertia subgroup of $G^t_q$, we have that $\sigma\tau\sigma^{-1}=\tau^q$,  and the deformation is given by $\sigma\mapsto A$ and $\tau\mapsto I+N$, where
\begin{equation}
\label{eq:AN}
A\colonequals \begin{bmatrix} 
q_v+a & b \\ 
c & 1-a-X 
\end{bmatrix} 
\quad \text{and} \quad
N\colonequals 
\begin{bmatrix} 
\alpha & \beta \\ 
\gamma & -\alpha 
\end{bmatrix}
\end{equation}
The salient aspects of the structure of $R_v^{\uni}$, for the purposes of computing congruence ideals, are recorded in the result below, extracted from \cite[Lemma~5.4]{Bockle/Khare/Manning:2024}.
\begin{lemma}
\label{le:uni-properties}
With notation as above, one has
\begin{enumerate}[\quad\rm(1)]
\item 
$I^{\st}= I^{\uni}+(X)$ and $I^{\unr}= I^{\uni}+(\alpha,\beta,\gamma)$; these are prime ideals.
\item 
 $I^{\uni}=I^{\st}\cap I^{\unr}$.
\item 
 $R_v^{\uni}$ is a Gorenstein ring of Krull dimension $4$. 
\item 
 $R^{\st}$ is a Cohen-Macaulay domain of Krull dimension $4$.
\item 
 $R_v^{\unr}$ is a regular ring of Krull dimension $4$.\qed
\end{enumerate}
\end{lemma}

It follows from the discussion above that, as subsets of $\Spec R$, there is an equality
\[
\Spec R_v^{\uni}=\Spec R_v^{\st}\cup \Spec R_v^{\unr}\,.
\]
Since $\Spec R_v^{\st}$ and $\Spec R_v^{\unr}$ are components of $\Spec R_v^{\uni}$ a smooth point must lie in one or the other of the components, but not both.  One has the following, Mayer-Vietoris, exact sequence of $R^{\uni}$-modules
\begin{equation}
    \label{eq:uni-Mayer-Vietoris}
0\longrightarrow R_v^{\uni}\longrightarrow R_v^{\st}\oplus R_v^{\unr}\longrightarrow 
        \frac{\mco\pos{a,b,c}}{(a(a+q_v-1)+bc)}\longrightarrow 0\,.
\end{equation}
This sequence is helpful in computing the congruence ideals of $R_v^{\uni}$ in terms of those of $R_v^{\st}$ and $R_v^{\unr}$; compare Corollary~\ref{co:fiber-product}. 

Since $R_v^{\st} = R/I^{\st}$ with $I^{\st} = I^{\uni}+(X)$, we can also write it  as the quotient  of $\mco\pos{a,b,c,\alpha,\beta,\gamma}$. The defining ideal is  generated by the size $2$ minors of the matrix
\begin{equation}
\label{eq:St-relations}
    \begin{bmatrix}
   \alpha & \beta & q_v-1+a & b \\
   \gamma & -\alpha & c & -a 
\end{bmatrix}
\end{equation}

Since $R_v^{\unr}=R/I^{\unr}$ with $I^{\unr}=I^{\uni}+(\alpha, \beta, \gamma)$ one gets that
 \begin{equation}
 \label{eq:unr-relations}
R_v^{\unr}   = \frac{\mco\pos{a,b,c,X}}{(a-(q_v+a)(a+X)-bc)}\cong \mco\pos{a,b,c}\,.     
 \end{equation}

\subsection*{Modified deformation rings}
The deformation ring corresponding to modified unipotent deformations, where one keeps track of an eigenvalue of  Frobenius, is
\begin{equation}
\label{eq:mod-Uni}
R_v^{\phuni} \colonequals  \frac{R_v^{\uni}[y]}{(y^2-(q_v+1-X)y+q_v,(y-1)I^{\unr})}\,.
\end{equation}
The quadratic polynomial in sight is the characteristic polynomial of the matrix $A$ from \eqref{eq:AN}. This can be viewed as an $R_v^{\uni}$-algebra. It is a local ring because, modulo the maximal ideal $\fm$ of $R_v^{\uni}$, it is the local ring $k[y]/((y-1)^2)$.

By the same token, the modified unramified deformation ring
\begin{equation}
\label{eq:mod-unramified}
R_v^{\phunr}\colonequals \frac{R_v^{\unr}[y]}{(y^2-(q_v+1-X)y+q_v)}\,,
\end{equation}
is local. In the ring $R[y]=\mco\pos{a,b,c,\alpha,\beta,\gamma,X}[y]$ consider the ideals
\begin{align}
\label{eq:phuni-relations}
\begin{aligned}
 J^{\phuni} &\colonequals I^{\uni} R[y]+ (y^2-(q_v+1-X)y+q_v) + (y-1)I^{\unr}     \\
J^{\st}    &\colonequals J^{\phuni} + (y-1) + (X) R[y] = I^{\uni}R[y]+(y-1) \\
J^{\phunr} &\colonequals J^{\phuni} + (\alpha,\beta,\gamma) R[y]\,.
\end{aligned}
\end{align}
The notation is justified by the following, easy to verify, isomorphisms.

\begin{lemma}
\label{le:phuni-components}
With the notation above, one has isomorphisms of rings:
\[
\frac{R[y]}{J^{\phuni}} \cong R_v^{\phuni}\,,\quad 
\frac{R[y]}{J^{\st}} \cong R_v^{\st}\,, \quad\text{and}\quad
\frac{R[y]}{J^{\phunr}} \cong R_v^{\phunr}\,.\qed
\]
\end{lemma}
Since $R_v^{\st}$ and $R_v^{\phunr}$ are domains, this explains the first part of the following result. It is an analogue of Lemma~\ref{le:uni-properties} for the framed unipotent deformation ring, and it is extracted from the proof of \cite[Lemma~5.3]{Bockle/Khare/Manning:2024}; our notation and presentation differ from the one in \emph{op.~cit.}

\begin{lemma}
\label{le:phuni-properties}
With notation as above, one has
\begin{enumerate}[\quad\rm(1)]
\item 
$J^{\st}$ and $J^{\phunr}$ are prime ideals.
\item 
 $J^{\phuni}=J^{\st}\cap J^{\phunr}$.
\item 
 $R_v^{\phuni}$ is a Gorenstein ring of Krull dimension $4$. 
\item 
 $R_v^{\phunr}$ is a complete intersection ring of Krull dimension $4$.\qed
\end{enumerate}
\end{lemma}
Thus we get the Mayer-Vietoris sequence
\begin{equation}
    \label{eq:phuni-Mayer-Vietoris}
0\longrightarrow R_v^{\phuni}\longrightarrow R_v^{\st}\oplus R_v^{\phunr}\longrightarrow 
        \frac{\mco\pos{a,b,c}}{(a(a+q_v-1)+bc)}\longrightarrow 0\,.
\end{equation}
Akin to the situation with $\Spec R_v^{\uni}$, a smooth point of $\Spec R_v^{\phuni}$ lies in
one of its components, $\Spec R_v^{\st}$ or $\Spec R_v^{\phunr}$, but not both. 

\subsection*{Steinberg augmentations}
Let $\lambda\colon R_v^{\st} \to\mco$ be an augmentation such that the induced representation $\rho_\lambda\colon G_{F_v} \to \GL_2(\mco)$ lifts $\rhobar$ and is of the form
 \[
 \begin{bmatrix} 
 \epsilon_p\chi & * \\ 0 & \chi 
 \end{bmatrix}
 \] 
 for $\chi$ an unramified quadratic character and $\epsilon_p$ the $p$-adic cyclotomic character. Let $n_v$ (respectively, $m_v$) be the largest integer $N$  such that $\rho_\lambda$  (respectively, $\rho_\lambda|_{I_q}$) mod $\varpi^N$ is scalar. We assume  $\rho_\lambda$ is ramified, so $\lambda$ does not factor through $R_v^{\unr}$.

\begin{lemma}
Any Steinberg augmentation $\lambda\colon R_v^{\uni}\to \mco$ extends uniquely to an augmentation 
$ R_v^{\phuni}\to \mco$, and this is determined by $y\mapsto 1$.
\end{lemma}

\begin{proof}
Set $\fp\colonequals \Ker\lambda$ and $E$ the residue field at $\fp$, namely, the fraction field of $\mco$. The prime ideals in $R_v^{\phuni}$ lying over $\fp$ are the prime ideals in spectrum of the ring
\begin{align*}
E\otimes_{R_v^{\uni}}R_v^{\phuni} 
            &\cong \frac{E[y]}{(y^2-(q_v+1-X)y+q_v,(y-1)I^{\unr})} \\
            &\cong \frac{E[y]}{(y^2-(q_v+1) + q_v,y-1)}\\
            &\cong \frac{E[y]}{(y-1)}\\
            &\cong E
\end{align*}
viewed as a subset of $\Spec R_v^{\phuni}$. The second isomorphism above holds because $\fp$ contains $X$, and $I^{\unr}$ is not contained in $\fp$. The rest of the isomorphisms are clear. This justifies the claim.
\end{proof}

One interpretation of the result above is that Steinberg augmentations of $R_v^{\uni}$ are induced, via restriction, by Steinberg augmentations of $R_v^{\phuni}$. This is the prespective we take in the next result. It recovers the calculations in \cite{Bockle/Khare/Manning:2021b, Bockle/Khare/Manning:2024}. 

\begin{theorem}
\label{thm:uni-steinberg}
When $\lambda\colon R_v^{\phuni} \to \mco$ is a Steinberg augmentation, one has
\begin{enumerate}[\quad\rm(1)]
    \item $\eta_{\lambda}(R_v^{\st}) = (\varpi^{n_v})$ 
            and $\fitt_0 \Phi_{\lambda}(R_v^{\st})= (\varpi^{3n_v})$,
    \item  $\eta_{\lambda}(R_v^{\uni}) = (\varpi^{m_v+n_v})$ 
            and $\fitt_0 \Phi_{\lambda}(R_v^{\uni})= (\varpi^{m_v+2n_v})$,
    \item $\eta_{\lambda}(R_v^{\phuni}) = (\varpi^{m_v+n_v})$ 
            and $\fitt_0\Phi_{\lambda}(R_v^{\phuni})= (\varpi^{m_v+3n_v})$. 
\end{enumerate}
\end{theorem}

\begin{proof}
 We only justify the claims when $q_v$ is $1$ mod $p$ and $\rhobar$ is scalar; outside this case $n_v=0$ (and then for instance $R^{\st}$ is smooth), and the proofs are easier. One may easily reduce to the case $\rhobar$ is the trivial  representation.

(1) In terms of the description in~\eqref{eq:St-relations}, the augmentation  $\lambda\colon R^{\st}_v\to \mco$   satisfies
\begin{gather*}
\lambda(a)=\lambda(c)=\lambda(\alpha)=\lambda(\gamma)=0  \\
\lambda(b)=s,\ \lambda(\beta)=t \quad \text{with $s,t \in (\varpi)$ and $t \neq 0$.}
\end{gather*}
For instance the  fact that $t \neq 0$ follows from our assumption that $\rho_\lambda$ is ramified; more precisely $\ord(t)=m_v$.  
Thus $\lambda$ is given by specializing the matrix \eqref{eq:St-relations} to
\[
\begin{bmatrix}
   0 & t & q_v-1 & s \\
   0 & 0 & 0 & 0
\end{bmatrix}
\]
In particular, $n_v= \min\{ \ord(q_v-1),\ord(s),\ord(t)\}$.

It is clear from the description of $R_v^{\st}$ as the ring defined by the size $2$ minors of the matrix \eqref{eq:St-relations} that  it is isomorphic to the quotient of the determinantal ring $A$ consider in Section~\ref{se:determinantal-rings}, with $n=4$, modulo the relations 
\[
x_{11} +x_{22}=0\quad \text{and}\quad x_{13}+x_{24}=q_v-1\,.    
\]
It is clear that these relations form a regular sequence on $A$.  With $\fp$ denoting the kernel of the composition $\alpha\colon A\to R_v^{\st}\to\mco$, it is easy to verify that the order ideals of the residue classes of the elements  above in $\fp/\fp^2$ is $\mco$. So from Proposition~\ref{pr:det-congruence}, Theorem~\ref{th:deformation}, and Lemma~\ref{le:deformation-cot} we deduce that 
\begin{align*}
\eta_{\lambda}(R_v^{\st}) & = \eta_{\lambda\alpha}(A) = (\varpi^{n_v})    \\
\fitt_0\Phi_{\lambda}(R_v^{\st}) &= \fitt_0 \Phi_{\lambda\alpha}(A) = (\varpi^{3n_v}) \\
\delta_{\lambda}(R_v^{\st}) & = \delta_{\lambda\alpha}(A) = 2n_v\,. 
\end{align*}
This justifies (1).

(2) Next we compute the congruence ideal for the pullback of the Steinberg augmentation to  $R_v^{\uni}$. In particular, $\lambda(X)=0$ also holds. As $\rho_\lambda$ is ramified, $\lambda$ is a smooth point of $R_v^{\uni}$. We apply Corollary~\ref{co:fiber-product} to the ring $R_v^{\uni}$ and ideals $I^{\st}R_v^{\uni}$ and $I^{\unr}R_v^{\uni}$. It follows from 
Lemma~\ref{le:uni-properties} and the Mayer-Vietoris sequence~\ref{eq:uni-Mayer-Vietoris} that the hypothesis of Corollary~\ref{co:fiber-product} are satisfied so, building on the computation of the congruence ideal of $\lambda\colon R_v^{\st} \to \mco$ computed in (1), we get the first equality below
\[
\eta_{\lambda}(R^{\uni}_v) 
    = \lambda(I^{\unr})\cdot \eta_{\lambda}(R^{\st}_v)
    = \lambda(I^{\unr})\cdot (\varpi^{n_v}) 
    = (\varpi^{n_v+\ord(t)}) 
    = (\varpi^{m_v +n_v}).
\]
The second equality is from the already established part (1). The torsion part of the cotangent module of $\lambda$ has length
\[
\length_{\mco}\Phi_{\lambda}(R_v^{\uni})=m_v+2n_v.
\]
This follows from the  computation of $\Phi_{\lambda}(R_v^{\st})$,  as in \cite{Bockle/Khare/Manning:2021b} it is proved that 
\[
\length_{\mco}\Phi_{\lambda}(R_v^{\uni}) - \length_{\mco}\Phi_{\lambda}(R_v^{\st})=m_v-n_v\,.
\]
This justifies the equalities in (2).

(3) 
 As in the proof of (2), given Lemma~\ref{le:phuni-properties} and the exact sequence \eqref{eq:phuni-Mayer-Vietoris}, the claim about the congruence ideal of $R_v^{\phuni}$ can be deduced from (1). The claim about the torsion in the cotangent space of $\lambda$ is a direct computation.
\end{proof}

Let $\iota\colon R_v^{\uni}\to R_v^{\phuni}$ be the inclusion and $\lambda\colon R_v^{\phuni}\to\mco$ a Steinberg augmentation.
With $\fp=\Ker \lambda\iota$ and $\fq=\Ker \lambda$, the map $\iota$ induces a map of $\mco$-modules
\[
\Phi_\iota\colon \fp/\fp^2\longrightarrow \fq/\fq^2\,.
\]
The result below is needed for some applications in Section~\ref{se:NT}. It can be deduced as a consequence of Theorem~\ref{thm:uni-steinberg}, but we give a direct proof.

\begin{lemma}
\label{le:transfer-factor}
The dual map $\Phi^*_\iota$, and hence also $\wedge^c\Phi^*_\iota$, is an isomorphism.
\end{lemma}

\begin{proof}
We have only to verify that $\Phi^*_\iota$ is surjective, since it is a map of free $\mco$-modules of the same rank, namely, $3$. With $\fp,\fq$ as above, $\Phi^*_\iota$ fits in the (cohomological version of the) Jacobi-Zariski exact sequence of $\mco$-modules
\[
\Hom_{\mco}(\fq/\fq^2,\mco) \xrightarrow{\ \Phi^*_\iota\ } \Hom_{\mco}(\fp/\fp^2,\mco) \longrightarrow \mathrm{D}^1(R_v^{\phuni}/R_v^{\uni},\mco)
\]
arising from maps $R_v^{\uni}\xrightarrow{\iota} R_v^{\phuni}\xrightarrow{\lambda} \mco$, where the term on the right is the first Andr\'e-Quillen cohomology module of $R^{\phuni}$ over $R_v^{\uni}$ with coefficients in $\mco$, viewed as an $R_v^{\phuni}$-module via $\iota$. It thus suffices to verify that this module is $0$.

Consider the factorization $R_v^{\uni}\to R_v^{\uni}[y]\to R_v^{\phuni}$ of the map $\iota$, where the second map is the quotient modulo the ideal
\[
I\colonequals (y^2-(q_v+X-1)y+ q_v) +  (y-1)I^{\unr}
\]
This induces a map
\[
\frac{I}{I^2}\otimes_{R_v^{\phuni}}\mco \cong  \frac{I}{I^2+\fq I} \longrightarrow \mco dy
                        \cong \Omega\otimes_{R_v^{\phuni}}\mco
\]
where $\Omega$ is the module of K\"ahler differentials of $R_v^{\phuni}$ over $R_v^{\uni}$. It is straightforward to check that the torsion-free quotient of $I/(I^2+\fq I)$ is of rank $1$ and generated by the class of $y^2-(q_v+X-1)y+ q_v$, and also that the map above is induced by
\begin{gather*}
\tfree{\big(\frac{I}{I^2+\fq I}\big)} \longrightarrow \mco dy \\
[y^2-(q_v+X-1)y+ q_v]\mapsto (q_v-1)dy
\end{gather*}
and hence that it is an isomorphism. It remains to note that the cohomological version of the Jacobi-Zariski sequence with coefficients in $\mco$ arising from the factorization of $\iota$ given above  reads
\[
\Hom_{\mco}(\mco dy,\mco) \longrightarrow  \Hom_{\mco}(\frac{I}{I^2+\fq I},\mco) \longrightarrow \mathrm{D}^1(R_v^{\phuni}/R_v^{\uni},\mco) \longrightarrow 0
\]
It follows from the discussion above that the map on the left is an isomorphism, and hence that $\mathrm{D}^1(R_v^{\phuni}/R_v^{\uni},\mco)= 0$, as desired.
\end{proof}

\subsection*{Unramified augmentations}
\label{ssec:unramified-augmentations}
Assume that $\rhobar$ is unramified and $\lambda\colon R^{\unr} \to \mco$ is an augmentation such that $\rho_\lambda\colon G_{F_v} \to \GL_2(\mco)$ lifts $\rhobar$.   Recall our standing assumption $\trace(\rhobar(\Frob_v))=\pm (q_v+1)$, and by twisting we may reduce to the case  $\trace(\rhobar(\Frob_v))=(q_v+1)$ The universal representation sends $\Frob_v$ to the matrix
\[
A\colonequals
\begin{bmatrix} 
q_v+a & b \\ 
c & 1-a-X 
\end{bmatrix} 
\]
We may assume that $\rho_\lambda$ is upper triangular and $\rho_\lambda(\Frob_v)$ is  
\[
\begin{bmatrix} \theta+ \hbar & \varepsilon \\ 0 & \theta-\hbar \end{bmatrix}
\]
By pull back, $\lambda$ also  induces an augmentation of  $R^{\uni}$ such that 
\[
\lambda(\alpha) = \lambda(\beta)=\lambda(\gamma)=0\,.
\]
We  define $r_v$  to be the largest power  $N$ of $\varpi$ such that 
\[
\rho_\lambda(\Frob_v) \equiv \begin{bmatrix} 
q_v & 0 \\ 
0 & 1 
\end{bmatrix} \mod \varpi^N
\quad\text{and}\quad q_v\equiv 1 \mod \varpi^N\,.
\]

As we assume that $\det \rho_\lambda=\epsilon_p$, one  may rephrase this as saying that 
$r_v$ is  the largest power  $N$ of $\varpi$ such that 
\[
\rho_\lambda(\Frob_v) \equiv \begin{bmatrix} 
1 & 0 \\ 
0 & 1 
\end{bmatrix} \mod \varpi^N\,.
\]

Equivalently
\begin{align*}
r_v  &=\min\{\ord(q_v-1), \ord \lambda(a),\ord \lambda(b),\ord\lambda(c),\ord\lambda(X)  \}    \\
     &= \min\{\ord(q_v-1), \ord \hbar, \ord(\theta-1),\ord \varepsilon\}\,.
\end{align*}
In particular, $r_v$ is bounded above by $\ord(q_v-1)<\infty$.

\begin{lemma}
For any unramified augmentation $\lambda\colon R_v^{\uni}\to \mco$ and a choice of an eigenvalue for $\rho_\lambda(\Frob_v)$ there is a unique extension of $\lambda$ to $R_v^{\phuni}$, sending $y$ to that eigenvalue.
\end{lemma}

\begin{proof}
Recall that $I^{\unr}=I^{\uni}+(\alpha,\beta,\gamma)$, by Lemma~\ref{le:uni-properties}. This ideal is contained in $\Ker\lambda$ and one has
\[
\frac{R_v^{\phuni}}{I^{\unr}R_v^{\phuni}} 
    \cong \frac{R[y]}{(J^{\phuni}+(\alpha,\beta,\gamma))}\cong R_v^{\phunr}\,,
\]
where the last isomorphism is from Lemma~\ref{le:phuni-components}. This justifies the claim.
\end{proof}

The analog of Theorem~\ref{thm:summary-unr} for unramified augmentations is:

\begin{theorem}
\label{thm:summary-unr}
Assume  $\trace(\rhobar)(\Frob_v)=  (q_v+1)$.  For an unramified augmentation $\lambda\colon R_v^{\uni}\to \mco$ the following equalities hold:
\[
\eta_{\lambda}(R_v^{\uni})  = (\lambda(X)) \quad \text{and}\quad 
\fitt_0 \Phi_{\lambda}(R_v^{\uni}) = (\varpi^{r_v} \lambda(X))\,.
\]
Moreover, for any augmentation $\lambda\colon R_v^{\phunr}\to \mco$ one has
\begin{align*}
&\eta_{\lambda}(R_v^{\phunr})  = (\hbar,\varepsilon) = \fitt_0 \Phi_{\lambda}(R_v^{\phunr})\\
&\eta_{\lambda}(R_v^{\phuni}) = (\hbar,\varepsilon)(\lambda(X)) \text{ and }
\fitt_0 \Phi_{\lambda}(R_v^{\phuni}) = (\hbar,\varepsilon) (\varpi^{r_v} \lambda(X))\,.
\end{align*}
\end{theorem}

\begin{proof}
Using Corollary~\ref{co:fiber-product} and  the fact that $\cmod_{\lambda}(R_v^{\unr})=0$ because the ring $R_v^{\unr}$ is regular---see Theorem~\ref{th:regular-eta}---we get that
\[
\eta_{\lambda}(R_v^{\uni}) 
    = \lambda(I^{\st}) \cdot \eta_{\lambda}(R_v^{\unr}) 
    =  \lambda(I^{\st})
    = (\varpi^{\ord \lambda(X)})\,.
\]
 In particular, $\lambda$ is a smooth point on $\Spec R_v^{\uni}$ if and only if $\lambda(X)\ne 0$ (that is to say, if $\lambda$ does not factor through $R^{\st}$). When this holds, we claim  the length of the torsion part of the cotangent module is given by
\begin{align*}
 \ord \lambda(X) + \min\{\ord(q_v-1), \ord \lambda(a),\ord \lambda(b),\ord\lambda(c),\ord\lambda(X)  \}  
\end{align*}   
which justifies the first part of the statement.

To get this, we start with the fact that the cotangent module is the cokernel of the matrix obtained by applying $\lambda$ to the matrix
\[
\begin{bmatrix}
 -(q_v-1 + 2a +X) & 0 & 0 & 0 & 0 & 0 & 0 & 0 \\
  - c & 0 & 0 & 0 & 0 & 0 & 0 & 0 \\
  - b & 0 & 0 & 0 & 0 & 0 & 0 & 0 \\
  - (q_v+a) & 0 & 0 & 0 & 0 & 0 & 0 & 0 \\
  0 & c & a+X & q_v-1+a & -b & X & 0 & 0 \\
  0 & 0 & 0 & c & a+X & 0 & X & 0 \\
  0 & -(q_v-1+a) & b & 0 & 0 & 0 & 0 & X
\end{bmatrix}
\]
The length of $\Phi_{\lambda}(R_v^{\uni})$ is thus the order of the ideal of minors of size 4 of this matrix. Since $(q_v+\lambda(a))$ is invertible in $\mco$, this is the same as the ideal of minors of size $3$ of the matrix
\[
\begin{bmatrix}
c & a+X & q_v-1+a & -b & X & 0 & 0 \\
  0 & 0 & c & a+X & 0 & X & 0 \\
 -(q_v-1+a) & b & 0 & 0 & 0 & 0 & X 
\end{bmatrix}
\]
A Magma computation, treating $q_v-1,a,b,c$ as indeterminates, yields that this ideal is generated by the images of the elements
\begin{gather*}
(q_v-1)^2X, (q_v-1)aX, (q_v-1)bX, (q_v-1)cX, (q_v-1)X^2 \\    
a^2X,abX, acX, aX^2, 
b^2X, bcX, bX^2, 
c^2X,cX^2,
X^3
\end{gather*}
and  the elements
    \begin{align*}
      (q_v-1)^2a + (q_v-1)bc - a^3 - abc  & = (q_v-1-a)X(q_v+a) \\
      (q_v-1)a^2 + a^3 + abc &= -aX(q_v+a)\\
      (q_v-1)ab + a^2b + b^2c&= -bX(q_v+a)\\
      (q_v-1)ac + a^2c + bc^2 &=-cX(q_v+a)
\end{align*}
where the equalities are obtained using the relation $\det A=q_v$. Since the $q_v+a$ element is going to be invertible in $\mco$ (and in fact already in $R_v^{\uni}$) one gets that (already in $R_v^{\uni}$) this ideal is generated by the elements
\[
(q_v-1)X, aX, bX,cX, X^2\,.
\]
One sees by inspection that the order of this ideal is exactly what is claimed.

The ring $R_v^{\phunr}$ is a complete intersection, so we get the first equality below
\[
\eta_{\lambda}(R_v^{\phunr}) 
        = \fitt_{\mco}\Phi_{\lambda}(R_v^{\phunr})
        = (\hbar,\varepsilon)\,.
\]
The second holds because, by a direct computation, the Fitting ideal in question is
\[
(\lambda(y) - (\theta+\hbar)\,, \lambda(y) - (\theta-\hbar)\,, \varepsilon)\,,
\]
and $\lambda(y)$ is either $\theta+\hbar$ or $\theta-\hbar$.

Finally, the augmentation $\lambda\colon R_v^{\phuni}\to\mco$ factors through $R_v^{\phunr}$, given the structure of $R_v^{\phuni}$ given in Lemma~\ref{le:phuni-properties}, we can deduce the claim about the congruence ideal from already established part (2), using Corollary~\ref{co:fiber-product}.

The calculation of the torsion in the cotangent module is left to the reader.
\end{proof}

\begin{remark}
One can define the corresponding deformation rings $R_v^{\st}, R_v^{\uni}, R_v^{\unr}$ and $R_v^{\phuni}$ also when $\trace(\rhobar)(\Frob_v)\neq \pm (q_v+1)$.  It turns out that $R_v^{\uni}=R_v^{\unr}$ and $R_v^{\phuni}=R_v^{\phunr} $. Thus an augmentation $\lambda\colon R_v^{\uni} \to \mco$, or  $\lambda\colon R_v^{\phuni} \to \mco$ is necessarily unramified. One gets with notation as before, for these two augmentations:
\[
\eta_\lambda(R_v^{\uni})=\mco, \Phi_\lambda(R_v^{\uni})=0 \quad\text{and}\quad \eta_\lambda(R_v^{\phuni})=\fitt_\mco \Phi_\lambda(R_v^{\uni})= (\hbar,\varepsilon)\,.    
\] 
One moral conclusion from these computations is that  the congruence ideal has better functorial properties than the torsion of cotangent spaces, or the invariants of \cite{Bockle/Khare/Manning:2024}. Evidence for this is that the growth of the congruence ideals along $R_v^{\uni} \to R_v^{\st}$, or  $R_v^{\uni} \to R_v^{\unr}$,   namely  by $\ord(t)$, the depth of ramification of $\rho_\lambda$,  or $\ord \lambda(X)=\ord (\trace \rho_\lambda(\Frob_v)-(q_v+1))$ ,  depending on whether we consider a Steinberg or unramified augmentation, does not depend on $n_v$. 
\end{remark}

\subsection*{Errata}
We take this opportunity to record some corrections to \cite{Bockle/Khare/Manning:2024}.
\begin{enumerate}[\quad\rm(1)]
\item  On \cite[pp.~2495]{Bockle/Khare/Manning:2024} the matrix  
\[
\begin{bmatrix} \epsilon_p\chi & * \\ 
0 & \chi 
\end{bmatrix}
\]  
should be replaced by  
\[
\begin{bmatrix} \epsilon_p\chi^{-1} & * \\ 
0 & \chi 
\end{bmatrix}
\] 
\item In \cite[Corollary 5.25]{Bockle/Khare/Manning:2024}, the expressions  
\[
(s,t,q-1)^3/((s-t)t)\quad\text{and}\quad   (s,t,q-1)^3/(st)
\]
should be replaced by  
\[
(s,t,q-1)^2/((s-t)t)\quad\text{and}\quad  (s,t,q-1)^2/(st)
\]
respectively.  
\item In \cite[Theorem 5.26]{Bockle/Khare/Manning:2024} the equations should read
\[
D_{1,\lambda}(R_v^{\phuni}) = \frac{5n_v}e,\quad  c_{1,\lambda}(R_v^{\phuni}) = \frac{3n_v}e,  \quad
  \delta_\lambda(R^{\phuni}) = \frac{2n_v}e\,.
  \]
  This impacts the statement also of Theorem 6.5 in \emph{op.~cit.}
 \item    On \cite[pp.~2501]{Bockle/Khare/Manning:2024}, 4 lines before Lemma 5.4, $I_{\mathbb{Z}}^{\unr}=(\alpha,\beta,\gamma)$ should be replaced by $I_{\mathbb{Z}}^{\unr}=(\alpha,\beta,\gamma,r_9^{\uni})$.
\item
On \cite[pp.~2501]{Bockle/Khare/Manning:2024}, in the line before Lemma 5.4 “by $q$” should be replaced by “by $q-1$”.
\item
On \cite[pp.~2501]{Bockle/Khare/Manning:2024}, Lemma 5.4, line 3, the ideal $\mathcal I$ should be replaced by
  $\mathcal I^{\uni}$.
\end{enumerate}

\section{Number theoretic applications.}

\label{se:NT}
In this section we outline  applications of the new commutative algebra constructions and results to number theory. We focus on two:
\begin{itemize}
\item
The Jacquet-Langlands correspondence for mod $p^n$ modular forms of weight $1$. Theorem \ref{th:defect0} is the key commutative algebra input for this application.
\item
Using our study of cogruence ideals of local deformation rings in \S \ref{se:BKM-revisited}, we give a different approach to the results of \cite{Bockle/Khare/Manning:2021b} and \cite{Bockle/Khare/Manning:2024} about {\em Wiles defects} of Hecke algebras acting on the integral Betti cohomology of Shimura curves and modular curves. Theorem \ref{th:deformation} is the key commutative algebra input for this application.
\end{itemize}

We  first give a sketch of Wiles's proof of the Shimura, Taniyama, Weil conjecture for semistable elliptic curves over $\QQ$; the  applications of  our results on congruence ideals,  Theorem \ref{th:deformation} and Theorem \ref{th:defect0}, are inspired by it. 

One formulation of the Shimura, Taniyama, Weil conjecture is in terms of representations of $G_{\mbb Q}$, the absolute Galois group of $\mbb Q$. Given any elliptic curve $E$ defined over $\mbb Q$, and an odd prime number $p$, by considering the action of $G_{\mbb Q}$ on the $p$-power torsion points on $E$, one gets a two-dimensional Galois representation $\rho_E\colon G_{\mbb Q}\to \GL_2(\mbb Z_p)\to \GL_2(\overline{\mbb Q}_p)$, where $\mbb Z_p$ is the $p$-adic completion of $\mbb Z$. On the other hand, one can also get a Galois representation $\rho_f\colon G_{\mbb Q}\to \GL_2(\overline{\mbb Q}_p)$ starting from a modular eigenform $f$, following work of Eichler and Shimura. The Shimura, Taniyama, Weil conjecture is equivalent to the statement that $\rho_E$ is isomorphic to a representation $\rho_f$ arising from such a weight $2$ modular form.

Wiles puts himself in a situation in which he can assume  that the mod $p$ representation  $\rhobar$, arising from the reduction of $\rho_E$,  arises from a newform and $\rhobar|_{G_{\QQ(\zeta_p)}}$ is irreducible. Using Mazur's deformation theory of Galois representations Wiles reduces this conjecture to proving that certain surjective maps $\alpha_\Sigma\colon R_\Sigma\to {\mbb T}_\Sigma$ of (noetherian, commutative) local algebras over a discrete valuation ring $\mathcal{O}$ are isomorphisms.  The ring $R_\Sigma$ parametrizes  Galois deformations subject to constraints determined by a finite collection of primes $\Sigma$, while ${\mbb T}_\Sigma$ is a Hecke algebra subject to the same constraints, which can be thought of as parameterizing a family of modular forms. Hence proving that $\alpha_\Sigma$ is an isomorphism proves that every Galois representation (like $\rho_E$) parameterized by $R_\Sigma$ corresponds to a modular form.

The algebra ${\mbb T}_\Sigma$ acts faithfully  on a module $M_\Sigma$ constructed from the cohomology of a modular curve. To prove that  $\alpha_\Sigma$  is an isomorphism of complete intersections,     in the version of Wiles's approach given by Diamond~\cite{Diamond:1997}, one  simultaneously shows that  $M_\Sigma$ is free over $R_\Sigma$ and that $R_\Sigma$ is a complete intersection. 

The argument that $M_\Sigma$ is free over $R_\Sigma$ has two distinct parts. In the \emph{minimal case}, when $\Sigma=\varnothing$,  the technique of ``patching" introduced by Taylor and Wiles~\cite{Taylor/Wiles:1995} and adapted by Diamond~\cite{Diamond:1997}, yields a noetherian local ring $R_{\varnothing,\infty}$, a finitely generated $R_{\varnothing,\infty}$-module $M_{\varnothing,\infty}$, and surjective maps  
\[
R_{\varnothing,\infty} \twoheadrightarrow R_{\varnothing} \quad \text{and}\quad M_{\varnothing,\infty} \twoheadrightarrow M_{\varnothing}
\] 
where the map on the left is a homomorphism of rings, say with kernel $I$, and $M_{\varnothing}\cong M_{\varnothing,\infty}/I M_{\varnothing,\infty}$. The point of the patching process is that $R_{\varnothing,\infty}$ is a regular local ring and $M_{\varnothing,\infty}$ is a maximal Cohen-Macaulay module over it. Thus the Auslander-Buchsbaum formula implies  $M_{\varnothing,\infty}$ is free $R_{\varnothing,\infty}$-module, and hence $M_{\varnothing}$ is a free $R_{\varnothing}$-module, as desired. It turns out that $I$ is generated by a regular sequence, so as a byproduct of the patching argument, one gets that $R_{\varnothing}$ is a complete intersection.

The passage from the minimal  to the general, non-minimal, case   in which we allow additional ramification at primes in $\Sigma$, 
is facilitated by a  criterion of  Diamond (building on work of Wiles, and also Lenstra) that detects when a finitely generated module over a local ring is free and  the ring is a complete intersection; compare Theorem~\ref{th:defect0}. To describe this passage, we fix a  complete discrete valuation ring $\mco$ (for instance, $\mbb Z_p$), a complete local $\mco$-algebra $A$ (which could be $R_\Sigma$ or $\mbb T_\Sigma$) and an $\mco$-valued point in $\Spec A$, namely, a map of $\mco$-algebras $\lambda\colon A\to \mco$. A key hypothesis that holds in the context of \cite{Taylor/Wiles:1995,Wiles:1995} is that the map $\lambda_\fp$ is an isomorphism, for $\fp\colonequals \Ker \lambda$; compare Theorem~\ref{th:regular-eta}. In particular, $\fp$ is a minimal prime and hence defines a component of $\Spec A$. 

One shows that   if $\Sigma\subseteq \Sigma'$ then there exists a nonnegative integer $\nu$ such that 
\begin{equation}
\label{eq:goingup}
\begin{aligned}
\eta_{\lambda}(M_{\Sigma'})&\subseteq (\varpi^{\nu})\cdot \eta_{\lambda}(M_{\Sigma}) \\ 
\fitt_\mco \Phi_\lambda(R_{\Sigma'}) &\supseteq (\varpi^{\nu})\cdot \fitt_\mco \Phi_\lambda(R_{\Sigma}) \,. 
\end{aligned}
\end{equation}
This is where Proposition~\ref{pr:invariance-of-domain} and Lemma~\ref{le:change-of-module} come in, but only for $c=0$. Given these inequalities and the criterion of Diamond and Wiles, it is immediate that if $M_\Sigma$ is free as an $R_\Sigma$-module and $R_\Sigma$ is complete intersection, then the same properties hold for $M_{\Sigma'}$ and $R_{\Sigma'}$. Hence the freeness of $M_{\varnothing}$  as an $R_{\varnothing}$-module and the complete intersection property of $R_{\varnothing}$, proved by  patching, propagates to each $M_\Sigma$ and $R_\Sigma$, as desired.

The Langlands conjectures predict that far more general classes of two dimensional Galois representations, beyond those associated to elliptic curves, also correspond to modular forms. Wiles' original argument considers the case of Galois representations expected to correspond to weight $2$ modular forms. The case of Galois representations expected to correspond to modular forms of weight $k>2$ behaves fairly similarly to the weight $2$ case, when the weight is small compared to the residue characteristic $p$;  extensions of  techniques of Wiles can be used to prove $R_\Sigma = {\mbb T}_\Sigma$ results in many such cases.\footnote{In sufficiently high weights, the local deformation rings at $p$ are not smooth, and consequently the ring $R_\varnothing$ may fail to be a complete intersection, and so additional arguments are needed.}  On the other hand, the weight $1$ case behaves very differently to the weight $\ge 2$ case, and a number of additional complications arise when attempting to generalize Wiles' argument to this context.

From here on we make the blanket assumption that $p\ge 3$. 

\subsection*{Weight one.}
Let $D$ be a quaternion algebra over $\mbb Q$, ramified at a set of primes $\mathfrak D$ of even cardinality, with $p\not\in \mathfrak{D}$. One can consider a system of \emph{Shimura curves} $X^D(K)$ associated to compact open subgroups $K\subseteq (D\otimes \mbb{A}^\infty)^\times$, which behave similarly to the classical modular curves. One can define weight $1$ modular forms on the curve $X^D(K)$ analogously to the modular curve case.

Fix a Galois representation $\rhobar:G_{\QQ}\to \GL_2(\overline{\mbb{F}}_p)$ and assume that:
\begin{itemize}
\item $\rhobar|_{G_{\QQ(\zeta_p)}}$ is absolutely irreducible;
    \item $\rhobar|_{G_{\QQ_p}}$ is unramified;
    \item $\rhobar|_{G_{\QQ_q}}$ is non-scalar and has a Steinberg lift for each $q\in\mathfrak{D}$.
\end{itemize}
For any prime $q$ and any $e\ge 0$, let 
\[
K_1(q^e) = \left\{
\begin{pmatrix}a&b\\c&d\end{pmatrix}\in\GL_2(\mbb{Z}_q)\middle|\begin{pmatrix}a&b\\c&d\end{pmatrix}\equiv \begin{pmatrix}*&*\\0&1\end{pmatrix} \pmod{q^e}\right\}\subseteq \GL_2(\mbb{Z}_q).
\]
Take any finite set of prime $\Sigma$, disjoint from $\mathfrak{D}\cup\{p\}$, and consider the compact open subgroup $K_\Sigma = \prod_q K_{\Sigma,q}\subseteq (D\otimes \mbb{A}^\infty)^\times$ where:
\begin{itemize}
    \item $K_{\Sigma,q}$ is the maximal compact subgroup of $(D\otimes {\mbb Q}_q)^\times$ for each $q\in\mathfrak{D}$;
    \item $K_{\Sigma,q} = K_1(q^{c_q+d_q})$ for $q\in\Sigma$, where $q^{c_q}$ is the conductor of $\rhobar|_{G_{\QQ_q}}$ and $d_q = \dim(\rhobar^{I_{\QQ_q}})$;
    \item $K_{\Sigma,q} = K_1(q^{c_q})$ for $q\not\in\Sigma\cup\mathfrak{D}$.
\end{itemize}

Write $X^D_\Sigma$ for the Shimura curve $X^D(K_\Sigma)$. For each $\Sigma$ one can again define a ring $R^D_\Sigma$ parameterizing  lifts of minimal conductor outside $ \mathfrak{D} \cup \Sigma$,  unramified at $p$, and ``Steinberg''  at the primes in $\mathfrak{D}$. One may also define a complex of $\mco$-modules $M_\Sigma^D$ which computes the sheaf cohomology groups defining the space of weight $1$ modular forms on $X^D_\Sigma$, and a Hecke algebra $\mbb{T}^D_\Sigma$ with a derived action on $M_\Sigma^D$.

Again, there is a natural surjective map $\alpha^D_\Sigma\colon R_\Sigma^D\twoheadrightarrow \mbb{T}^D_\Sigma$, which is conjecturally an isomorphism. However there are a number of significant obstacles to generalizing the classical numerical criterion (and hence Wiles' overall strategy) to show that $\alpha_\Sigma^D$ is an isomorphism. In particular:
\begin{enumerate}
	\item The rings $R^D_\Sigma$ and ${\mbb T}^D_\Sigma$ are not in general  complete intersections, and so the numerical criterion cannot be directly used.
	\item Unlike in Wiles' case, the ring ${\mbb T}^D_\Sigma$ may not be flat over $\mco$ and in fact can be entirely torsion, which means that the augmentation $\lambda\colon {\mbb T}^D_\Sigma \to \mco$ used in Wiles' argument may not exist.
	\item The module $M_\Sigma$ from Wiles' argument has been replaced by a complex  $M^D_\Sigma$ of $\mco$-modules with a derived action of ${\mbb T}^D_\Sigma$. There is no clear notion of what it means for such an object to be free over ${\mbb T}^D_\Sigma$, and it's unclear how to generalize the freeness criterion due to Wiles, Lenstra, and Diamond.
\end{enumerate}

It was not clear how the numerical criterion could be adapted to serve in such situations, for the complete intersection property is essential to its application. Some of these problems have been overcome thanks to the work of Kisin, who developed the patching technique so it applies also in the non-minimal case, and  by Calegari and Geraghty, who extending patching to complexes. The patched rings, $R^D_{\Sigma,\infty}$, are often complete intersections and the complexes in question sometimes patch to maximal Cohen-Macaulay \emph{modules} $M^D_{\Sigma,\infty}$. However this is not sufficient to deduce freeness of $M^D_{\Sigma,\infty}$ because the $R^D_{\Sigma,\infty}$ is not usually regular outside the minimal case. All that one can deduce is that $M^D_{\Sigma,\infty}$ is generically free, only yielding $R^D_\Sigma\cong \mbb{T}^D_\Sigma$ up to torsion.

In \cite{Iyengar/Khare/Manning:2024a}, we sidestep the  issues above to prove more refined $R_\Sigma^D=\mbb{T}_\Sigma^D$ results that capture also torsion information. Our strategy is to apply the numerical criterion argument of Wiles and Diamond \emph{after} patching, rather than using patching and the numerical criterion independently. This eliminates the obstacles mentioned above to the use of the numerical criterion in the weight $1$ cases. Specifically:
\begin{enumerate}
	\item The rings $R^D_{\Sigma,\infty}$ can be computed explicitly and are complete intersections in the cases relevant to us, even when $R^D_\Sigma$ fails to be a complete intersection.
	\item By construction $R^D_{\Sigma,\infty}$ is flat over $\mco$, and there always exist suitable augmentations $\lambda\colon R^D_{\Sigma,\infty}\to \mco$, for which $(R^D_{\Sigma,\infty},\lambda)$ is in $\cato_\mco(c)$, even if these augmentations do not factor through the quotient $R^D_{\Sigma,\infty}\to R_\Sigma\to \mco$.
	\item While $M^D_\Sigma$ is a complex, $M^D_{\Sigma,\infty}$ is an $R^D_{\Sigma,\infty}$-module, so complexes need not be  considered in the argument.
\end{enumerate}

In the  forthcoming  paper with Diamond  \cite{Diamond/Iyengar/Khare/Manning:2026a} we generalize the results of \cite[Theorem 14.9]{Iyengar/Khare/Manning:2024a}, allowing for more varied types of ramification of the lift, to prove the following theorem; see \cite{Iyengar/Khare/Manning:2024a} for unexplained notation. The main work is on the automorphic side to generalize ``level raising'' arguments which enable one to allow general ramification behavior.

\begin{theorem}
\label{th:weightone}
With the notation and assumptions from above, there is some $\mu \ge 1$, such that for all finite sets of primes $\Sigma$ disjoint from $\mathfrak{D}\cup\{p\}$, there are $R_\Sigma^D$-modules $W_\Sigma^D$ and isomorphisms
\[
\hh_0(M_\Sigma^D) \cong (R_\Sigma^D)^\mu \oplus W_\Sigma^D.
\]
In particular, $R_\Sigma^D$ acts faithfully on $\hh_0(M_\Sigma^D)$ and hence the map  $\alpha_{\Sigma}^D\colon R_\Sigma^D\to {\mbb T}_\Sigma^D$ is an isomorphism. 
\end{theorem}

From this we can deduce a Jacquet-Langlands correspondence for weight $1$ Hecke algebras, which we state impressionistically here: the Hecke algebra $\TT_\Sigma^D$ acting on the cohomology of the Shimura curve above  is an (explicit) quotient of a Hecke algebra acting on the  weight one cohomology  $\hh^1(X_1(N),\omega)$ for some level $N$.  See \cite[Theorem 14.10]{Iyengar/Khare/Manning:2024b} for a precise statement in a more restricted setting. As these algebras typically have a lot of torsion, the statement  cannot be deduced from the  correspondence for weight  one forms in characteristic 0.

\begin{proof}[Sketch of Proof of Theorem~\ref{th:weightone}]
The strategy primarily follows Wiles'~\cite{Wiles:1995} and Diamond's~\cite{Diamond:1997} original approach, except that we apply the numerical criterion after patching. (For the purpose of the sketch, we make the simplifying assumption that  $\rhobar$ has minimal conductor among its twists, it has  no vexing primes,  and consider only deformations whose determinant is the Teichm\"uller lift of $\det\rhobar$, and modify the level structure accordingly.)  Specifically, we use Theorem \ref{th:defect0} in place of Diamond's numerical criterion. We fix an augmentation $\lambda\colon R_{\varnothing,\infty}^D\to \mco$ for which each $R_{\Sigma,\infty}^D$ is regular at $\Ker\lambda$. Since $R_{\varnothing,\infty}^D$ is a regular local ring, the Auslander-Buchsbaum formula implies $M_{\varnothing,\infty}^D$ is free.  Note that $M_{\varnothing}^D$, and hence  $M_{\varnothing,\infty}^D$, is non-zero by Serre's Conjecture and an analogue of the companion forms theorem for Shimura curves.  This gives the first equality below:
\[
\eta_{\lambda}(M_{\varnothing,\infty}^D) =\mco=\fitt_{\mco} \Phi_\lambda(R^D_{\varnothing,\infty})\,.
\]
The second equality follows as  $R^D_{\varnothing,\infty}$ is a power series ring over $\mco$. For the present purpose, we care only about the equality $\eta_{\lambda}(M_{\varnothing,\infty}^D) = \fitt_{\mco} \Phi_\lambda(R^D_{\varnothing,\infty})$. To complete the argument, we must then propagate it 
to non-minimal $\Sigma$. We use an  argument similar to that of  Wiles and Diamond, using Lemma \ref{le:change-of-module}. 

Indeed for $\Sigma\subseteq \Sigma'$, one has level-raising and lowering maps 
\[
\pi_\infty\colon M_{\Sigma',\infty}^D\longrightarrow M_{\Sigma,\infty}^D\quad\text{and} \quad
\pi_\infty^\vee\colon M_{\Sigma,\infty}^D\longrightarrow M_{\Sigma',\infty}^D\,,
\]
which are constructed by applying the patching process to certain families of natural maps $M_{\Sigma'}^D\to M_\Sigma^D$ and $M_{\Sigma}^D\to M_{\Sigma'}^D$. Direct computations of the compositions $\pi_\infty \circ\pi_\infty^\vee$ and $\pi_\infty^\vee \circ\pi_\infty$ show that $\pi_\infty$ is an isomorphism after localizing at $\Ker \lambda$.

A key number theoretic fact known as \emph{Ihara's Lemma}---proven in this generality in \cite[Proposition 5]{Diamond/Taylor:1994}---implies that $\pi_\infty$ is surjective. Lemma \ref{le:change-of-module} gives an equality 
\[
\eta_{\lambda}(M_{\Sigma',\infty}^D) = (\varpi^{\nu})\cdot \eta_{\lambda}(M_{\Sigma,\infty}^D)
\]
for some nonnegative integer $\nu$  which can be explicitly determined by computing the composition $\pi_\infty^\vee \circ\pi_\infty$. On the other hand, we can also show that 
\[
\fitt_{\mco}(\Phi_\lambda(R^D_{\Sigma',\infty})) =(\varpi^{\nu})\cdot \fitt_{\mco}(\Phi_\lambda(R^D_{\Sigma,\infty}))
\]
by direct computation, which is possible in our case, unlike in Wiles' argument, as the rings $R^D_{\Sigma,\infty}$ and $R^D_{\Sigma',\infty}$ can be described explicitly. This gives the equality  
\[
\eta_{\lambda}(M_{\Sigma',\infty}^D) =\fitt_{\mco}(\Phi_\lambda(R^D_{\Sigma,\infty}))
\]
for all $\Sigma$, and so Theorem \ref{th:defect0} gives the desired result.

The claim that the rank $\mu$ is independent of $\Sigma$ is a consequence of the fact that the maps $\pi_\infty$ are isomorphisms after localizing at $\Ker \lambda$.
\end{proof}

\subsection*{Factorization formula for congruence ideals of Hecke rings}
The main goal of this section is the factorization for congruence ideals of Hecke algebras in Theorem \ref{th:BK}.  The results of \cite{Bockle/Khare/Manning:2021b} and \cite{Bockle/Khare/Manning:2024} which  proved formulas for Wiles defect for the Hecke algebras we study are relatively simple  consequences  of this.

As before, let  $D$ be a quaternion algebra over $\mbb Q$, ramified at a set of primes $\mathfrak D$ of even cardinality. We consider Shimura curves $X^D(K)$ associated to compact open subgroups $K\subseteq (D\otimes \mbb{A}^\infty)^\times$; we will use  a different choice of level subgroups as compared to the previous section.

Take any finite set of prime $\Sigma$, disjoint from $\mathfrak{D}$, and consider the compact open subgroup $K_\Sigma = \prod_q K_{\Sigma,q}\subseteq (D\otimes \mbb{A}^\infty)^\times$ where:
\begin{itemize}
    \item $K_{\Sigma,q}$ is the maximal compact subgroup of $(D\otimes {\mbb Q}_q)^\times$ for each $q\in\mathfrak{D}$;
    \item $K_{\Sigma,q} = K_0(q)$ for $q\in\Sigma$. Here 
\[
K_0(q^e) = \left\{
\begin{pmatrix}a&b\\c&d\end{pmatrix}\in\GL_2(\mbb{Z}_q)\middle|\begin{pmatrix}a&b\\c&d\end{pmatrix}\equiv \begin{pmatrix}*&*\\0&*\end{pmatrix} \pmod{q^e}\right\}\subseteq \GL_2(\mbb{Z}_q).
\]
\item  $K_{\Sigma,q} = \GL_2(\ZZ_q)$  for $q\notin\Sigma\cup {\mathfrak D}$.
\end{itemize}

  This difference in the  choice of level structure at primes in $\Sigma$ than the  previous section reflects the key difference between this section and the previous one, namely that while in the previous section we studied weight one forms, here  we focus on   weight 2  newforms $f \in S_2(\Gamma_0(N),\mco)$ of square-free level $N={\mathfrak D}\Sigma$.   (By abuse of notation, we also  use $\mathfrak D$ and $\Sigma$  to denote the product of all the primes in the set $\mathfrak D$ and $\Sigma$. The context will make clear which meaning is intended.) 

We consider  a prime $p$  below not dividing $2\Sigma \mathfrak D$ and a finite extension $E/\QQ_p$, and let $\mco$ be the ring of integers in $E$, $\varpi$ a uniformizer and $k=\mco/\varpi$ the residue field. We assume below that $E$ is sufficiently large: in particular    $E$ contains the field of Fourier coefficients $K_f$ of  $f \in S_2(\Gamma_0(N))$.  Let $\rho_f\colon G_\Q \to \GL_2(\mco)$ be the Galois representation associated by Eichler and Shimura to $f$ and an embedding $\iota\colon K_f \hookrightarrow E$,  and assume   that the corresponding residual  Galois representation 
\[
\rhobar_f=\rhobar\colon G_\Q\to \GL_2(k)
\]
is  absolutely irreducible. Observe that $\det \rho_f =\epsilon_p$ the $p$-adic cyclotomic character. 

Note that   $N(\rhobar)>1$, and our conditions imply that the order of ${\rm im}(\rhobar)$ is divisible by $p$  and hence that  $\Ad^0(\rhobar)$ is absolutely irreducible and (otherwise $\rhobar$ has Serre weight 2 and Artin conductor 1 by assumptions on determinant and ramification above). This implies that  $\rhobar\colon G_{\Q(\mu_p)} \to \GL_2(k)$  is irreducible and thus satisfies the ‘Taylor–Wiles hypothesis’ needed  implicitly for the arguments below.

By enlarging $\mco$ if necessary,  we may assume that $k$ contains all eigenvalues of $\rhobar(\sigma)$ for all $\sigma\in G_\Q$. The Galois representation $\rho_f\colon G_\Q \to \GL_2(E)$, with irreducible residual representation $\rhobar$,  is locally at $q$ of the form\[
 \begin{bmatrix} 
 \epsilon_p\chi & * \\ 0 & \chi 
 \end{bmatrix}
 \]   
 for an unramified character $\chi$ of order dividing 2. 

 Consider $ \mco[T_n|(n,p N ) = 1] \subseteq \End_{\mco}(S_k(\Gamma_0(N),\mco)))$ the anemic Hecke algebra,  and  the full Hecke algebra in which we include  $U_q$ for all $q|N$. 
 These algebras act faithfully on $\hh^1(X_0(N),\mco)$. The newform $f$ induces augmentations of these Hecke algebras that we denote by $\lambda_f$. Consider the  maximal ideal $\fm$ of  the full Hecke algebra  that  contains $\Ker \lambda_f$. Denote  by $\TT$ and $\fullT$ respectively the images of the anemic and full Hecke algebra acting on $\hh^1(X_0(N),\mco)_\fm$. By the Jacquet-Langlands correspondence $f$ also induces augmentations to $\mco$  --  that we again denote by $\lambda_f$ -- of the full and anemic  Hecke algebras acting on $\hh^1(X^D(K_\Sigma),\mco)$ for every factorization $N=\mathfrak D\Sigma$ with $\mathfrak D$  divisible by an even number of primes.  We abuse notation and denote again  by   $\fm$  the maximal ideal of the full Hecke algebra acting 
 on $\hh^1(X^D(K_\Sigma), \mco)$ (which is the inverse image of the maximal ideal of $\mco$ under $\lambda_f$), and denote   by $\TT_\Sigma^D$ and $\fullT_\Sigma^D$ respectively the images of the anemic and full Hecke algebra acting on $\hh^1(X^D(K_\Sigma),\mco)_\fm$. We may also consider the case when $\mathfrak D$ is the empty set, in which case we are back to $\TT,\fullT$ acting on $\hh^1(X_0(N),\mco)_\fm$.

 Let $R_p$ be the framed local deformation ring parameterizing crystalline  lifts of $\rhobar_p=\rhobar|_{G_{\QQ_p}}$ with weight $2$.  This is a smooth ring over $\mco$ of relative dimension 4.  For primes $q$ dividing $N$ we consider $\Rt_q$  and $R_q$, the local deformation rings  which parametrize unipotent deformations  with, and without,  choice of  Frobenius eigenvalue,  of $\rhobar|_{G_{\QQ_q}}$; these are the rings  $R_q^{\phuni} $ and $R_q^{{\rm uni}}$ of \S  \ref{se:BKM-revisited} above  and \cite{Bockle/Khare/Manning:2024}.   The newform $f$ induces augmentations  that we again denote by $\lambda_f$ of $R_p,\Rt_q,R_q$: by the genericity of local components of the automorphic representation $\pi_f$ corresponding to $f$ of $\GL_2(\A_\Q)$, these local rings are all smooth at $\lambda_f$ (see \cite[Proposition 1.2.2, Theorem 1.2.7]{Allen:2016}), and thus these rings are in our category $C_\mco(3)$ for $q \neq p$, and $C_\mco(4)$ for $q=p$.

 For each factorization of $N$ as $N=\Sigma \mathfrak D$, let:
\begin{align*}
 R_{\Sigma,{\rm loc}}^D
  &=  \widehat{\bigotimes_{q \in \Sigma}} R_q^{\uni}  \widehat{\otimes}\widehat{\bigotimes_{q \in \mathfrak D}} R_q^{\rm St} \widehat{\otimes} R_p \\
 \Rt_{\Sigma, {\rm loc}}^D
 &=\widehat{\bigotimes_{q \in \Sigma}} R_q^{\phuni} \widehat{\otimes}\widehat{\bigotimes_{q \in \mathfrak D}} R_q^{\rm St}  \widehat{\otimes} R_p \,.
 \end{align*}
 Let $R_\Sigma^{D,\square}$  (respectively,  $R_\Sigma^{D}$)   be the  framed  (respectively, unframed) global Galois deformation ring of $\rhobar$ parameterizing lifts unramified outside $Np$, unipotent at $q\in \Sigma$, Steinberg at $q \in \mathfrak D$,  and crystalline  of weight $2$ at $p$.  There is a canonical map  $R_\Sigma^D \to R_\Sigma^{D,\square}$;  as it is smooth, we may also regard $R_\Sigma^D$ as a quotient of $R_\Sigma^{D, \square}$.

 We also have the  modified  $\Rt_\Sigma^D$ and  $\Rt_\Sigma^{D,\square}$.  We consider quotient maps
    $(R^D_{\Sigma})^{\square} \to R^D_{\Sigma}$ and $(\Rt^D_{\Sigma})^{\square} \to \Rt^D_{\Sigma}$, such that the compositions $R^D_{\Sigma} \to (R^D_{\Sigma})^{\square} \to R^D_{\Sigma}$  and  $\Rt^D_{\Sigma} \to (\Rt^D_{\Sigma})^{\square} \to \Rt^D_{\Sigma}$ are  the identity, and that arise from a consistent choice of frames at all places in $\Sigma \cup \mathfrak D$,  so that the following diagram commutes:
\[
\begin{tikzcd}[row sep=small]
       R^D_{\Sigma,{\rm loc}}  \arrow[d] \arrow[r] 
                & \Rt^D_{\Sigma,\rm{loc}} \arrow[d] \\
       (R^D_{\Sigma})^{\square}  \arrow[d] \arrow[r] 
                & (\Rt^D_{\Sigma})^{\square} \arrow[d] \\ 
       R^D_{\Sigma} \arrow[r] & \Rt^D_{\Sigma}         
\end{tikzcd}
\]
By \cite{Carayol:1994} there is a universal modular deformation $\rho_\Sigma^D\colon G_\Q \to \GL_2(\TT_\Sigma^D)$ that is a specialization of the universal representation $\rho_\Sigma^{D,  \rm univ}\colon G_\Q \to \GL_2(R_\Sigma^D)$ via a surjective map  $R_\Sigma^D \to \TT_\Sigma^D$; this map  extends  naturally (mapping choice of Frobenius eigenvalue at $q$, for $q \in \Sigma$, to $U_q$)  to $\Rt_\Sigma^D \to \fullT_\Sigma^D$. We define  
\[
M_\Sigma^D={\Hom}_{\TT_\Sigma^D[G_\Q]} (\rho_\Sigma^{D, \rm univ},\hh^1(X_\Sigma^D,\mco)_\fm^*)\,.
\]  
Thus $\hh^1(X_\Sigma^D,\mco)^*_\fm=M_\Sigma^D \oplus M_\Sigma^D$, and the action of $\TT_\Sigma^D$ on $M_\Sigma^D$ extends to an action of $\fullT_\Sigma^D$.  The theorem below refines the main result of \cite{Bockle/Khare/Manning:2024} for elliptic modular forms.

\begin{theorem}
\label{th:BK}
The following equalities hold:
\begin{align*}
\eta_{\lambda_f}(\TT_\Sigma^D)&=\eta_{\lambda_f}(\fullT_\Sigma^D)= \fitt_\mco \hh^1_{\mathrm{f}}(\QQ, \Ad \rho_f \otimes E/\mco )\prod_{q\in \Sigma} (\varpi^{m_q+n_q}) \prod_{q \in \mathfrak D} (\varpi^{n_q}). \\
\fitt_\mco \Phi_{\lambda_f}(\TT_\Sigma^D)
            &=\fitt_\mco \hh^1_{\mathrm{f}}(\QQ, \Ad \rho_f \otimes E/\mco)  \prod_{q\in \Sigma} (\varpi^{m_q+2n_q}) \prod_{q \in \mathfrak D} (\varpi^{3n_q});       \\
\fitt_\mco \Phi_{\lambda_f}(\fullT_\Sigma^D)
            &=\fitt_\mco \hh^1_{\mathrm{f}}(\QQ, \Ad \rho_f \otimes E/\mco)  \prod_{q\in \Sigma} (\varpi^{m_q+3n_q}) \prod_{q \in \mathfrak D} (\varpi^{3n_q}).    
\end{align*}
Thus the Wiles defects are
\[
\delta_{\lambda_f}(\TT_\Sigma^D)=\sum_{q \in \Sigma} n_q + \sum_{q \in \mathcal D} 2n_q \, 
\quad
\text{and}
\quad
\delta_{\lambda_f}(\fullT_\Sigma^D)=\sum_{q \in \Sigma} 2n_q + \sum_{q \in \mathcal D} 2n_q\,.
\]
\end{theorem}

 The proof has several inputs and is given later on in this section.

\subsubsection*{Patching}
The result below, proved using  patching, paraphrases \cite[Theorem  6.4]{Bockle/Khare/Manning:2024}. 

\begin{proposition}
    \label{pr:patching}
	There exist rings 
	\[
	R_{\Sigma,\infty}^D=R_{\Sigma, {\rm loc}}^D\pos{x_1,\ldots,x_g}\quad\text{and}\quad
    \Rt_{\Sigma,\infty}^D =\Rt_{\Sigma, {\rm loc}}^D\pos{x_1,\ldots,x_g}
	\]
	and $S_\infty = \mco\pos{y_1,\ldots,y_c}$ satisfying the following:
	\begin{enumerate}[\quad\rm(1)]
		\item $\dim S_\infty = \dim R_{\Sigma,\infty}^D= \dim  \Rt_{\Sigma,\infty}^D $.
		\item  
        Consider the map $\iota\colon R_{\Sigma,\infty}^D \to \Rt_{\Sigma,\infty}^D$ that extends the map   $R_{\Sigma,{\rm loc}}^D \to \Rt_{\Sigma,{\rm loc}}^D$ and sends $x_i$ to $x_i$ for $i=1,\cdots, g$. Then there   exists a  continuous $\mco$-algebra morphism $j\colon S_\infty\to R_{\Sigma,\infty}^D $ such that $R_{\Sigma,\infty}^D$ and   $\Rt_{\Sigma,\infty}^D$,  with their induced structure as  modules over  $S_\infty$, are  finite and  free.
		\item 
        There are  isomorphisms of $R_{\Sigma,{\rm loc}}^D$-algebras, and  $\Rt_{\Sigma,{\rm loc}}^D$-algebras,
        \[
         \mco \otimes_{S_\infty} R_{\Sigma, \infty}^{D} \cong R_{\Sigma}^D \quad\text{and}\quad
          \mco  \otimes_{S_\infty}\Rt_{\Sigma, \infty}^{D} \cong \Rt_{\Sigma}^D
         \]
        respectively,  such that the natural diagram
        \[
\begin{tikzcd}[column sep=small]
S_{\infty} \arrow[r,"j"]
                &R^D_{\Sigma,\infty} \arrow[d,twoheadrightarrow] \arrow[rr,"\iota"] 
                   && \Rt^D_{\Sigma,\infty} \arrow[d,twoheadrightarrow] \\
                &R^D_{\Sigma} \arrow[rr] 
                   && \Rt^D_{\Sigma}         
\end{tikzcd}
\] 
commutes. 
		\item The rings $R_\Sigma^D$ and $\Rt_\Sigma^D$  are finite free over $\mco$, and both act faithfully on $M_\Sigma^D$. 
	\end{enumerate}
\end{proposition}

\begin{proof}
We refer to  \cite[Theorem  6.4]{Bockle/Khare/Manning:2024} and its proof. While  (2)  and (3)  are not explicitly proved there as stated--as in \emph{op.cit.}
the map  $R_{\Sigma,\infty}^D \to \Rt_{\Sigma,\infty}^D $ is not considered--it is easy to modify the proof there to obtain (2) and (3). 

The  proof of (4) relies on the fact that  the local rings $R_q,\Rt_q,R_p$ are Cohen-Macaulay. This  type of argument goes back to  the paragraph before Corollary 4.7 of \cite{Khare/Wintenberger:2009b}; see also  \cite[Proposition 5.1.1, Section 5.2]{Snowden:2018}.
\end{proof}

The following extract of Proposition \ref{pr:patching} is needed below.

\begin{corollary}
\label{cor:patching-output}  
The sequence $\bos{y}\colonequals y_1,\ldots, y_c$ is regular on $\Rt_{\Sigma, \infty}^D$ and   $R_{\Sigma, \infty}^D$ and
\[
\Rt_{\Sigma, \infty}^D/(\bos{y}) \cong \Rt_\Sigma^D\quad\text{and} \quad 
R_{\Sigma, \infty}^D/(\bos{y})\cong R_\Sigma^D\,.
\]
Furthermore, $R_\Sigma^D \cong \TT_\Sigma^D $ and $\Rt_\Sigma^D \cong \fullT_\Sigma^D$, and   $(\Rt_\Sigma^D,\lambda_f)$ and $(R_\Sigma^D,\lambda_f)$ are in $\cato_\mco(0)$, and are Cohen-Macaulay. 
\end{corollary}

\begin{proof}
We know from the theorem that $R_\Sigma^D, \Rt_\Sigma^D$ both act faithfully on $M_\Sigma^D$;  as their actions factor through $\TT_\Sigma^D$ and $\fullT_\Sigma^D$, we get  $R_\Sigma^D \cong \TT_\Sigma^D $ and $\Rt_\Sigma^D \cong \fullT_\Sigma^D$. The last statement about $(R_\Sigma^D,\lambda_f)$ and  $(\Rt_\Sigma^D,\lambda_f)$  being in the category $C_\mco(0)$ follows from these isomorphisms and  the  fact that, as $f$ is a newform of level $N$, the localization  of $\fullT_\Sigma^D$ at $\fp_f=\ker \lambda_f$ is $E$.   
\end{proof}

This together with Theorem \ref{th:deformation} is key to the  factorizations of congruence ideals $\eta_\lambda(\TT)$ of Hecke algebras.

\subsubsection*{Cotangent spaces of deformation rings}
In what follows, to lighten the notation we drop $\Sigma,D$ from the notation: for instance  we denote the rings $R_{\Sigma}^D$ and $\Rt_\Sigma^D$ by $R,\Rt$. In fact as the argument is the same in all cases, we further assume that $\mathfrak D$ is the empty set.  The ideals $\fn_\infty,\fp_\infty$, and $\fpt_\infty$ denote the kernels of the augmentation of $S_\infty,R_\infty$, and $\Rt_\infty$ induced by $\lambda_f$; by $\fp_q,\fpt_q$  kernels of the augmentation  induced by $\lambda_f$ for  $R_q$ and $\Rt_q$; and  by $\fp,\fpt$  kernels of the augmentation  induced by $\lambda_f$ for  $R$ and $\Rt$.  Note that the isomorphism  $R_{\infty}=R_{{\rm loc}}\pos{x_1,\ldots,x_g}$  induces  $\fp_\infty/\fp_\infty^2 \cong \prod_{q|Np} \fp_q/\fp_q^2 \times \mco^g $ with projection maps ${\rm pr}_q\colon \fp_\infty/\fp_\infty^2 \to \fp_q/\fp_q^2$.

One can prove the  formula for $\fitt_\mco \Phi_\lambda(R)$ using  Poitou-Tate sequence; see \cite[Lemma 2.1]{Diamond/Flach/Guo:2004}.  We sketch a different argument using patching, in particular Corollary \ref{cor:patching-output}.  This allows us to give a unified argument proving the formulas that pertain to congruence ideals and Fitting ideals of cotangent spaces.  The Jacobi-Zariski exact sequence associated to $R_\infty\to R\to \mco$ reads
\begin{equation}
\label{eq:cotangent}
0\longrightarrow \fn_\infty/\fn_\infty^2 \xrightarrow{\ \iota_R} \fp_\infty/\fp_\infty^2  \longrightarrow \fp/\fp^2 \longrightarrow 0\,.
\end{equation}
The exactness on the left  follows from $R_{\fp}=E$. From this we get:
\begin{equation}
\label{eq:cotangent1}
0\longrightarrow (\fp_\infty/\fp_\infty^2)^*\xrightarrow{\ \iota_R^*\ } (\fn_\infty/\fn_\infty^2)^* \longrightarrow
 \Phi_{\lambda_f}(R)^\vee \xrightarrow{ \alpha} \Phi_{\lambda_f}(R_\infty)^\vee \longrightarrow 0\,
\end{equation}
Note that the map  $\alpha$  factors as 
\begin{equation}
\label{eq:cotangent2}
\Phi_{\lambda_f}(R)^\vee \hookrightarrow (\fp_\infty/\fp_\infty^2)^\vee \longrightarrow  (\fp_\infty/\fp_\infty^2)^\vee/(\fp_\infty/\fp_\infty^2)_{\rm div} ^\vee
= \Phi_{\lambda_f}(R_\infty)^\vee,
\end{equation}
 with $(\fp_\infty/\fp_\infty^2)_{\rm div}^\vee$ the divisible  submodule  of $(\fp/\fp^2)^\vee$.

   We let   $\mathfrak L=(\mathfrak L_q)$  be  the local \emph{Selmer conditions}:

\begin{itemize}

    \item $\mathfrak L_q=\hh^1_{\rm unr} (G_q,\Ad_f \otimes E/\mco)$ for $(q,Np)=1$,  and 
   
   \item $\mathfrak L_q$ is the image in $\hh^1(G_q,\Ad_f \otimes E/\mco)$ of $(\fp_q/\fp_q^2)^\vee$ that sits inside the split exact sequence
\[
   0 \to (\fp_q/\fp_q^2)_{\rm div}^\vee \to  (\fp_q/\fp_q^2)^\vee \to \Phi_\lambda(R_q)^\vee \to 0  ;
\] 
  
\end{itemize}
Define the global (deformation theoretic) Selmer group  as   
 \[
 \hh^1_{\mathfrak L}(\Q,\Ad_f \otimes E/\mco)\colonequals 
    \Ker\big(\hh^1(\Q,\Ad_f \otimes E/\mco) \longrightarrow  
    \prod_{q} \frac{\hh^1(G_q,\Ad_f \otimes E/\mco)}{\mathfrak L_q}\big)\,.
 \]
With $\hh^1_{\mathrm{f}}(G_q,\Ad_f \otimes E)$ the  Bloch-Kato local Selmer group with $E$-coefficients, set
  \[
  \hh^1_{\mathrm{f}}(G_q,\Ad_f \otimes E/\mco) \colonequals \image\big(
  \hh^1_{\mathrm{f}}(G_q,\Ad_f \otimes E)\longrightarrow \hh^1(G_q,\Ad_f \otimes E/\mco)\big)\,.
  \]
The Bloch-Kato  global  Selmer group is
\[
 \hh^1_{\mathrm{f}}(\Q,\Ad_f \otimes E/\mco)\colonequals 
    \Ker\big(\hh^1(\Q,\Ad_f \otimes E/\mco) \longrightarrow  
    \prod_{q} \frac{\hh^1(G_q,\Ad_f \otimes E/\mco)}{\hh^1_{\mathrm f}(G_q,\Ad_f \otimes E/\mco)}\big).
 \]
 
The  following omnibus Lemma \ref{le:Kisin-lemma} is useful for us. 

\begin{lemma}
  \label{le:Kisin-lemma} 
The following statements hold.
\begin{enumerate}[\quad\rm(1)]
\item   
The image of $\Hom_\mco(\fp_q/\fp_q^2,E)$ in  $\hh^1(G_q,\Ad_f \otimes E)$ is  $\hh^1_{\mathrm{f}}(G_q,\Ad_f \otimes E)$. 
\item
The image of $(\fp_q/\fp_q^2)_{\rm div}^\vee$ in $\hh^1(G_q,\Ad_f \otimes E/\mco)$
is  $\hh^1_{\mathrm{f}}(G_q,\Ad_f \otimes E/\mco)$.
\item  We have an injection 
\[
\Phi_{\lambda_f}(R_\infty)^\vee \hookrightarrow  \prod_{q|Np} \frac{\hh^1(G_q,\Ad_f \otimes E/\mco)}{\hh^1_{\mathrm{f}}(G_q,\Ad_f \otimes E/\mco)}.
\]
\item  
$\Phi_{\lambda_f}(R)^\vee\cong\hh^1_{\mathfrak L}(\Q,\Ad_f \otimes E/\mco).$
\end{enumerate}
  \end{lemma}

\begin{proof}
We note that  $\Hom_\mco(\fp_q/\fp_q^2,E)$ is   the tangent space  at the maximal ideal of   the localization of  $R_q$ at $\fp_q$; see \eqref{eq:rank-cotangent}.  Parts (1) and (2) then  
follow from a result of Kisin \cite[Proposition 2.3.5]{Kisin:2009a} (see also \cite[Lemma 1.2.5]{Allen:2016}),  which identifies the completion of  localisation of $R_q$ at $\fp_q$ with deformation ring that parametrizes  lifts  of  the representation 
\[
\rho_{\fp_q}\colonequals \rho_f|_{G_q} \otimes E\colon G_q \to   \GL_2(E)
\]
of fixed Hodge-Tate weights  $(0,1)$(in the case of $q=p$), and of inertial type that is  semistable  for $q \neq p$,  and that is  unramified (for $q=p$).

As to (3), part (1)  implies  that the image of 
 \[ 
 (\fp_\infty/\fp_\infty^2)_{\rm div}^\vee \to \prod_{q|Np} \hh^1(G_q,\Ad_f \otimes E/\mco) 
 \] 
 is 
$\prod_{q|Np} \hh_f^1(G_q,\Ad_f \otimes E/\mco)$.

Observe that  the space of   coboundaries $B^1(G_q,\Ad_f \otimes E/\mco)$ is the image of the divisible module $\Ad_f \otimes E/\mco$: 
\[
0 \to (\Ad_f \otimes E/\mco)^{G_q} \to \Ad_f \otimes E/\mco \to  B^1(G_q,\Ad_f \otimes E/\mco) \to 0, 
\]   
and hence is divisible. Thus  the kernel of the natural map   
\[ 
(\fp_\infty/\fp_\infty^2)^\vee \to \prod_{q|Np} \hh^1(G_q,\Ad_f \otimes E/\mco) 
\] 
is  divisible, and contains  the   coboundaries $\prod B^1(G_q,\Ad_f \otimes E/\mco)$. Furthermore as   $\prod_{q|Np} \hh_{\mathrm{f}}^1(G_q,\Ad_f \otimes E/\mco)$ is divisible we deduce that the kernel of   
\[ 
(\fp_\infty/\fp_\infty^2)^\vee \to \prod_{q|Np} \frac{\hh^1(G_q,\Ad_f \otimes E/\mco)}{\hh^1_{\mathrm{f}}(G_q,\Ad_f \otimes E/\mco)} 
\] 
is also divisible. As we   can write  (non-canonically) 
 \[ 
 (\fp_\infty/\fp_\infty^2)^\vee= (\fp_\infty/\fp_\infty^2)_{\rm div}^\vee \oplus\Phi_{\lambda_f}(R_\infty)^\vee, 
 \] 
 we conclude that one has an injection
  \[
  \Phi_{\lambda_f}(R_\infty)^\vee \hookrightarrow  \prod_{q|Np} \frac{\hh^1(G_q,\Ad_f \otimes E/\mco)}{\hh^1_{\mathrm{f}}(G_q,\Ad_f \otimes E/\mco)}.
  \]

Part (4) is a standard argument so we skip it; see  \cite[Lemma 2.40]{Darmon/Diamond/Taylor:1997} and also  \cite[Proposition 3.24]{Gee:2022}.
\end{proof}

Another crucial input for  the proof of Theorem \ref{th:BK}  is the following proposition whose proof relies on Lemma \ref{le:Kisin-lemma}.

 \begin{proposition}
 \label{prop:BKSelmer}
 The cokernel of the map $\iota_R^*$ in \eqref{eq:cotangent1}  is $\hh^1_{\mathrm{f}}(\Q,\Ad_f \otimes E/\mco)$. 
 \end{proposition}

\begin{proof}
Keeping in mind that
\[
\Phi_{\lambda_f}(R)^\vee \cong \hh^1_{\mathfrak L}(\Q,\Ad_f \otimes E/\mco) \quad\text{and}\quad
\Phi_{\lambda_f}(R_\infty)^\vee \cong 
\prod_{q|Np} (\fp_q/\fp_q^2)^\vee/(\fp_q/\fp_q^2)_{\rm div}^\vee
\]
from  \eqref{eq:cotangent2} we get:
\[ 
\Phi_{\lambda_f}(R)^\vee \longrightarrow 
\Phi_{\lambda_f}(R_\infty)^\vee \hookrightarrow  \prod_{q|Np} \frac{\hh^1(G_q,\Ad_f \otimes E/\mco)}{\hh^1_{\mathrm{f}}(G_q,\Ad_f \otimes E/\mco)}.
\]
The injectivity of last map is by Lemma \ref{le:Kisin-lemma}(3).  We claim that the resulting map  \[ \hh^1_{\mathfrak L}(\Q,\Ad_f \otimes E/\mco) \to \frac{\hh^1(G_q,\Ad_f \otimes E/\mco)}{\hh^1_{\mathrm{f}}(G_q,\Ad_f \otimes E/\mco)} ,\]  arises  from  the  global-to-local restriction map of Galois cohomology. To justify the claim, recall that   for $q|Np$   the map $R_\infty \to R^\square$ induces the  natural map $R_q \to R\square$ between  a local and global deformation ring. 
 
Thus we get the exact sequence: 
\[   
0 \longrightarrow \hh^1_{\mathrm{f}}(\mbb{Q},\Ad_f \otimes E/\mco) \longrightarrow \Phi_{\lambda_f}(R)^\vee  \xrightarrow{\alpha} \prod_{q|Np} \Phi_{\lambda_f}(R_q)^\vee \longrightarrow 0,
\]
with $\alpha$ arising from  global-to-local restriction map of Galois cohomology.   Comparing the exact sequence above to \eqref{eq:cotangent1} gives the desired isomorphism. 
\end{proof}

\subsection*{Proof of Theorem~\ref{th:BK}}
 Using Corollary \ref{cor:patching-output} (which gives    $\TT$ is a quotient of $R_\infty = (\widehat{\bigotimes}_{q|Np} R_q)\pos{x_1,\cdots, x_g}$ by a regular sequence) and  applying   Theorem~\ref{th:deformation}  and  Proposition \ref{prop:BKSelmer} we get  that:
\[
\eta_{\lambda_f}(\TT)=\fitt_\mco \hh^1_{\mathrm{f}}(\QQ,\Ad_f \otimes E/\mco) \prod_{q|Np} \eta_\lambda(R_q^{\uni})\,.
\]
We also get 
\[
\fitt_\mco(\Phi_{\lambda,f}(R))=\fitt_\mco \hh^1_{\mathrm{f}}(\QQ,\Ad_f \otimes E/\mco) \prod_{q|Np} \fitt_\mco(\Phi_{\lambda_f}(R_q^{\uni}))\,.
\]
Similarly for maps $\Rt_\infty\to \Rt\to \mco$ we  have the Jacobi-Zariski exact sequence:
\[
0\longrightarrow \fn_\infty/\fn_\infty^2 \xrightarrow{\ \iota_{\Rt}} \fpt_\infty/\fpt_\infty^2  \longrightarrow \fpt/\fpt^2 \longrightarrow 0     
\]
The exactness on the left follows from $\Rt_{\fp_{\lambda_f}}=E$.  This yields an exact sequence
\begin{gather*}
0\longrightarrow (\fpt_\infty/\fpt_\infty^2)^*\xrightarrow{\ \iota_{\Rt}^*\ } (\fn_\infty/\fn_\infty^2)^* \longrightarrow
 \Phi_{\lambda_f}(\Rt)^\vee \longrightarrow \Phi_{\lambda_f}(\Rt_\infty)^\vee \longrightarrow 0\, 
\end{gather*}
Since $\fullT$ is a quotient of $\Rt_\infty = (\widehat{\bigotimes}_{q|Np} \Rt_q)\pos{x_1,\cdots, x_g}$ by a regular sequence, by applying  Theorem~\ref{th:deformation} to $\Rt_\infty$ we get:
 \[
 \eta_{\lambda_f}(\fullT)
            =\fitt_\mco (\coker(\iota_{\Rt}^*))\prod_{q|Np} \eta_{\lambda_f}(R_q^{\phuni}).
            \] 
 One can compare the cokernels of $\iota_{\Rt}^*$ and $\iota_R^*$ using
\[ 
0 \to  \coker((\fpt_\infty/\fpt_\infty^2)^* \to (\fp_\infty/\fp_\infty^2)^*) \to   \coker(\iota_{\Rt}^*)   \to \coker(\iota_{R}^*) \to 0.     
\]
From Proposition \ref{pr:patching}(3)  we deduce:
\[ 
0 \to  \prod_{q|Np} \coker\big((\fpt_q/\fpt_q^2)^* \to (\fp_q/\fp_q^2)^*\big) \to   \coker(\iota_{\Rt}^*)   \to \coker(\iota_{R}^*) \to 0\,.     
\]  
We know that  $ \prod_{q|Np} \coker\big((\fpt_q/\fpt_q^2)^*   \to (\fp_q/\fp_q^2)^*\big)$ is surjective from Lemma \ref{le:transfer-factor}, and thus   
\[
\coker(\iota_{\Rt}^*)   \cong \coker(\iota_{R}^*).
\]
From Proposition~\ref{prop:BKSelmer} we get  $\coker(\iota_R^*) \cong \hh^1_{\mathrm{f}}(\mbb{Q},\Ad_f \otimes E/\mco)$. Thus 
 $\coker(\iota_{\Rt}^*) \cong \hh^1_{\mathrm{f}}(\mbb{Q},\Ad_f \otimes E/\mco)$. From this we get that 
 \begin{align*}
 \eta_{\lambda_f}(\fullT)
&=\fitt_\mco \hh^1_{\mathrm{f}}(\QQ,\Ad_f \otimes E/\mco) \prod_{q|Np} \eta_{\lambda_f}(R_q^{\phuni}) \\
\fitt_\mco(\Phi_{\lambda_f}(\Rt))
&=\fitt_\mco\hh^1_{\mathrm{f}}(\QQ,\Ad_f \otimes E/\mco) \prod_{q|Np} \fitt_\mco(\Phi_{\lambda_f}(R_q^{\phuni})).
\end{align*}
 Using the  formula from Theorem \ref{thm:uni-steinberg}:
\[
\lambda_f(R_q^{\rm uni})=\lambda_f(R_q^{\phuni})=(\varpi^{m_q+n_q}), 
\]
we get the formulas of Theorem \ref{th:BK} for congruence ideals. We also know  from Theorem \ref{thm:uni-steinberg}: 
\begin{align*}
\fitt_\mco(\Phi_{\lambda_f}(R_q^{\uni})) &=(\varpi^{m_q+2n_q})\\
\fitt_\mco(\Phi_{\lambda_f}(R_q^{\phuni})) &=(\varpi^{m_q+3n_q}).
\end{align*}
This gives the formulas in Theorem \ref{th:BK} for Fitting ideals of cotangent spaces, and thus we have proved all the claims made in statement of the theorem.

\subsection*{Local Tamagawa and congruence ideals  for $\Ad_f$}

\begin{corollary}
\label{cor:tamagawa}
The congruence ideal $\eta_{\lambda_f}(R_q^{\rm uni})$ is the Tamagawa ideal in $\mco$  of the  rank 3 adjoint motive $\Ad_f$ at $q$.
    \end{corollary}

\begin{proof}
 Consider the surjective map $R_q^\square \to R_q^{\rm uni}$ in the category $\cato_{\mco}(3)$.  By \cite[Proposition 7.9 (2)]{Bockle/Khare/Manning:2021b} the Fitting ideal of the  relative cotangent space 
 \[
 \fitt_\mco(\Phi_{\lambda_f}(R_q^\square/R_q^{\rm uni}))=(q^2-1)(\varpi^{-n_q}).
 \] It is also  known that  that $R_q^\square$ is a complete intersection.  Thus using the numerical criterion for complete intersection  from Theorem \ref{th:defect0}, and  knowing the defect of $R_q^{\rm uni} $  at $\lambda$ is $n_q$,  we get:   \[\eta_\lambda(R_q^{\rm uni})=\fitt_\mco {\rm cotors}(\hh^1(G_q,\Ad_f \otimes E/\mco))(q^2-1)^{-1}.\] We  also note that 
 \[
 (q^2-1)=({\det(1-\Frob_qq^{-s}|_{\Ad_f^{I_q}(1)})}|_{s=0})=L_q(0,\Ad_f(1))^{-1}
 \]
 as $\Frob_q$ acts by $q^2$ on $\Ad_f^{I_q}(1)$.   Here $L_q(0,\Ad_f(1))$ is the Euler factor at $p$, evaluated  at $s=0$,  of $L(s,\Ad_f(1))$. Thus 
 \[\eta_\lambda(R_q^{\rm uni})= \fitt_\mco {\rm cotors}(\hh^1(G_q,\Ad_f \otimes E/\mco))L_q(0,\Ad_f(1)).\] By the arguments in  \cite[I.4.2.2]{Fontaine/Perrin-Riou:1994} and \cite[Section 2.4]{Diamond/Flach/Guo:2004},  we  know  that the  Tamagawa ideal  in $\mco$ of $\Ad_f$  at a prime $q \neq p$ is given by 
\[
\fitt_\mco {\rm cotors}(\hh^1(G_q,\Ad_f \otimes E/\mco))L_q(0,\Ad_f(1)),
\] 
which proves the result.
    \end{proof}

\begin{remark}
We could prove the result (as explained to us by Diamond)  by using that $\eta_\lambda(R_q^{\rm uni})=(\varpi^{m_q+n_q})$, and proving separately that the  Tamagawa ideal in $\mco$  of the adjoint motive $\Ad_f$ at $q$ is the same. We have  given a slightly different proof which shows the equality in the statement without computing fully  each side separately. This strategy    might help  to link Tamagawa ideals to congruence ideals  when considering local deformation rings at $q$ (for $q \neq p$)  beyond the 2-dimensional case.  Note that as $\pi: R_q^\square \to R_q^{\uni}$ is  a surjective map of Gorenstein rings, we know that $\ann(\ker \pi)=(\pi^\vee (1))$ with $\pi^\vee$ the dual map. It will be interesting to evaluate directly, without computation of change of cotangent spaces that 
\[
\lambda(\pi^\vee(1)) =({\det(1-\Frob_qq^{-s}|_{\Ad_f^{I_q}(1)})}|_{s=0})=(q^2-1)\,.
\]
\end{remark}

We also note in passing:

\begin{lemma}
    
   One has $\hh^1_{\mathrm{f}}(G_q,\Ad_f \otimes E)=0$ for all primes $q|N$. Furthermore $\dim_E \hh^1_{\mathrm{f}}(G_p,\Ad_f \otimes E)= \dim_E \hh^0_{\mathrm{f}}(G_p,\Ad_f \otimes E)+1$.
   \end{lemma}
   
   \begin{proof}
    A  computation using  Euler characteristics and Tate duality  shows that, as  $\rho_f|_{G_q}$ is generic,  in fact all of  $\hh^1(G_q, \Ad_f \otimes E)=0$.  The last statement follows from \cite[Theorem 1.2.7]{Allen:2016}.
   \end{proof}

\begin{remark}

   For  all newforms $f$ of weight $k \geq 2$ and level $N$, not necessarily squarefree, it is true that 
   \[
   \hh^1(G_q,\Ad_f \otimes E)= \hh^1_{\mathrm{f}}(G_q,\Ad_f \otimes E).
   \] 
for all primes $q \neq p$. This follows from the genericity of $f$ at $q$. Also for all primes $q \neq p$ 
 \[
   \hh^1(G_q,\Ad_f \otimes E/\mco)_{\rm div}= \hh^1_{\mathrm{f}}(G_q,\Ad_f \otimes E/\mco).
   \] 

\end{remark}

\subsection*{From congruence ideals of rings to congruence ideals of modules}
Now we want to discuss how one can deduce $\eta_\lambda(M_\Sigma^D)=\eta_\lambda(\TT_\Sigma^D)$, or equivalently   $ \eta_\lambda(\TT_\Sigma^D)=\eta_{\lambda}(\hh^1(X_\Sigma^D,\mco)_\fm)$,  using the ``numerical'' arguments in \cite[Section 7]{Bockle/Khare/Manning:2024}.   This equality of congruence ideals is  (somewhat surprisingly)  true in spite of  the modules  $\hh^1(X_\Sigma^D,\mco)_\fm$ not always being  free over $\TT_\Sigma^D$; this happens for instance when $\mathfrak D$ contains a  prime $q$ such that $\rhobar$ is unramified and $\rhobar(\Frob_q)$ is scalar,

Let  $\langle \ , \ \rangle $ be the twist  by the Atkin-Lehner involution $w_N$ of the Poincar\'e  pairing $\langle \ , \ \rangle' $ on $\hh^1(X_0(N),\mco)$: thus $\langle x , y \rangle= \langle x, w_Ny \rangle'  $. The action of the Hecke operators $T_\ell$ for $(\ell,N)=1$  is self-adjoint for this pairing, and we   use only these Hecke operators in the proof below. In the statement of the corollary below, when we refer to
$\eta_{\lambda}(\hh^1(X_\Sigma^D,\mco)_\fm)$, the ambiguity of whether we are considering $\hh^1(X_\Sigma^D,\mco)_\fm$ as a $\TT$ or $\fullT$ module does not matter as the congruence ideal $\eta_{\lambda}(\hh^1(X_\Sigma^D,\mco)_\fm)$ is the same in either case. This is the easy case of Proposition \ref{pr:finite-module} for $c=0$.

The following result is deduced from formulas for the congruence ideals $\eta_\lambda(\TT_\Sigma^D)$ in Theorem \ref{th:BK}.

\begin{corollary}\label{cor:cong-mod}
    We have the equality  of ideals  \[ \eta_{\lambda_f}(\TT_\Sigma^D)=\eta_{\lambda_f}(\fullT_\Sigma^D)= \eta_{\lambda_f}(M_\Sigma^D)= \eta_{\lambda_f}(\hh^1(X_\Sigma^D,\mco)_\fm)= (\langle x, y \rangle)\]\[ =\fitt_\mco \hh^1_{\mathrm{f}}(\Q,\Ad_f \otimes E/\mco) \prod_{q\in \Sigma} (\varpi^{m_q+n_q}) \prod_{q \in \mathfrak D} (\varpi^{n_q}) ,
\] 
 with $\hh^1(X_0(N_\Sigma),\mco)[\Ker\lambda_f]=\mco x \oplus \mco y$.
 \end{corollary}

\begin{remark}
The equality \[\eta_{\lambda_f}(\hh^1(X_0(N),\mco)_\fm)=\fitt_\mco \hh^1_{\mathrm{f}}(\Q,\Ad_f \otimes E/\mco) \prod_{q|N} (\varpi^{m_q+n_q}),
\] namely the case  $\mathfrak D=\varnothing$ in the corollary, is a consequence of \cite{Diamond/Flach/Guo:2004}.
\end{remark}

\begin{proof}
  
The main idea for the proof  is  to use, in conjunction with  computation of $\eta_\lambda(\TT^D_\Sigma)$ done in Theorem \ref{th:BK},  a larger Hecke algebra $\TT(N^2)$  and a module $M(N^2)$ over it  of generic rank one, which has the property that $\TT(N^2)$  acts freely on $M(N^2)$ and is a complete intersection.  This serves to ``resolve'' the fact that  $\TT^D_\Sigma$ need not be a complete intersection, and  need not act freely on $M_\Sigma^D$. One of the key properties of congruence modules we use for these arguments is the inclusion  $\eta_\lambda(R) \subset \eta_\lambda(M)$ for $(R,\lambda) \in \cato_{\mco}(c)$, and a finitely generated $R$-module $M$ that is supported at $\lambda$. We will sketch the argument in enough (but not all) detail for the reader to see the strategy, and refer to \cite[Section 7]{Bockle/Khare/Manning:2024} for details in a  more general setting.

To define  $\TT(N^2)$  and its  module $M(N^2)$   we  proceed  by considering the oldform $f^N$ in $S_2(\Gamma_0(N^2))$  arising from  $f$ which is characterized by the property that it is an eigenform for the  Hecke operators  $T_\ell$ for $(\ell,N)=1$ with the same eignevalues as $f$,  and  $f^N|U_\ell =0$  for $\ell|N$. Let ${\fm}_N$
be maximal ideal of the Hecke action on $\hh^1(X_0(N^2),\mco)$ corresponding to $f^N$, and consider 
 \[
 M(N^2)={\Hom}_{R^N[G_\Q]} (\rho^N,\hh^1(X_0(N^2),\mco)_{\fm_N}^*).
 \]  
 Here $R^N$ is the ring which parametrizes deformations  of $\rhobar$ that are unramified outside $Np$, with cyclotomic determinant, finite flat at $p$, and no conditions at primes $q$ dividing $N$, and $\rho^N$ is the corresponding universal deformation.   We have  
 \[
 \hh^1(X_0(N^2),\mco)^*_{\fm_N}=M(N^2)  \oplus M(N^2)\,.
 \] 
 Furthermore the corresponding (anemic) Hecke algebra $\TT(N^2)$ acts on $M(N^2)$, and we denote the augmentation of $\TT(N^2)$ induced by $f^N$ by $\lambda_f$.

As recalled in \cite[Theorem 5.2]{Bockle/Khare/Manning:2021b}, by work of Wiles and Diamond,  we know that $R^N=\TT(N^2)$, that $M(N^2)$ is a free $\TT(N^2)$-module (of rank one), and  that $\TT(N^2)$ is a complete intersection. Thus
      \begin{equation}
      \label{eq:eq1.5} 
      \fitt_\mco(\Phi_{\lambda_f}(R^N))=\eta_{\lambda_f}(M(N^2)).
    \end{equation}

Using patching we show as in \cite[Section 7]{Bockle/Khare/Manning:2021b}, or as in the proof of Theorem \ref{th:BK} above (using   Jacobi-Zariski
sequences like  (\ref{eq:cotangent}), and their exactness on the left which is a consequence of patching) 
that 

\begin{equation}
\label{eq:eq0}
\fitt_\mco(\Phi_{\lambda_f} (R^N))=\fitt_\mco(\Phi_{\lambda_f}(R_\Sigma^D)) \prod_{q \in \Sigma} (q^2-1)(\varpi^{-n_q})\prod_{q \in \mathfrak D}(q^2-1)(\varpi^{m_q-2n_q}).
\end{equation}

Using geometric arguments (see \cite[Section 7]{Bockle/Khare/Manning:2024}) we prove the following inclusion of congruence ideals:
\begin{equation}
\label{eq:eq1}
\eta_{\lambda_f}(M_\Sigma^D) \prod_{q \in \Sigma} (q^2-1) \prod_{q \in \mathfrak D}(q^2-1)(\varpi^{m_q}) \subset \eta_{\lambda_f}(M(N^2))
\end{equation}    

 Since  $\eta_\lambda(R^D_\Sigma) \subset \eta_\lambda(M_\Sigma^D)$, one also has
 \begin{equation}
 \label{eq:eq2}
 \begin{aligned}
    \fitt_\mco(\Phi_{\lambda_f}(R^D_\Sigma))\eta_{\lambda_f}(M_\Sigma^D)^{-1} 
            &\subseteq  \fitt_\mco(\Phi_{\lambda_f}(R^D_\Sigma))\eta_{\lambda_f}(R^D_\Sigma)^{-1}\\
            &=\prod_{q \in \Sigma}(\varpi^{n_q})\prod_{q \in \mathfrak D}(\varpi^{2n_q}).
            \end{aligned}
    \end{equation} 
 
    From  (\ref{eq:eq0}), (\ref{eq:eq1}),  (\ref{eq:eq1.5}) and (\ref{eq:eq2}) we conclude that the containment in both (\ref{eq:eq1}) and (\ref{eq:eq2}) are equalities and that indeed  
    $\eta_{\lambda_f}(\TT_\Sigma^D)=\eta_{\lambda_f}(M_\Sigma^D)$.  As we already know that $ \eta_{\lambda_f}(\TT_\Sigma^D)=\eta_{\lambda_f}(\fullT_\Sigma^D)$ we have proved the result.
    \end{proof}

\begin{remark}
 As explained in \cite[Section 7]{Bockle/Khare/Manning:2024} the  ``upper bounds'' on change of congruence ideals in (\ref{eq:eq1}) are deduced from monodromy arguments in \cite[Theorem 2]{Ribet/Takahashi:1997} which are less delicate  than \cite[Theorem 1]{Ribet/Takahashi:1997}  (that needs a further input of  certain maps on component groups being surjective) that proves these bounds are in fact an equality. The proof above gives a method to deduce 
\cite[Theorem 1]{Ribet/Takahashi:1997} from the easier \cite[Theorem 2]{Ribet/Takahashi:1997}, and thus weaken hypotheses in proving  the former.
\end{remark}

\begin{remark}
In the case when $\mathfrak D = \varnothing$, and so $X^{\varnothing}_{\Sigma}=X_0(N)$ is the usual modular curve, we know that $\hh^1(X_0(N),\mco)_\fm^*$ is a free $\fullT$-module of rank two by the $q$-expansion principle, and so,  as $\hh^1(X_0(N),\mco)_\fm^*\cong M_{\Sigma}\oplus M_{\Sigma}$, $M_\Sigma$ is a free $\fullT$-module of rank one. As we have 
\[
M_{\Sigma,\infty}/(y_1,\cdots, y_c)=M_\Sigma
\]
and the $y_i$'s are an $M_{\Sigma,\infty}$ a regular sequence, we  deduce $M_{\Sigma,\infty}$  is also a free $\Rt_{\Sigma,\infty}$-module of rank one.

For general $\mathfrak D$, let ${\mathcal Q}\subseteq {\mathfrak D}$ be the subset of primes $q$ for which $\rhobar|_{G_{\QQ_q}}$ is trivial (or on other words, the primes $q$ for which $n_q > 0$).

In line with the categorical Langlands correspondence, we conjecture that 
     \[
     M_{\Sigma,\infty}^D=\widehat \bigotimes_{q \in \Sigma} R_q^{\phuni} \widehat{\otimes} \widehat \bigotimes_{q \in {\mathfrak D}\setminus {\mathcal Q}}  R_q^{\rm St}  \widehat{\otimes}\widehat \bigotimes_{q \in {\mathcal Q}}  M_q^{\rm St} \widehat{\otimes} \mco\pos{x_1,\ldots,x_g}
     \]
     where for each $q\in \mathcal Q$, $M_q^{\rm St}$ is a self-dual maximal Cohen--Macaulay $R_q^{\st}$-module of generic rank $1$ which depends only on the prime $q$. Note that $R_q^{\rm St}$ is formally smooth for $q \in {\mathfrak D}\setminus {\mathcal Q}$.
     
     In the case when $\mathcal Q = \varnothing$, the ring $\Rt_{\Sigma,\infty}$ is Gorenstein, and it is straightforward to prove this conjecture via the numerical criterion.
     
      In the case where $\Sigma=\varnothing$ but $\mathcal Q \ne \varnothing$, the work of the third author \cite{Manning:2021} establishes a ``mod $p$'' version of this conjecture\footnote{It is likely that a refinement of the methods of that paper could also be used to prove the integral version of the conjecture in the $\Sigma=\varnothing$ case.}, namely that
     \[
     M_{\Sigma,\infty}^D/\varpi M_{\Sigma,\infty}^D= \widehat \bigotimes_{q \in {\mathfrak D}\setminus {\mathcal Q}}  \overline{R}_q^{\rm St}  \widehat \bigotimes_{q \in {\mathcal Q}}  \overline{M}_q^{\rm St} \widehat{\otimes} \mco\pos{x_1,\ldots,x_g}
     \]
where $\overline{R}_q^{\rm St} := R_q^{\rm St}/\varpi R_q^{\rm St}$ and $\overline{M}_q^{\rm St}$ is a self-dual maximal Cohen--Macaulay module over $\overline{R}_q^{\rm St}$ of generic rank $1$, depending only on $q$. Moveover, $\overline{M}_q^{\rm St}$ is the unique self-dual $\overline{R}_q^{\rm St}$-module of generic rank $1$.

Note that this conjecture implies that the trace map 
\[
\Hom_{R_{\Sigma,\infty}^D}(M_{\Sigma,\infty}^D,\omega_{R_{\Sigma,\infty}^D}) \otimes_{R_{\Sigma,\infty}^D} M_{\Sigma,\infty}^D \to \omega_{R_{\Sigma,\infty}^D}
\]
from Proposition \ref{pr:trace-map} is onto. Indeed, applying this conjecture in the case when $\Sigma = \varnothing$ and $|\mathcal Q| = 1$, and combining it with the results of \cite{Manning:2021}, implies that $M_q^{\rm St}/\varpi M_q^{\rm St} = \overline{M}_q^{\rm St}$ (by the uniqueness of $\overline{M}_q^{\rm St}$) for all $q$. \cite[Theorem 3.3]{Manning:2021} implies that the trace map of $\overline{M}_q^{\rm St}$ is onto, and so the same holds for $M_q^{\rm St}$ by Nakayama's Lemma. It follows that the trace map of $M_{\Sigma,\infty}^D$ is also surjective.

In particular, Proposition \ref{pr:trace-map}(2) implies that for any $\lambda$ one has
 \[
\eta_\lambda(M_{\Sigma,\infty}^D)=\eta_\lambda(\Rt_{\Sigma,\infty}^D)=\eta_\lambda(R_{\Sigma,\infty}^D)\,.
\]
 One can  in fact prove this implication of the  conjecture unconditionally. Namely,  we know from Corollary \ref{cor:cong-mod} that $\eta_{\lambda_f}(M_\Sigma^D)=\eta_{\lambda_f}(\Rt_\Sigma^D)=\eta_{\lambda_f}(R_\Sigma^D)$. This  implies $\eta_\lambda(M_{\Sigma,\infty}^D)=\eta_\lambda(\Rt_{\Sigma,\infty}^D)=\eta_\lambda(R_{\Sigma,\infty}^D)$ by Theorem \ref{th:deformation} and the fact  that $M_\Sigma^D=M_{\Sigma,\infty}^D \otimes_{S_\infty} \mco$,  $\Rt_{\Sigma}^D=\Rt_{\Sigma,\infty}^D \otimes_{S_\infty} \mco$ and $R_{\Sigma}^D=R_{\Sigma,\infty}^D \otimes_{S_\infty} \mco$.
\end{remark}

\begin{remark} 
Our  two number theory applications (Theorem \ref{th:weightone} and Theorem \ref{th:BK}) differ in spirit. The first relies on our numerical criterion applied after patching to prove integral $R=\TT$ theorems. The second application uses an integral  $R=\TT$ theorem as a starting point. 
 
Using the expository nature of this paper as  poetic license  to conjure a more visual image of the local-global aspect of patching,  we can think of patching as  a cut and paste argument: writing $R=R_\infty/(y_1,\ldots,y_c)$ glues the local deformation rings that make up  $R_\infty$  to  form the global deformation ring $R$. The global term $\fitt_\mco \hh^1_{\mathrm{f}}(\Q,\Ad_f \otimes E/\mco)$  in the  factorization of $\eta_\lambda(R)$ comes from the gluing data (the equations that cut out $\Spec(R) \hookrightarrow \Spec(R_\infty)$), while the local terms $\eta_\lambda(R_q)$  come   from the ambient $\Spec R_\infty$.
\end{remark}

\section*{Acknowledgments.}
The work of the first author is partly supported by National Science Foundation grants DMS-200985 and DMS-2502004. The  third author received funding from the European Research Council (ERC) under the European Union's Horizon 2020 research and innovation programme (grant agreement No. 884596).  We would like to thank Gebhard B\"ockle for helpful correspondence. Our thanks to Matt Emerton for helpful conversations, and to Fred Diamond for allowing us include results from on-going collaborations. A special thanks to Fred Diamond for reading various versions of this manuscript and providing constructive feedback. The first author also thanks Kashif Khan and Aryaman Maithani  for numerous conversations on the material in Section~\ref{se:determinantal-rings}.

\bibliographystyle{amsplain}
\newcommand{\noopsort}[1]{}
\providecommand{\bysame}{\leavevmode\hbox to3em{\hrulefill}\thinspace}
\providecommand{\MR}{\relax\ifhmode\unskip\space\fi MR }
\providecommand{\MRhref}[2]{%
  \href{http://www.ams.org/mathscinet-getitem?mr=#1}{#2}
}
\providecommand{\href}[2]{#2}

\end{document}